\pdfpageattr{/Group <</S /Transparency /I true /CS /DeviceRGB>>}
\documentclass[letterpaper, twoside, 10pt, leqno]{amsart}

\usepackage{amssymb}
\usepackage{mathtools}
\usepackage{mathrsfs}
\usepackage{tikz}
    \usetikzlibrary{cd}
\usepackage{enumitem}
    \setlist[enumerate]{label={\textnormal{(\alph*) }}, nosep, align=left, left={\parindent}, labelsep=0pt, widest=b, rightmargin=0pt}
\usepackage[alphabetic, nobysame]{amsrefs}
\usepackage[colorlinks=true, linkcolor=blue, citecolor=blue]{hyperref}
\usepackage[hyphenbreaks]{breakurl}

\allowdisplaybreaks
\numberwithin{equation}{section}

\theoremstyle{definition}
    \newtheorem{pdef}{Definition}[section]
\theoremstyle{plain}
    \newtheorem{pthrm}{Theorem}[section]
    \newtheorem{pmain}{Theorem}
    \newtheorem{pmaincor}[pmain]{Corollary}
    \newtheorem{plem}[pthrm]{Lemma}
    \newtheorem{pcor}[pthrm]{Corollary}
    \newtheorem{pprop}[pthrm]{Proposition}
\theoremstyle{remark}
    \newtheorem{prmrk}{Remark}[section]
    \newtheorem{pwarn}{Warning}[section]

\DeclareMathOperator{\csch}{csch}

\newcommand{\iso}{\cong}
\newcommand{\id}{\mathrm{id}}
\newcommand{\injto}{\hookrightarrow}
\newcommand{\surjto}{\twoheadrightarrow}

\newcommand{\isoto}[1][\sim]{\stackrel{#1}{\to}}
\newcommand{\longto}{\longrightarrow}

\newcommand{\longbijto}{\longleftrightarrow}

\newcommand{\Hom}{\mathrm{Hom}}
\newcommand{\Isom}{\mathrm{Isom}}
\newcommand{\del}{\nabla}

\DeclareMathOperator{\tr}{tr}

\DeclareMathOperator{\adj}{adj}
\DeclareMathOperator{\Ad}{Ad}

\newcommand{\I}{\mathrm{I}}
\newcommand{\II}{\mathrm{I\!\!\:I}}
\newcommand{\vol}{\mathrm{vol}}
\newcommand{\Vol}{\mathrm{Vol}}
\newcommand{\cs}{\mathrm{cs}}
\newcommand{\CS}{\mathrm{CS}}
\DeclareMathOperator{\re}{Re}
\DeclareMathOperator{\im}{Im}

\begin{document}

\title[The renormalization of volume and Chern-Simons invariant]{The renormalization of volume \\ and Chern-Simons invariant \\ for hyperbolic 3-manifolds}
\author[Dongha Lee]{Dongha Lee}
\address{Korea Advanced Institute of Science and Technology, 291, Daehak-ro, Yuseong-gu, Daejeon, Republic of Korea}
\email{leejydh97@kaist.ac.kr}
%\urladdr{\url{https://sites.google.com/view/dongha-lee/home}}
\date{\today}
\begin{abstract}
We renormalize the Chern-Simons invariant for convex-cocompact hyperbolic 3-manifolds by finding the asymptotics along an equidistance foliation. We prove that the metric Chern-Simons invariant has an exponentially divergent term given by the integral of the torsion 2-form with respect to a Weitzenb{\"o}ck connection. This produces the asymptotics of hyperbolic volume plus the metric Chern-Simons invariant, which is often called complex volume. The leading coefficient of the asymptotics introduces a complex-valued quantity consisting of mean curvature and torsion 2-form, which is defined on smooth surfaces embedded in a Riemann-Cartan 3-manifold.
\end{abstract}
\maketitle

\section{Introduction}

For noncompact Einstein manifolds having infinite volume, the notion of \emph{renormalized volume} has been introduced by Henningson and Skenderis~\cite{MR1644988} and further developed by Graham~\cite{MR1758076} as an invariant that replaces classical volume. The underlying principle is intuitive: One computes the asymptotics of the volume of the \( r \)-neighborhood of a compact subset as \( r \to \infty \), and then discards the divergent terms to obtain a finite value. The renormalized volume has been extensively studied for both mathematical and physical purposes.

In particular, it has played a significant role in hyperbolic geometry. The case of noncompact hyperbolic \( 3 \)-manifolds having infinite volume was studied by Karsnov and Schlenker~\cite{MR2386723}, following the ideas of Epstein's work in~\cite{EpsteinWS} and~\cite{MR1813434}. They introduced insightful notions such as \emph{\( \mathrm{W} \)-volume} and \emph{metric at infinity}, which have become central tools in the field. A number of subsequent works have explored the role of renormalized volume in hyperbolic geometry. Notable contributions include comparisons with the volume of convex core~\cites{MR3188032, MR3732685}, estimates for certain classes of hyperbolic \( 3 \)-manifolds~\cite{MR4887772}, and analyses of the behavior on the deformation space of hyperbolic structures~\cites{MR3934591, MR4028105, MR4668096, MR4571574}. It also served as a key ingredient in~\cite{MR3784525} to establish a lower bound of the topological entropy of pseudo-Anosov maps on surfaces.

Meanwhile, Thurston~\cite{MR0648524} and Neumann and Zagier~\cite{MR0815482} observed that the \emph{Chern-Simons invariant}, introduced in~\cite{MR0353327}, plays the role of an ``imaginary counterpart'' to the volume of hyperbolic \( 3 \)-manifolds. At the level of differential forms, it turned out that the \( \mathrm{PSL}_2 {\left( \mathbb{C} \right)} \)-Chern-Simons \( 3 \)-form \( \cs {\left( \omega_{\hat{\varrho}} \right)} \) consists of the volume form \( \mathrm{d} \vol \) and the \( \mathrm{SO}(3) \)-Chern-Simons \( 3 \)-form \( \cs {\left( \bar{\omega} \right)} \) up to an exact form as follows (Proposition~\ref{P:CS-form}).
\begin{align}
4 \boldsymbol{i} \boldsymbol{\pi}^2 \sigma^* \cs {\left( \omega_{\hat{\varrho}} \right)} & = \mathrm{d} \vol + \mathrm{d} \gamma + \boldsymbol{i} \boldsymbol{\pi}^2 s^* \cs {\left( \bar{\omega} \right)}. \label{E:CS-form-simpler}
\end{align}
For closed hyperbolic \( 3 \)-manifolds, this implies that the \( \mathrm{PSL}_2 {\left( \mathbb{C} \right)} \)-Chern-Simons invariant is decomposed into hyperbolic volume and the \( \mathrm{SO}(3) \)-Chern-Simons invariant as follows (Corollary~\ref{C:CS}), which is often called \emph{complex volume}.
\begin{align}
4 \boldsymbol{i} \boldsymbol{\pi}^2 \CS_{\mathrm{PSL}_2 {\left( \mathbb{C} \right)}} {\left( M \right)} & = \Vol(M) + \boldsymbol{i} \boldsymbol{\pi}^2 \CS_{\mathrm{SO}(3)} {\left( M \right)}. \label{E:CS}
\end{align}
For noncompact hyperbolic \( 3 \)-manifolds having finite volume, Yoshida~\cite{MR807069} obtained a well-defined complex-valued quantity corresponding to the complex volume by adding certain length and torsion terms associated with cusps, and showed that it defines a holomorphic function near the original hyperbolic structure in the deformation space of (complete or incomplete) hyperbolic structures. This result also supports the profound relationship between hyperbolic volume and the Chern-Simons invariant.

It is natural to consider the case of noncompact hyperbolic \( 3 \)-manifolds having infinite volume. This requires us to define the \emph{renormalized Chern-Simons invariant} through the aforementioned renormalization procedure, which is the main concern of this paper. This type of problem was also studied in~\cite{MR3159164} and~\cite{MR3228423} in different contexts. The main result and contribution of this paper is to present the explicit asymptotics of the Chern-Simons invariant and to elucidate the geometric meaning of its divergent terms within hyperbolic geometry. Such a step is essential to effectively handle renormalized invariants in practice. This reveals a connection to \emph{Weitzenb{\"o}ck geometry}, which is the central concept of Einstein's teleparallelism in general relativity introduced in~\cite{EinsteinTP}. It turns out that the divergent terms in the asymptotics are completely expressed in terms of the data from the Weitzenb{\"o}ck geometry of hyperbolic ends and the conformal boundary, which makes intensive use of \( 3 \)-dimensional hyperbolic geometry.

The main object of this paper is noncompact hyperbolic \( 3 \)-manifolds having infinite volume, specifically \emph{convex-cocompact} hyperbolic \( 3 \)-manifolds (Definition~\ref{D:convex-cocompact}). In order to renormalize the Chern-Simons invariant, we mainly follow Krasnov and Schlenker's approach of renormalized volume in~\cite{MR2386723}. Unlike the case of volume, however, we inevitably encounter a crucial difficulty: The Chern-Simons invariant essentially depends on the choice of a global smooth section, but there is no canonical choice. Instead of looking for a canonical choice, we suggest a natural convergence condition on the global smooth section to renormalize the Chern-Simons invariant. As a result, we successfully define the renormalized Chern-Simons invariant, which depends on the choice of a global smooth section.

The outline of this paper is as follows. In Section~\ref{S:asymptotics}, we find the asymptotics of the \( \mathrm{SO}(3) \)-Chern-Simons invariant. As noted earlier, the Chern-Simons invariant requires the specific choice of a global smooth section, or a global smooth frame in the case of \( \mathrm{SO}(3) \). We observe in Subsection~\ref{SS:setup} that there is a natural convergence condition on global smooth frame, which introduces the notion of \emph{constant frame} (Definition~\ref{D:constant-frame}). Under this condition, we compute the asymptotics in Subsections~\ref{SS:I}--\ref{SS:infinity}. A pivotal idea in the computation is to construct a \emph{foliation at infinity} associated with each hyperbolic end, which extends the notion of metric at infinity introduced in~\cite{MR2386723} to suit our purpose. We finish Section~\ref{S:asymptotics} by stating the desired asymptotics~\eqref{E:asymptotics-result} without providing a geometric meaning of the exponentially divergent term.

In Section~\ref{S:connections}, for the geometric meaning of the divergent term, we discuss hypersurfaces in Weitzenb{\"o}ck geometry. A \emph{Weitzenb{\"o}ck connection} (Definitions~\ref{D:Weitzenbock-affine} and~\ref{D:Weitzenbock-principal}) is a nonholonomic flat connection induced by a fixed global smooth frame. We find at the end of Section~\ref{S:connections} that the leading coefficient of the asymptotics corresponds to the integral of the \emph{torsion\/ \( 2 \)-form} (Definition~\ref{D:torsion-form}) with respect to the Weitzenb{\"o}ck connection induced by the \emph{smooth frame at infinity}. It is a geometric quantity that is completely determined by the data from infinity. The first result of the present paper is the following.

\begin{pmain} \label{M:A}
Let \( M \) be a convex-cocompact hyperbolic\/ \( 3 \)-manifold with a compact submanifold \( C \) in Definition~\ref{D:convex-cocompact}. For each \( r \in [0, \infty) \), let \( C_r \) be the compact submanifold consisting of the points in \( M \) at distance\/ \( \le r \) from \( C \). Let \( s \) be a global oriented orthonormal smooth frame for \( M \) that is constant along the equidistant foliation\/ \( {\left\{ \partial C_r \right\}}_{r \in [0, \infty)} \) (Definition~\ref{D:constant-frame}). Then the following limit converges.
\begin{align}
\CS_{\mathrm{SO}(3)}^{\textnormal{R}} {\left( M, C, s \right)} & = \lim_{r \to \infty} {\left( \int_{C_r} s^* \cs {\left( \bar{\omega} \right)} + \frac{\boldsymbol{e}^r}{4 \sqrt{2} \boldsymbol{\pi}^2} \int_{\partial C^{\infty}} \tau {\left( s^{\infty} \right)} \right)}. \label{E:A}
\end{align}
Here,\/ \( \cs {\left( \bar{\omega} \right)} \) is the\/ \( \mathrm{SO}(3) \)-Chern-Simons\/ \( 3 \)-form of \( M \), \( \partial C^{\infty} \) denotes the surface \( \partial C \) with the metric at infinity, and \( \tau {\left( s^{\infty} \right)} \) is the torsion\/ \( 2 \)-form of \( \partial C^{\infty} \) with respect to the Weitzenb{\"o}ck connection induced by the smooth frame \( s^{\infty} \) at infinity.
\end{pmain}

\noindent This result defines the \emph{renormalized\/ \( \mathrm{SO}(3) \)-Chern-Simons invariant} of a convex-cocompact hyperbolic \( 3 \)-manifold. Although we do not have a canonical choice of the global smooth frame \( s \) in Theorem~\ref{M:A}, there is a canonical choice of the compact submanifold \( C \): the \emph{Epstein foliation} (Theorem~5.8 of~\cite{MR2386723}). Also, it is worth noting that the invariant depends on the behavior of the global smooth frame only in the hyperbolic ends, not in the whole interior, up to an even integral difference. It satisfies \emph{holographic principle} in this sense, as with the renormalized volume.

In Section~\ref{S:monodromy}, we compute the asymptotics of the \( \mathrm{PSL}_2 {\left( \mathbb{C} \right)} \)-Chern-Simons invariant, using the asymptotics of \( \mathrm{W} \)-volume given in~\cite{MR2386723} and the asymptotics of the \( \mathrm{SO}(3) \)-Chern-Simons invariant given in Theorem~\ref{M:A}\@. A key observation in the computation is that the exact form \( \mathrm{d} \gamma \) appearing in the equality~\eqref{E:CS-form-simpler} between the Chern-Simons \( 3 \)-forms is expressed as the difference between usual mean curvature and Weitzenb{\"o}ck mean curvature (Proposition~\ref{P:exact-form}).

\begin{pmain} \label{M:B}
In the setup of Theorem~\ref{M:A}, let \( \sigma \) be the global smooth section in~\eqref{E:sigma} induced by \( s \). Then the following limit converges.
\begin{align}
\begin{aligned}
& 4 \boldsymbol{i} \boldsymbol{\pi}^2 \CS_{\mathrm{PSL}_2 {\left( \mathbb{C} \right)}}^{\textnormal{R}} {\left( M, C, \sigma \right)} \\
& \qquad = \lim_{r \to \infty} {\left( \begin{aligned}
4 \boldsymbol{i} \boldsymbol{\pi}^2 \int_{C_r} \sigma^* \cs {\left( \omega_{\hat{\varrho}} \right)} & + \frac{\boldsymbol{e}^r}{4 \sqrt{2}} \int_{\partial C^{\infty}} {\left( H {\left( s^{\infty} \right)} \, \mathrm{d} a^{\infty} + \boldsymbol{i} \tau {\left( s^{\infty} \right)} \right)} \\
& + \boldsymbol{\pi} r \chi {\left( \partial_{\infty} M \right)}
\end{aligned} \right)}.
\end{aligned}
\end{align}
Here,\/ \( \cs {\left( \omega_{\hat{\varrho}} \right)} \) is the\/ \( \mathrm{PSL}_2 {\left( \mathbb{C} \right)} \)-Chern-Simons\/ \( 3 \)-form of \( M \), \( H {\left( s^{\infty} \right)} \) is the mean curvature of \( \partial C^{\infty} \) with respect to the Weitzenb{\"o}ck connection induced by the smooth frame \( s^{\infty} \) at infinity, and \( \chi {\left( \partial_{\infty} M \right)} \) is the Euler characteristic of the conformal boundary \( \partial_{\infty} M \) of \( M \).
\end{pmain}

\noindent This result defines the \emph{renormalized\/ \( \mathrm{PSL}_2 {\left( \mathbb{C} \right)} \)-Chern-Simons invariant}, or the \emph{renormalized complex volume}, of a convex-cocompact hyperbolic \( 3 \)-manifold. Again, there is no canonical choice of the global smooth frame \( s \) in Theorem~\ref{M:B}, but one can take the Epstein foliation for the compact submanifold \( C \).

The leading coefficient appearing in the asymptotics in Theorem~\ref{M:B} is the integral of an interesting complex-valued quantity
\begin{align}
\boldsymbol{H} & = H + \boldsymbol{i} {\star} \tau
\end{align}
consisting of mean curvature \( H \) and the Hodge dual of torsion \( 2 \)-form \( \tau \) with respect to a Weitzenb{\"o}ck connection. Indeed, this complex-valued quantity is defined on any smooth surface embedded in a (not necessarily hyperbolic) Riemannian \( 3 \)-manifold with a metric-compatible connection, which is commonly referred to as a \emph{Riemann-Cartan\/ \( 3 \)-manifold}. A more focused discussion of this geometric quantity from the general perspective of Riemann-Cartan geometry and Weitzenb{\"o}ck geometry is provided in~\cite{LeeCM}. It exhibits properties similar to those of the usual mean curvature in Riemannian geometry, generalizing a number of well-known results in classical minimal surface theory.

We also obtain the equality of the renormalized version as a direct corollary, which parallels one of the results in~\cite{MR3159164}.

\begin{pmaincor} \label{M:C}
In the setup of Theorem~\ref{M:A}, let \( \sigma \) be the global smooth section in~\eqref{E:sigma} induced by \( s \). Then the following equality holds.
\begin{align}
4 \boldsymbol{i} \boldsymbol{\pi}^2 \CS_{\mathrm{PSL}_2 {\left( \mathbb{C} \right)}}^{\textnormal{R}} {\left( M, C, \sigma \right)} & = \mathrm{W}^{\textnormal{R}} {\left( M, C \right)} + \boldsymbol{i} \boldsymbol{\pi}^2 \CS_{\mathrm{SO}(3)}^{\textnormal{R}} {\left( M, C, s \right)}.
\end{align}
Here,\/ \( \mathrm{W}^{\textnormal{R}} {\left( M, C \right)} \) is the renormalized\/ \( \mathrm{W} \)-volume of \( M \) relative to \( C \).
\end{pmaincor}

\noindent In Corollary~\ref{M:C}, the \emph{renormalized\/ \( \mathrm{W} \)-volume} is a geometric quantity introduced by Krasnov and Schlenker~\cite{MR2386723}, which satisfies
\begin{align}
\mathrm{W}^{\textnormal{R}} {\left( M, C \right)} & = \Vol^{\textnormal{R}} {\left( M, C \right)} - \frac{\boldsymbol{\pi}}{2} \chi {\left( \partial_{\infty} M \right)},
\end{align}
where \( \Vol^{\textnormal{R}} {\left( M, C \right)} \) is the renormalized volume of \( M \) relative to \( C \) (Lemma~4.5 of~\cite{MR2386723}). The difference from the renormalized volume is a topological constant proportional to the Euler characteristic of the conformal boundary. Compared to~\eqref{E:CS}, the renormalized \( \mathrm{W} \)-volume plays the role of the real part in the case of noncompact hyperbolic \( 3 \)-manifolds having infinite volume. This also agrees with an invariant defined in~\cite{MR3228423} in a different context.

Corollary~\ref{M:C} implies that the renormalized \( \mathrm{PSL}_2 {\left( \mathbb{C} \right)} \)-Chern-Simons invariant depends on the behavior of the global smooth section only at infinity, up to \( 2 \boldsymbol{i} \boldsymbol{\pi}^2 \mathbb{Z} \). In this sense, it satisfies the holographic principle as well. Meanwhile, as an imaginary counterpart to the renormalized \( \mathrm{W} \)-volume, one might try avoiding the difficulty of choosing a global smooth section by employing the variational formula of renormalized volume. However, this attempt fails, as the renormalized volume is known to become a K{\"a}hler potential for the Weil-Petersson metric. A more detailed explanation is provided in Remark~\ref{R:variation}.

\subsubsection*{Notations and conventions}

All orientations in this paper obey the right-hand rule. We say a smooth hypersurface \( S \) embedded in an oriented Riemannian manifold \( M \) is oriented by a unit normal vector field \( N \) along \( S \) iff the orientation of \( S \) is given as follows: A frame \( {\left( E_1, \dotsc, E_{n-1} \right)} \) for \( S \) is oriented iff the frame \( {\left( N, E_1, \dotsc, E_{n-1} \right)} \) for \( M \) is oriented. We orient the boundary of a smooth manifold by an outward-pointing vector field (i.e., the Stokes orientation).

We will often make use of the Einstein summation convention to achieve brevity: An index variable appearing twice in a single term (once in an upper position and once in a lower position) assumes the summation of the term over the index. For example, \( A^i B_i = \sum_i A^i B_i \). The ``cyclic'' summation is denoted as follows.
\begin{align}
\sum_{\mathrm{cyc}} A_{ijk} & = \sum_{\substack{i, j, k \\ \mathrm{cyc}}} A_{ijk} = A_{123} + A_{231} + A_{312}.
\end{align}
The \( (i,j) \)-entry of a matrix or a matrix-valued quantity \( A \) is denoted by \( A^i_j \). A \emph{connection} means either a Koszul connection (denoted by \( \del \)) or a principal connection (denoted by \( \omega \)), depending on the context. Quantities related to the Levi-Civita connection are denoted by a bar for the reader's convenience. For example, \( \bar{\del} \), \( \bar{\omega} \), \( \bar{H} \), etc. However, this is irrelevant to the notation \( \bar{s} \) of smooth frame with a bar.

\section{Asymptotics of metric Chern-Simons invariant} \label{S:asymptotics}

\subsection{Setup} \label{SS:setup}

We start with some preliminaries and a description of the problem. The goal of this section is to find the asymptotics of the \( \mathrm{SO}(3) \)-Chern-Simons invariant. We first recall the definition and basic properties of the Chern-Simons invariant, which are originated in~\cite{MR0353327}. A gentle introduction to \( 3 \)-dimensional Chern-Simons theory can be found in~\cite{BaseilhacCS} and~\cite{MR1337109}.

\begin{pdef} \label{D:CS}
Let \( G \) be a Lie group and \( \mathfrak{g} \) be its Lie algebra. Suppose that we are given an \( \Ad \)-invariant\footnote{\( {\left\langle \Ad_g a, \Ad_g b \right\rangle} = {\left\langle a, b \right\rangle} \) for all \( g \in G \) and \( a, b \in \mathfrak{g} \).} symmetric bilinear form \( {\left\langle {}\cdot{}, {}\cdot{} \right\rangle} \colon \mathfrak{g} \times \mathfrak{g} \to \mathbb{C} \). Let \( \pi \colon P \surjto M \) be a smooth principal \( G \)-bundle with a connection \( \omega \in \Omega^1 {\left( P, \mathfrak{g} \right)} \). The \emph{Chern-Simons\/ \( 3 \)-form} is \( \cs {\left( \omega \right)} \in \Omega^3 {\left( P, \mathbb{C} \right)} \) defined by
\begin{align}
\cs {\left( \omega \right)} & = {\left\langle \omega \wedge \Omega \right\rangle} - \frac{1}{6} {\left\langle \omega \wedge {\left[ \omega \wedge \omega \right]} \right\rangle},
\end{align}
where \( \Omega = \mathrm{d} \omega + \frac{1}{2} {\left[ \omega \wedge \omega \right]} \in \Omega^2 {\left( P, \mathfrak{g} \right)} \) is the curvature of \( \omega \). In particular, if \( M \) is a compact oriented smooth \( 3 \)-manifold with or without boundary, and if there is a global smooth section \( \sigma \colon M \injto P \), the \emph{Chern-Simons invariant} is defined by
\begin{align}
\CS_G {\left( M, \omega, \sigma \right)} & = \int_M \sigma^* \cs {\left( \omega \right)} \in \mathbb{C},
\end{align}
which depends on the choice of the global smooth section \( \sigma \).
\end{pdef}

\noindent If \( G \) is a Lie group of matrices, the standard choice of the \( \Ad \)-invariant symmetric bilinear form \( {\left\langle {}\cdot{}, {}\cdot{} \right\rangle} \colon \mathfrak{g} \times \mathfrak{g} \to \mathbb{C} \) in Definition~\ref{D:CS} is
\begin{align}
{\left\langle a, b \right\rangle} & = - \frac{1}{8 \boldsymbol{\pi}^2} \tr \! {\left( a b \right)} \qquad (a, b \in \mathfrak{g}), \label{E:Ad-invariant-form}
\end{align}
which is what we follow throughout the paper.

Some basic properties of the Chern-Simons \( 3 \)-form are shown in the following. See, e.g., Proposition~3.2 of~\cite{BaseilhacCS} and Proposition~1.27 of~\cite{MR1337109}.

\begin{pprop} \label{P:CS}
Let \( G \) be a Lie group. Let \( \pi \colon P \surjto M \) be a smooth principal \( G \)-bundle with a connection \( \omega \). Then the following two hold.
\begin{enumerate}
\item \( \mathrm{d} \cs {\left( \omega \right)} = {\left\langle \Omega \wedge \Omega \right\rangle} \), where\/ \( \Omega \) is the curvature of \( \omega \).
\item If \( \varphi \colon P \to P \) is a gauge transformation with the associated map \( g_{\varphi} \colon P \to G \) given by \( \varphi(p) = p \cdot g_{\varphi}(p) \) for all \( p \in P \), then
\begin{align}
\varphi^* \cs {\left( \omega \right)} & = \cs {\left( \varphi^* \omega \right)} = \cs {\left( \omega \right)} + \mathrm{d} {\left\langle \Ad_{g_{\varphi}^{-1}} \omega \wedge g_{\varphi}^* \mu \right\rangle} - \frac{1}{6} g_{\varphi}^* {\left\langle \mu \wedge {\left[ \mu \wedge \mu \right]} \right\rangle},
\end{align}
where \( \mu \) is the Maurer-Cartan\/ \( 1 \)-form on \( G \).
\end{enumerate}
\end{pprop}

If \( M \) is a compact oriented Riemannian \( 3 \)-manifold with or without boundary, the oriented orthonormal frame bundle \( FM \surjto M \) is a smooth principal \( \mathrm{SO}(3) \)-bundle with the Levi-Civita connection \( \bar{\omega} \in \Omega^1 {\left( FM, \mathfrak{so}(3) \right)} \). The corresponding Chern-Simons invariant is called the \emph{\( \mathrm{SO}(3) \)-Chern-Simons invariant} or the \emph{metric Chern-Simons invariant}. For closed oriented Riemannian \( 3 \)-manifolds, it is well known that this invariant does not depend on the choice of a global smooth section up to an even integral difference as follows. See Section~6 of~\cite{MR0353327} for more details.

\begin{pprop} \label{P:metric-CS}
Let \( M \) be a closed oriented Riemannian\/ \( 3 \)-manifold. If \( s \) and \( s' \) are two global oriented orthonormal smooth frames for \( M \), then
\begin{align}
& \CS_{\mathrm{SO}(3)} {\left( M, s' \right)} - \CS_{\mathrm{SO}(3)} {\left( M, s \right)} \in 2 \mathbb{Z}.
\end{align}
Since \( M \) is parallelizable,\/ \( \CS_{\mathrm{SO}(3)} (M) \in \mathbb{R} / 2 \mathbb{Z} \) is well defined.
\end{pprop}
\begin{proof}
By Proposition~\ref{P:CS}, the difference equals \( - \frac{1}{6} \int_M g^* {\left\langle \mu \wedge {\left[ \mu \wedge \mu \right]} \right\rangle} = 2 \deg g \), where \( g \colon M \to \mathrm{SO}(3) \) is the smooth map given by the change-of-basis matrix from \( s \) to \( s' \) (i.e., \( s' = s \cdot g \)) and \( \mu \) is the Maurer-Cartan \( 1 \)-form on \( \mathrm{SO}(3) \).
\end{proof}

\noindent However, this is no longer true for compact oriented Riemannian \( 3 \)-manifolds with boundary. Even up to an even integral difference, the invariant depends significantly on the boundary behavior of the global smooth section.

Our main interest is convex-cocompact hyperbolic \( 3 \)-manifolds with at least one hyperbolic end. One characterization is the following (Definition~4.1 of~\cite{MR2386723}).

\begin{pdef} \label{D:convex-cocompact}
A connected oriented complete hyperbolic \( 3 \)-manifold \( M \) without boundary is said to be \emph{convex-cocompact} iff there exists a compact embedded submanifold \( C \subseteq M \) with convex\footnote{The second fundamental form is negative semidefinite.} boundary \( \partial C \) such that the normal exponential map from \( \partial C \) to the conformal boundary \( \partial_{\infty} M \) is a diffeomorphism. Each connected component of the complement \( M \setminus C \) is called a \emph{hyperbolic end} of \( M \).
\end{pdef}

\noindent Typical examples of convex-cocompact hyperbolic \( 3 \)-manifolds include Schottky \( 3 \)-manifolds (having one hyperbolic end) and quasi-Fuchsian \( 3 \)-manifolds (having two hyperbolic ends). See Chapters~2--3 of~\cite{MR3586015} for more details.

Throughout Section~\ref{S:asymptotics}, we consider a convex-cocompact hyperbolic \( 3 \)-manifold \( M \) with a compact submanifold \( C \) in Definition~\ref{D:convex-cocompact}. For each \( r \in \mathbb{R} \), we define a subset \( S_r \subseteq M \) as follows. Let \( S_0 = S = \partial C \). For each \( r > 0 \), let \( S_r \) be the set of all points in \( M \setminus C \) at distance \( r \) from \( S \). For each \( r < 0 \), let \( S_r \) be the set of all points in \( C \) at distance \( -r \) from \( S \). Lemma~\ref{L:foliation} below describes the geometry of each \( S_r \). This can be found in Lemma~2.7 of~\cite{MR2328927} and Lemma~2.2 of~\cite{MR2386723}.

\begin{plem} \label{L:foliation}
There is \( \varepsilon_0 > 0 \) such that the following hold for every \( r \in (-\varepsilon_0, \infty) \).
\begin{enumerate}
\item \( S_r \) is a closed convex smooth surface embedded in \( M \), which is diffeomorphic to the conformal boundary \( \partial_{\infty} M \). We orient \( S_r \) by\/ \( \frac{\partial}{\partial r} \).
\item The closest-point projection \( u_r \colon S_r \to S \) is an orientation-preserving diffeomorphism (but not necessarily an isometry).
\item Let\/ \( \I_r \) and \( B_r \) be the first fundamental form (i.e., the induced metric) and the Weingarten map (i.e., the shape operator) on \( S_r \) respectively. Then
\begin{align}
\begin{aligned}
\I_r {\left( X, Y \right)} & = \I {\left( {\left( \cosh(r) I - \sinh(r) B \right)} {\left( u_r \right)}_* X, {\left( \cosh(r) I - \sinh(r) B \right)} {\left( u_r \right)}_* Y \right)}, \\
B_r & = - {\left( u_r^{-1} \right)}_* {\left( \cosh(r) I - \sinh(r) B \right)}^{-1} {\left( \sinh(r) I - \cosh(r) B \right)} {\left( u_r \right)}_*
\end{aligned}
\end{align}
for all \( X, Y \in \Gamma {\left( T S_r \right)} \), where\/ \( \I = \I_0 \), \( B = B_0 \), and \( I \) is the identity operator.
\end{enumerate}
\end{plem}

\begin{pwarn}
Some signs in Lemma~\ref{L:foliation} are opposite to those in~\cite{MR2328927} and~\cite{MR2386723}. This is due to the different choice of orientation: We take the outward orientation, while they take the inward orientation.
\end{pwarn}

\noindent We keep using the notations introduced in Lemma~\ref{L:foliation} throughout this paper. We obtain a foliation \( {\left\{ S_r \right\}}_{r \in (-\varepsilon_0, \infty)} \) by equidistant closed convex oriented smooth surfaces toward the conformal boundary \( S_{\infty} = \partial_{\infty} M \) of \( M \). For each \( r \in (-\varepsilon_0, \infty) \), let \( C_r \) be the compact embedded submanifold enclosed by \( S_r \) so that \( S_r = \partial C_r \). Figure~\ref{FIG:quasi-Fuchsian} provides a visualization of the setup.

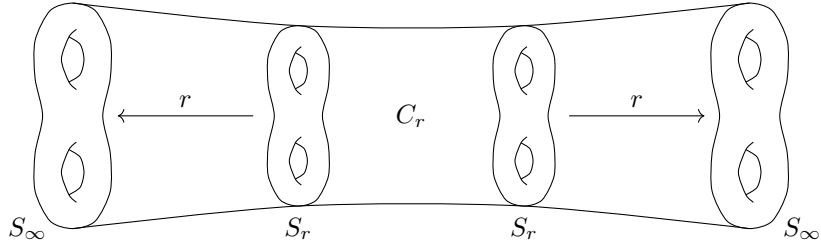
\begin{figure}
\begin{tikzpicture}
\draw plot [smooth] coordinates {(0,0) (3,.3) (6,.3) (9,0)};
\draw plot [smooth] coordinates {(0,3) (3,2.7) (6,2.7) (9,3)};
\draw plot [smooth cycle] coordinates {(0,0) (-.3,.1) (-.5,.5) (-.5,1) (-.4,1.5) (-.5,2) (-.5,2.5) (-.3,2.9) (0,3) (.3,2.9) (.5,2.5) (.5,2) (.4,1.5) (.5,1) (.5,.5) (.3,.1)};
\draw plot [smooth] coordinates {(.05,.35) (-.05,.45) (-.15,.75) (-.05,1.05) (.05,1.15)};
\draw plot [smooth] coordinates {(-.05,.45) (.1,.55) (.15,.75) (.1,.95) (-.05,1.05)};
\draw plot [smooth] coordinates {(.05,1.85) (-.05,1.95) (-.15,2.25) (-.05,2.55) (.05,2.65)};
\draw plot [smooth] coordinates {(-.05,1.95) (.1,2.05) (.15,2.25) (.1,2.45) (-.05,2.55)};
\draw plot [smooth cycle] coordinates {(9,0) (8.7,.1) (8.5,.5) (8.5,1) (8.6,1.5) (8.5,2) (8.5,2.5) (8.7,2.9) (9,3) (9.3,2.9) (9.5,2.5) (9.5,2) (9.4,1.5) (9.5,1) (9.5,.5) (9.3,.1)};
\draw plot [smooth] coordinates {(9.05,.35) (8.95,.45) (8.85,.75) (8.95,1.05) (9.05,1.15)};
\draw plot [smooth] coordinates {(8.95,.45) (9.1,.55) (9.15,.75) (9.1,.95) (8.95,1.05)};
\draw plot [smooth] coordinates {(9.05,1.85) (8.95,1.95) (8.85,2.25) (8.95,2.55) (9.05,2.65)};
\draw plot [smooth] coordinates {(8.95,1.95) (9.1,2.05) (9.15,2.25) (9.1,2.45) (8.95,2.55)};
\draw plot [smooth cycle] coordinates {(3,.3) (2.76,.38) (2.6,.7) (2.6,1.1) (2.68,1.5) (2.6,1.9) (2.6,2.3) (2.76,2.62) (3,2.7) (3.24,2.62) (3.4,2.3) (3.4,1.9) (3.32,1.5) (3.4,1.1) (3.4,.7) (3.24,.38)};
\draw plot [smooth] coordinates {(3.04,.58) (2.96,.66) (2.88,.9) (2.96,1.14) (3.04,1.22)};
\draw plot [smooth] coordinates {(2.96,.66) (3.08,.74) (3.12,.9) (3.08,1.06) (2.96,1.14)};
\draw plot [smooth] coordinates {(3.04,1.78) (2.96,1.86) (2.88,2.1) (2.96,2.34) (3.04,2.42)};
\draw plot [smooth] coordinates {(2.96,1.86) (3.08,1.94) (3.12,2.1) (3.08,2.26) (2.96,2.34)};
\draw plot [smooth cycle] coordinates {(6,.3) (5.76,.38) (5.6,.7) (5.6,1.1) (5.68,1.5) (5.6,1.9) (5.6,2.3) (5.76,2.62) (6,2.7) (6.24,2.62) (6.4,2.3) (6.4,1.9) (6.32,1.5) (6.4,1.1) (6.4,.7) (6.24,.38)};
\draw plot [smooth] coordinates {(6.04,.58) (5.96,.66) (5.88,.9) (5.96,1.14) (6.04,1.22)};
\draw plot [smooth] coordinates {(5.96,.66) (6.08,.74) (6.12,.9) (6.08,1.06) (5.96,1.14)};
\draw plot [smooth] coordinates {(6.04,1.78) (5.96,1.86) (5.88,2.1) (5.96,2.34) (6.04,2.42)};
\draw plot [smooth] coordinates {(5.96,1.86) (6.08,1.94) (6.12,2.1) (6.08,2.26) (5.96,2.34)};
\node at (4.5,1.5) {\( C_r \)};
\draw[->] (2.4,1.5) -- (.6,1.5);
\node[anchor=south] at (1.5,1.5) {\( r \)};
\draw[->] (6.6,1.5) -- (8.4,1.5);
\node[anchor=south] at (7.5,1.5) {\( r \)};
\node at (3,0) {\( S_r \)};
\node at (6,0) {\( S_r \)};
\node at (-.6,0) {\( S_{\infty} \)};
\node at (9.7,0) {\( S_{\infty} \)};
\end{tikzpicture}
\caption{An illustration of the case of quasi-Fuchsian \( M \).}
\label{FIG:quasi-Fuchsian}
\end{figure}

Let \( \pi \colon FM \surjto M \) be the oriented orthonormal frame bundle, which is a smooth principal \( \mathrm{SO}(3) \)-bundle. Let \( \bar{\omega} \in \Omega^1 {\left( FM, \mathfrak{so}(3) \right)} \) be its Levi-Civita connection. As an analogue of Lemma~4.2 of~\cite{MR2386723}, we want to find the asymptotic behavior of the integral of the \( \mathrm{SO}(3) \)-Chern-Simons \( 3 \)-form
\begin{align}
& \rho \longmapsto \int_{[S, S_{\rho}]} s^* \cs {\left( \bar{\omega} \right)} \quad \text{as} \ \rho \to \infty \label{E:asymptotics}
\end{align}
with respect to a given global smooth frame \( s \colon M \injto FM \), where \( [S, S_{\rho}] \) is the compact embedded submanifold enclosed by \( S \) and \( S_{\rho} \). Unlike the case of volume, it depends on the choice of the global smooth frame \( s \). At first sight, a natural candidate would be a global smooth frame consisting of two tangent vectors and one normal vector with respect to each \( S_r \), but such a frame exists iff each \( S_r \) is a union of tori by the Poincar{\'e}-Hopf index theorem. Since the Chern-Simons invariant is well known to depend significantly on the boundary behavior of the global smooth section, it is reasonable to impose a certain ``convergence'' condition on the global smooth frame along the foliation in order to achieve the asymptotics. We will see that being \emph{constant} in Definition~\ref{D:constant-frame} is the condition that works.

For convenience, let
\begin{align}
A_r & = \cosh(r) I - \sinh(r) B
\end{align}
for each \( r \in (-\varepsilon_0, \infty) \) so that \( \I_r = \I {\left( A_r {\left( u_r \right)}_* {}\cdot{}, A_r {\left( u_r \right)}_* {}\cdot{} \right)} \) by Lemma~\ref{L:foliation}. Each \( A_r \) is a linear automorphism on \( T_x S \) at each \( x \in S \).

\begin{plem} \label{L:Ar}
For each \( r \in (-\varepsilon_0, \infty) \), the map \( A_r \circ {\left( u_r \right)}_* \) induces an isomorphism \( A_r {\left( u_r \right)}_* \colon {\left. FM \right|}_{S_r} \to {\left. FM \right|}_S \) between smooth principal\/ \( \mathrm{SO}(3) \)-bundles.
\end{plem}
\begin{proof}
Since \( T_{x_r} M = T_{x_r} S_r \oplus {\big\langle {\big. \frac{\partial}{\partial r} \big|}_{x_r} \big\rangle} \) for all \( x_r \in S_r \) and \( \I_M = \I_r \oplus \mathrm{d} r^2 \), the isometry \( {\left. A_r \right|}_x \circ {\left. {\left( u_r \right)}_* \right|}_{x_r} \colon T_{x_r} S_r \to T_x S \) induces an isometry \( {\big( {\left. A_r \right|}_x \circ {\left. {\left( u_r \right)}_* \right|}_{x_r} \big)} \oplus {\big( {\left. J_r \right|}_{x_r} \big)} \colon T_{x_r} M \to T_x M \), where \( x = u_r {\left( x_r \right)} \in S \) and \( {\left. J_r \right|}_{x_r} \colon {\big\langle {\big. \frac{\partial}{\partial r} \big|}_{x_r} \big\rangle} \to {\big\langle {\big. \frac{\partial}{\partial r} \big|}_x \big\rangle} \) is the map sending \( {\left. \frac{\partial}{\partial r} \right|}_{x_r} \) to \( {\left. \frac{\partial}{\partial r} \right|}_x \). This induces the desired map \( {\left. FM \right|}_{x_r} \to {\left. FM \right|}_x \).
\end{proof}

Given a smooth frame \( s \colon M \injto FM \), Lemma~\ref{L:Ar} provides a way to ``compare'' smooth frames \( {\left. s \right|}_{S_r} \colon S_r \injto {\left. FM \right|}_{S_r} \). This introduces the notion of constant frame.

\begin{pdef}[Constant frame] \label{D:constant-frame}
Let \( R \) be a smooth submanifold embedded in \( S \) of codimension zero (e.g., \( R = S \)), and let \( R_r = u_r^{-1}[R] \subseteq S_r \) for each \( r \in (-\varepsilon_0, \infty) \). Let \( (R_{-\varepsilon_0}, R_{\infty}) \) be the smooth submanifold formed by the equidistant foliation \( {\left\{ R_r \right\}}_{r \in (-\varepsilon_0, \infty)} \). Note that a smooth frame \( s \colon (R_{-\varepsilon_0}, R_{\infty}) \injto {\left. FM \right|}_{(R_{-\varepsilon_0}, R_{\infty})} \) determines a family \( {\left\{ g_r \right\}}_{r \in (-\varepsilon_0, \infty)} \) of smooth maps \( g_r \colon R \to \mathrm{SO}(3) \) given by
\begin{align}
A_r {\left( u_r \right)}_* {\left( {\left. s \right|}_{x_r} \right)} & = {\left. s \right|}_x \cdot {\left. g_r \right|}_x \qquad (x_r \in R_r, \ x = u_r {\left( x_r \right)} \in R).
\end{align}
For an interval \( J \subseteq (-\varepsilon_0, \infty) \), the smooth frame \( s \) is said to be \emph{constant} along the foliation \( {\left\{ R_r \right\}}_{r \in J} \) iff the map \( r \mapsto g_r \) is constant on \( J \) (i.e., \( g_r = g_{r'} \) for all \( r, r' \in J \)).
\end{pdef}

\begin{prmrk}
In the setup of Definition~\ref{D:constant-frame}, if \( s \colon (R_{-\varepsilon_0}, R_{\infty}) \injto {\left. FM \right|}_{(R_{-\varepsilon_0}, R_{\infty})} \) is a smooth frame that is constant along the foliation \( {\left\{ R_r \right\}}_{r \in J} \) for some interval \( 0 \in J \subseteq (-\varepsilon_0, \infty) \), then the following diagram commutes for all \( r \in J \).
\begin{align}
\begin{tikzcd}[ampersand replacement=\&]
R_r \ar[r, hook, "{\left. s \right|}_{R_r}"] \ar[d, "u_r"'] \& {\left. FM \right|}_{R_r} \ar[d, "A_r {\left( u_r \right)}_*"] \\
R \ar[r, hook, "{\left. s \right|}_R"] \& {\left. FM \right|}_R
\end{tikzcd}
\end{align}
\end{prmrk}

\begin{pprop}[Existence of constant frame] \label{P:constant-frame}
For any interval \( J \subseteq (-\varepsilon_0, \infty) \) with\/ \( \inf J > -\varepsilon_0 \), there exists a global smooth frame \( s \colon M \injto FM \) that is constant along the foliation\/ \( {\left\{ S_r \right\}}_{r \in J} \). In particular, there exists a global smooth frame \( s \colon M \injto FM \) that is constant along the foliation\/ \( {\left\{ S_r \right\}}_{r \in [0, \infty)} \).
\end{pprop}
\begin{proof}
Since every compact orientable smooth \( 3 \)-manifold is parallelizable, so is \( M \). It has a global smooth frame \( \hat{s} \colon M \injto FM \). Let \( j = \inf J \) and \( i = {\left( j - \varepsilon_0 \right)} / 2 \) so that \( -\varepsilon_0 < i < j < \infty \). Define \( s' \colon M \injto FM \) by
\begin{align*}
& x \longmapsto \begin{dcases*}
{\left. \hat{s} \right|}_x & if \( x \in C_i \), \\
{\left( A_r {\left( u_r \right)}_* \right)}^{-1} {\left( A_i {\left( u_i \right)}_* \right)} {\big( {\left. \hat{s} \right|}_{u_i^{-1} \circ u_r(x)} \big)} & if \( x \in S_r \) for some \( r \in [i, \infty) \).
\end{dcases*}
\end{align*}
Then it is a global \emph{continuous} frame that is smooth on \( M \setminus S_i \), which is constant along the foliation \( {\left\{ S_r \right\}}_{r \in [i, \infty)} \). Let \( g \colon M \to \mathrm{SO}(3) \) be the continuous map given by \( s' = \hat{s} \cdot g \). Note that it is smooth on \( [S_j, S_{\infty}) \). By Lemma~\ref{L:SAT}, there exists a smooth map \( h \colon M \to \mathrm{SO}(3) \) that coincides with \( g \) on \( [S_j, S_{\infty}) \). Then \( s = \hat{s} \cdot h \colon M \injto FM \) is the desired global smooth frame.
\end{proof}

From now on, we take \( \varepsilon = \varepsilon_0 / 2 > 0 \) for some technical convenience. The gauge transformation between constant frames is simply given as follows.

\begin{plem} \label{L:constant-frame-gauge-transformation}
In the setup of Definition~\ref{D:constant-frame}, let \( s, s' \colon (R_{-\varepsilon}, R_{\infty}) \injto {\left. FM \right|}_{(R_{-\varepsilon}, R_{\infty})} \) be two smooth frames that are constant along the foliation\/ \( {\left\{ R_r \right\}}_{r \in (-\varepsilon, \infty)} \). Let \( g \colon (R_{-\varepsilon}, R_{\infty}) \to \mathrm{SO}(3) \) be the smooth map given by \( s' = s \cdot g \). Then \( g = {\left. g \right|}_R \circ u \), where \( u \colon (R_{-\varepsilon}, R_{\infty}) \surjto R \) is the projection.
\end{plem}
\begin{proof}
Fix \( r \in (-\varepsilon, \infty) \). Fix \( x_r \in R_r \) and let \( x = u_r {\left( x_r \right)} \in R \). It suffices to show that \( {\left. g \right|}_{x_r} = {\left. g \right|}_x \). We have \( {\left. s' \right|}_x = {\left. s \right|}_x \cdot {\left. g \right|}_x \) and
\begin{align*}
{\left. s' \right|}_x & = A_r {\left( u_r \right)}_* {\left( {\left. s' \right|}_{x_r} \right)} = A_r {\left( u_r \right)}_* {\left( {\left. s \right|}_{x_r} \cdot {\left. g \right|}_{x_r} \right)} = A_r {\left( u_r \right)}_* {\left( {\left. s \right|}_{x_r} \right)} \cdot {\left. g \right|}_{x_r} = {\left. s \right|}_x \cdot {\left. g \right|}_{x_r}.
\end{align*}
This completes the proof.
\end{proof}

We want to compute the asymptotic behavior of~\eqref{E:asymptotics} with respect to a global smooth frame \( s = {\left( E_1, E_2, E_3 \right)} \colon M \injto FM \) that is constant along the equidistant foliation \( {\left\{ S_r \right\}}_{r \in (-\varepsilon, \infty)} \). In order to evaluate the integral, we divide \( S \) into pieces so that each admits a smooth frame consisting of two tangent vectors and one normal vector. For example, we can divide each connected component of \( S \) into two regions: a small disc \( R' \) and the complementary surface \( R'' \) with one boundary. Thus, let \( R \) be an arbitrary smooth submanifold embedded in \( S \) of codimension zero such that its boundary \( \partial R \) is diffeomorphic to the circle and there exists a smooth frame \( \bar{s} = {\left( \bar{E}_1, \bar{E}_2, \bar{E}_3 \right)} \colon R \injto {\left. FM \right|}_R \) with \( \bar{E}_3 = \frac{\partial}{\partial r} \). Let \( {\left\{ R_r \right\}}_{r \in (-\varepsilon, \infty)} \) be the foliation as in Definition~\ref{D:constant-frame}. We extend \( \bar{s} \) to \( (R_{-\varepsilon}, R_{\infty}) \) by defining \( {\left. \bar{s} \right|}_{R_r} = {\left( A_r {\left( u_r \right)}_* \right)}^{-1} {\left( {\left. \bar{s} \right|}_R \right)} \) for all \( r \in (-\varepsilon, \infty) \) so that it is constant along the foliation \( {\left\{ R_r \right\}}_{r \in (-\varepsilon, \infty)} \). Let \( g \colon (R_{-\varepsilon}, R_{\infty}) \to \mathrm{SO}(3) \) be the smooth map given by \( s = \bar{s} \cdot g \). By the gauge transformation formula in Proposition~\ref{P:CS}, we have
\begin{align}
\int_{[R, R_{\rho}]} s^* \cs {\left( \bar{\omega} \right)} & = \int_{[R, R_{\rho}]} \, \underbrace{\vphantom{\int} \bar{s}^* \cs {\left( \bar{\omega} \right)}}_{(\textnormal{I})} + \underbrace{\vphantom{\int} \mathrm{d} {\left\langle \Ad_{g^{-1}} \bar{s}^* \bar{\omega} \wedge g^* \mu \right\rangle}}_{(\textnormal{II})} - \underbrace{\vphantom{\int} \frac{1}{6} g^* {\left\langle \mu \wedge {\left[ \mu \wedge \mu \right]} \right\rangle}}_{(\textnormal{III})}, \label{E:terms}
\end{align}
where \( \mu \) is the Maurer-Cartan \( 1 \)-form on \( \mathrm{SO}(3) \). Note that \( (\textnormal{III}) = 0 \) by Lemma~\ref{L:constant-frame-gauge-transformation}. We will evaluate the integral of the term \( (\textnormal{I}) \) in Subsection~\ref{SS:I} and that of the term \( (\textnormal{II}) \) in Subsection~\ref{SS:II}. Readers who are not interested in the detailed computations may skip the following Subsections~\ref{SS:I}--\ref{SS:II} and go to the evaluation result~\eqref{E:PQ} at the end of Subsection~\ref{SS:II}.

\subsection{The integral of (I)} \label{SS:I}

In this subsection, we evaluate the integral of the term \( (\textnormal{I}) \) in~\eqref{E:terms}. The evaluation result is~\eqref{E:I}. Let \( {\left( \bar{\varepsilon}^1, \bar{\varepsilon}^2, \bar{\varepsilon}^3 \right)} \) be the dual of the smooth frame \( \bar{s} = {\left( \bar{E}_1, \bar{E}_2, \bar{E}_3 \right)} \). The pullback of the \( \mathrm{SO}(3) \)-Chern-Simons \( 3 \)-form via a smooth frame consisting of two tangent vectors and one normal vector has a tractable expression as follows. See Appendix~\ref{AA:CS-normal-frame} for its proof.
\begin{align}
\bar{s}^* \cs {\left( \bar{\omega} \right)} & = \frac{1}{4 \boldsymbol{\pi}^2} K_r \, {\left( \bar{s}^* \bar{\omega} \right)}^1_2 \wedge \bar{\varepsilon}^1 \wedge \bar{\varepsilon}^2 = \frac{1}{4 \boldsymbol{\pi}^2} K_r \, {\left( \bar{s}^* \bar{\omega} \right)}^1_2 {\left( \frac{\partial}{\partial r} \right)} \, \mathrm{d} a_r \wedge \mathrm{d} r \label{E:CS-normal-frame}
\end{align}
on \( (R_{-\varepsilon}, R_{\infty}) \), where \( K_r \) and \( \mathrm{d} a_r \) are the (intrinsic) Gaussian curvature and the area form of \( S_r \) respectively. Therefore, we have
\begin{align*}
\int_{[R, R_{\rho}]} (\textnormal{I}) & = \frac{1}{4 \boldsymbol{\pi}^2} \int_0^{\rho} \! \int_{R_r} {\left( \bar{s}^* \bar{\omega} \right)}^1_2 {\left( \frac{\partial}{\partial r} \right)} \, K_r \, \mathrm{d} a_r \, \mathrm{d} r \\
& = \frac{1}{4 \boldsymbol{\pi}^2} \int_0^{\rho} \! \int_{R_r} {\left( \bar{s}^* \bar{\omega} \right)}^1_2 {\left( \frac{\partial}{\partial r} \right)} \, u_r^* {\left( K \, \mathrm{d} a \right)} \, \mathrm{d} r \\
& = \frac{1}{4 \boldsymbol{\pi}^2} \int_0^{\rho} \! \int_R {\left( \bar{s}^* \bar{\omega} \right)}^1_2 {\left( \frac{\partial}{\partial r} \right)} \circ u_r^{-1} \, K \, \mathrm{d} a \, \mathrm{d} r \\
& = \frac{1}{4 \boldsymbol{\pi}^2} \int_R {\left( \int_0^{\rho} {\left( \bar{s}^* \bar{\omega} \right)}^1_2 {\left( \frac{\partial}{\partial r} \right)} \circ u_r^{-1} \, \mathrm{d} r \right)} \, K \, \mathrm{d} a.
\end{align*}

\begin{prmrk}
The \( r \)-integral in the last line is exactly (the minus of) the \emph{torsion number} introduced in~\cite{MR807069} of the geodesic orthogonal to \( S \): For each \( x \in R \),
\begin{align}
\int_0^{\rho} {\left. {\left( \bar{s}^* \bar{\omega} \right)}^1_2 {\left( \frac{\partial}{\partial r} \right)} \right|}_{u_r^{-1}(x)} \, \mathrm{d} r & = \int_{\gamma_{x, \rho}} \bar{s}^* \bar{\omega}^1_2 = - \tau {\left( \gamma_{x, \rho}, \bar{s} \right)},
\end{align}
where \( \gamma_{x, \rho} \) is the geodesic in \( [R, R_{\rho}] \) given by \( \gamma_{x, \rho}(r) = u_r^{-1}(x) \) for all \( r \in [0, \rho] \).
\end{prmrk}

Fix \( x \in R \), and let \( x_r = u_r^{-1}(x) \in R_r \) for each \( r \in [0, \infty) \). For each \( r \in [0, \infty) \), Koszul's formula gives
\begin{align*}
& {\left. {\left( \bar{s}^* \bar{\omega} \right)}^1_2 {\left( \frac{\partial}{\partial r} \right)} \right|}_{x_r} = \I_M {\left( {\left. \bar{E}_1 \right|}_{x_r}, {\left. \bar{\del}^M_{\frac{\partial}{\partial r}} \bar{E}_2 \right|}_{x_r} \right)} \\
& \qquad = \frac{1}{2} {\left( \I_M {\left( {\left. \bar{E}_2 \right|}_{x_r}, {\left. {\left[ \bar{E}_1, \frac{\partial}{\partial r} \right]} \right|}_{x_r} \right)} - \I_M {\left( {\left. \bar{E}_1 \right|}_{x_r}, {\left. {\left[ \bar{E}_2, \frac{\partial}{\partial r} \right]} \right|}_{x_r} \right)} \right)} \\
& \qquad = \frac{1}{2} {\left( \I_r {\left( {\left. \bar{E}_2 \right|}_{x_r}, {\left. {\left[ \bar{E}_1, \frac{\partial}{\partial r} \right]} \right|}_{x_r}^{\top} \right)} - \I_r {\left( {\left. \bar{E}_1 \right|}_{x_r}, {\left. {\left[ \bar{E}_2, \frac{\partial}{\partial r} \right]} \right|}_{x_r}^{\top} \right)} \right)},
\end{align*}
where the superscript \( \top \) denotes the projection of vector onto the leaf surface. Here, for each \( (i,j) \in \{ (1,2), (2,1) \} \), Lemma~\ref{L:Lie-bracket} gives
\begin{align*}
& \I_r {\left( {\left. \bar{E}_i \right|}_{x_r}, {\left. {\left[ \bar{E}_j, \frac{\partial}{\partial r} \right]} \right|}_{x_r}^{\top} \right)} = \I {\left( {\left. \bar{E}_i \right|}_x, A_r {\left( u_r \right)}_* {\left. {\left[ \bar{E}_j, \frac{\partial}{\partial r} \right]} \right|}_{x_r}^{\top} \right)} \\
& \qquad = \I {\left( {\left. \bar{E}_i \right|}_x, {\left( A_r|_x^{-1} \right)}^{\ell}_j \, A_r {\left. {\left[ \bar{E}_{\ell}, \frac{\partial}{\partial r} \right]} \right|}_x^{\top} \right)} + \sinh(r) {\left( \! {\left. {\left( I - B^2 \right)} A_r^{-1} \right|}_x \right)}^i_j \\
& \qquad = {\left( A_r|_x^{-1} \right)}^k_i {\left( A_r|_x^{-1} \right)}^{\ell}_j \, \I {\left( A_r {\left. \bar{E}_k \right|}_x, A_r {\left. {\left[ \bar{E}_{\ell}, \frac{\partial}{\partial r} \right]} \right|}_x^{\top} \right)} + \sinh(r) {\left( \! {\left. {\left( I - B^2 \right)} A_r^{-1} \right|}_x \right)}^i_j,
\end{align*}
where the Einstein summation convention is assumed. We discard the last term by Lemma~\ref{L:self-adjoint} and obtain
\begin{align*}
& {\left. {\left( \bar{s}^* \bar{\omega} \right)}^1_2 {\left( \frac{\partial}{\partial r} \right)} \right|}_{x_r} = \frac{1}{2} {\left( \begin{aligned}
& {\left( A_r|_x^{-1} \right)}^k_2 {\left( A_r|_x^{-1} \right)}^{\ell}_1 \\
& - {\left( A_r|_x^{-1} \right)}^k_1 {\left( A_r|_x^{-1} \right)}^{\ell}_2
\end{aligned} \right)} \, \I {\left( A_r {\left. \bar{E}_k \right|}_x, A_r {\left. {\left[ \bar{E}_{\ell}, \frac{\partial}{\partial r} \right]} \right|}_x^{\top} \right)} \\
& \qquad = \frac{1}{2 \det {\left. A_r \right|}_x} {\left( \I {\left( A_r {\left. \bar{E}_2 \right|}_x, A_r {\left. {\left[ \bar{E}_1, \frac{\partial}{\partial r} \right]} \right|}_x^{\top} \right)} - \I {\left( A_r {\left. \bar{E}_1 \right|}_x, A_r {\left. {\left[ \bar{E}_2, \frac{\partial}{\partial r} \right]} \right|}_x^{\top} \right)} \right)}
\end{align*}
for each \( r \in [0, \infty) \). Consequently, Lemma~\ref{L:limit} implies that
\begin{align}
\int_{[R, R_{\rho}]} (\textnormal{I}) & = \frac{\rho}{4 \boldsymbol{\pi}^2} \int_R P K \, \mathrm{d} a + (\textnormal{convergent terms}) \label{E:I}
\end{align}
as \( \rho \to \infty \), where \( P \colon R \to \mathbb{R} \) is the smooth function given by
\begin{align}
P & = \frac{1}{2 \det(I-B)} {\left( \begin{aligned}
& \I {\left( (I-B) \bar{E}_2, (I-B) {\left[ \bar{E}_1, \frac{\partial}{\partial r} \right]}^{\top} \right)} \\
& - \I {\left( (I-B) \bar{E}_1, (I-B) {\left[ \bar{E}_2, \frac{\partial}{\partial r} \right]}^{\top} \right)}
\end{aligned} \right)}. \label{E:P}
\end{align}
We will show in Subsection~\ref{SS:infinity} that \( P = 0 \) by describing it from infinity.

\subsection{The integral of (II)} \label{SS:II}

In this subsection, we evaluate the integral of the term \( (\textnormal{II}) \) in~\eqref{E:terms}. The evaluation result is~\eqref{E:II}. First of all, let \( {\left\{ L_1, L_2, L_3 \right\}} \) be the basis for \( \mathfrak{so}(3) \) defined by
\begin{align}
L_1 & = \begin{pmatrix*}[r]
0 & 0 & 0 \\
0 & 0 & -1 \\
0 & 1 & 0
\end{pmatrix*}, \qquad L_2 = \begin{pmatrix*}[r]
0 & 0 & 1 \\
0 & 0 & 0 \\
-1 & 0 & 0
\end{pmatrix*}, \qquad L_3 = \begin{pmatrix*}[r]
0 & -1 & 0 \\
1 & 0 & 0 \\
0 & 0 & 0
\end{pmatrix*}.
\end{align}
The boundary of \( [R, R_{\rho}] \) consists of \( R \), \( R_{\rho} \), and the ``cylinder'' \( T_{\rho} = \bigcup_{r \in [0, \rho]} \partial R_r \) (Figure~\ref{FIG:boundary}). For the asymptotics as \( \rho \to \infty \), we need to evaluate the integrals over \( R_{\rho} \) and \( T_{\rho} \). Let us first consider the integral over \( R_{\rho} \). Lemma~\ref{L:constant-frame-gauge-transformation} gives
\begin{align*}
& \int_{R_{\rho}} {\left\langle \Ad_{{\left. g \right|}_{R_{\rho}}^{-1}} {\left. {\left( \bar{s}^* \bar{\omega} \right)} \right|}_{R_{\rho}} \wedge {\left. g \right|}_{R_{\rho}}^* \mu \right\rangle} = \int_{R_{\rho}} u_{\rho}^* {\left\langle \Ad_{{\left. g \right|}_R^{-1}} {\left( u_{\rho}^{-1} \right)}^* {\left. {\left( \bar{s}^* \bar{\omega} \right)} \right|}_{R_{\rho}} \wedge {\left. g \right|}_R^* \mu \right\rangle} \\
& \qquad = \int_R {\left\langle \Ad_{{\left. g \right|}_R^{-1}} {\left( u_{\rho}^{-1} \right)}^* {\left. {\left( \bar{s}^* \bar{\omega} \right)} \right|}_{R_{\rho}} \wedge {\left. g \right|}_R^* \mu \right\rangle} \\
& \qquad = \int_R {\left( \begin{aligned}
& {\left\langle \Ad_{{\left. g \right|}_R^{-1}} {\left( u_{\rho}^{-1} \right)^*} {\left. {\left( \bar{s}^* \bar{\omega} \right)} \right|}_{R_{\rho}} {\left( \bar{E}_1 \right)}, {\left. g \right|}_R^* \mu {\left( \bar{E}_2 \right)} \right\rangle} \\
& - {\left\langle \Ad_{{\left. g \right|}_R^{-1}} {\left( u_{\rho}^{-1} \right)}^* {\left. {\left( \bar{s}^* \bar{\omega} \right)} \right|}_{R_{\rho}} {\left( \bar{E}_2 \right)}, {\left. g \right|}_R^* \mu {\left( \bar{E}_1 \right)} \right\rangle}
\end{aligned} \right)} \, \mathrm{d} a \\
& \qquad = - \sum_{\mathrm{cyc}} \int_R {\left( \begin{aligned}
& {\left( \bar{s}^* \bar{\omega} \right)}^i_j {\left( {\left( u_{\rho}^{-1} \right)}_* \bar{E}_1 \right)} {\left\langle \Ad_{{\left. g \right|}_R^{-1}} L_k, {\left. g \right|}_R^* \mu {\left( \bar{E}_2 \right)} \right\rangle} \\
& - {\left( \bar{s}^* \bar{\omega} \right)}^i_j {\left( {\left( u_{\rho}^{-1} \right)}_* \bar{E}_2 \right)} {\left\langle \Ad_{{\left. g \right|}_R^{-1}} L_k, {\left. g \right|}_R^* \mu {\left( \bar{E}_1 \right)} \right\rangle}
\end{aligned} \right)} \, \mathrm{d} a.
\end{align*}

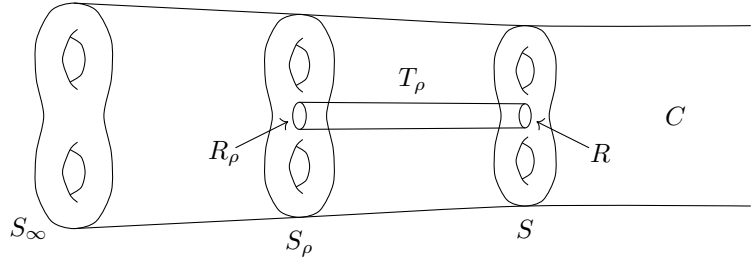
\begin{figure}
\begin{tikzpicture}
\draw plot [smooth] coordinates {(0,0) (3,.15) (6,.3) (9,.3)};
\draw plot [smooth] coordinates {(0,3) (3,2.85) (6,2.7) (9,2.7)};
\draw plot [smooth cycle] coordinates {(0,0) (-.3,.1) (-.5,.5) (-.5,1) (-.4,1.5) (-.5,2) (-.5,2.5) (-.3,2.9) (0,3) (.3,2.9) (.5,2.5) (.5,2) (.4,1.5) (.5,1) (.5,.5) (.3,.1)};
\draw plot [smooth] coordinates {(.05,.35) (-.05,.45) (-.15,.75) (-.05,1.05) (.05,1.15)};
\draw plot [smooth] coordinates {(-.05,.45) (.1,.55) (.15,.75) (.1,.95) (-.05,1.05)};
\draw plot [smooth] coordinates {(.05,1.85) (-.05,1.95) (-.15,2.25) (-.05,2.55) (.05,2.65)};
\draw plot [smooth] coordinates {(-.05,1.95) (.1,2.05) (.15,2.25) (.1,2.45) (-.05,2.55)};
\draw plot [smooth cycle] coordinates {(3,.15) (2.73,.24) (2.55,.6) (2.55,1.05) (2.64,1.5) (2.55,1.95) (2.55,2.4) (2.73,2.76) (3,2.85) (3.27,2.76) (3.45,2.4) (3.45,1.95) (3.36,1.5) (3.45,1.05) (3.45,.6) (3.27,.24)};
\draw plot [smooth] coordinates {(3.045,.465) (2.955,.555) (2.865,.825) (2.955,1.095) (3.045,1.185)};
\draw plot [smooth] coordinates {(2.955,.555) (3.09,.645) (3.135,.825) (3.09,1.005) (2.955,1.095)};
\draw plot [smooth] coordinates {(3.045,1.815) (2.955,1.905) (2.865,2.175) (2.955,2.445) (3.045,2.535)};
\draw plot [smooth] coordinates {(2.955,1.905) (3.09,1.995) (3.135,2.175) (3.09,2.355) (2.955,2.445)};
\draw plot [smooth cycle] coordinates {(6,.3) (5.76,.38) (5.6,.7) (5.6,1.1) (5.68,1.5) (5.6,1.9) (5.6,2.3) (5.76,2.62) (6,2.7) (6.24,2.62) (6.4,2.3) (6.4,1.9) (6.32,1.5) (6.4,1.1) (6.4,.7) (6.24,.38)};
\draw plot [smooth] coordinates {(6.04,.58) (5.96,.66) (5.88,.9) (5.96,1.14) (6.04,1.22)};
\draw plot [smooth] coordinates {(5.96,.66) (6.08,.74) (6.12,.9) (6.08,1.06) (5.96,1.14)};
\draw plot [smooth] coordinates {(6.04,1.78) (5.96,1.86) (5.88,2.1) (5.96,2.34) (6.04,2.42)};
\draw plot [smooth] coordinates {(5.96,1.86) (6.08,1.94) (6.12,2.1) (6.08,2.26) (5.96,2.34)};
\draw (3,1.5) ellipse (.09 and .18);
\draw (6,1.5) ellipse (.08 and .16);
\draw (3,1.32) -- (6,1.34);
\draw (3,1.68) -- (6,1.66);
\node at (8,1.5) {\( C \)};
\node at (-.6,0) {\( S_{\infty} \)};
\node at (3,-.2) {\( S_{\rho} \)};
\node at (6,0) {\( S \)};
\node[anchor=south] at (4.5,1.67) {\( T_{\rho} \)};
\draw[->] (6.85,1.075) -- (6.15,1.425);
\node at (7,1) {\( R \)};
\draw[->] (2.15,1.075) -- (2.85,1.425);
\node at (2,1) {\( R_{\rho} \)};
\end{tikzpicture}
\caption{The boundary of \( [R, R_{\rho}] \) when \( R \) is a small disc.}
\label{FIG:boundary}
\end{figure}

Fix \( x \in R \) and let \( x_{\rho} = u_{\rho}^{-1}(x) \in R_{\rho} \). For convenience, let \( A_{\rho}' = {\left( u_{\rho}^{-1} \right)}_* A_{\rho} {\left( u_{\rho} \right)}_* \). For each \( i, \ell \in \{ 1, 2 \} \), Koszul's formula and Lemma~\ref{L:foliation} give
\begin{align*}
& {\left( \bar{s}^* \bar{\omega} \right)}^1_2 {\left( {\left( u_{\rho}^{-1} \right)}_* {\left. \bar{E}_{\ell} \right|}_x \right)} = {\left( \bar{s}^* \bar{\omega} \right)}^1_2 {\left( A_{\rho}' {\left. \bar{E}_{\ell} \right|}_{x_{\rho}} \right)} \\
& \qquad = {\left( A_{\rho}'|_{x_{\rho}} \right)}^1_{\ell} \, {\left( \bar{s}^* \bar{\omega} \right)}^1_2 {\left( {\left. \bar{E}_1 \right|}_{x_{\rho}} \right)} + {\left( A_{\rho}'|_{x_{\rho}} \right)}^2_{\ell} \, {\left( \bar{s}^* \bar{\omega} \right)}^1_2 {\left( {\left. \bar{E}_2 \right|}_{x_{\rho}} \right)} \\
& \qquad = {\left( A_{\rho}'|_{x_{\rho}} \right)}^1_{\ell} \, \I_M {\left( {\left. \bar{E}_1 \right|}_{x_{\rho}}, {\left. \bar{\del}^M_{\bar{E}_1} \bar{E}_2 \right|}_{x_{\rho}} \right)} + {\left( A_{\rho}'|_{x_{\rho}} \right)}^2_{\ell} \, \I_M {\left( {\left. \bar{E}_1 \right|}_{x_{\rho}}, {\left. \bar{\del}^M_{\bar{E}_2} \bar{E}_2 \right|}_{x_{\rho}} \right)} \\
& \qquad = {\left( A_{\rho}'|_{x_{\rho}} \right)}^1_{\ell} \, \I_M {\left( {\left. \bar{E}_1 \right|}_{x_{\rho}}, {\left. {\left[ \bar{E}_1, \bar{E}_2 \right]} \right|}_{x_{\rho}} \right)} + {\left( A_{\rho}'|_{x_{\rho}} \right)}^2_{\ell} \, \I_M {\left( {\left. \bar{E}_2 \right|}_{x_{\rho}}, {\left. {\left[ \bar{E}_1, \bar{E}_2 \right]} \right|}_{x_{\rho}} \right)} \\
& \qquad = \I_M {\left( A_{\rho}' {\left. \bar{E}_{\ell} \right|}_{x_{\rho}}, {\left. {\left[ \bar{E}_1, \bar{E}_2 \right]} \right|}_{x_{\rho}} \right)} = \I_{\rho} {\left( A_{\rho}' {\left. \bar{E}_{\ell} \right|}_{x_{\rho}}, {\left. {\left[ \bar{E}_1, \bar{E}_2 \right]} \right|}_{x_{\rho}} \right)} \\
& \qquad = \I {\left( A_{\rho} {\left. \bar{E}_{\ell} \right|}_x, A_{\rho} {\left( u_{\rho} \right)}_* {\left. {\left[ \bar{E}_1, \bar{E}_2 \right]} \right|}_{x_{\rho}} \right)}, \\
& {\left( \bar{s}^* \bar{\omega} \right)}^i_3 {\left( {\left( u_{\rho}^{-1} \right)}_* {\left. \bar{E}_{\ell} \right|}_x \right)} = \I_M {\left( {\left. \bar{E}_i \right|}_{x_{\rho}}, {\left. \bar{\del}^M_{{\left( u_{\rho}^{-1} \right)}_* \bar{E}_{\ell}} \frac{\partial}{\partial r} \right|}_{x_{\rho}} \right)} \\
& \qquad = \I_M {\left( {\left. \bar{E}_i \right|}_{x_{\rho}}, - B_{\rho} {\left( u_{\rho}^{-1} \right)}_* {\left. \bar{E}_{\ell} \right|}_x \right)} = \I_{\rho} {\left( {\left. \bar{E}_i \right|}_{x_{\rho}}, - B_{\rho} {\left( u_{\rho}^{-1} \right)}_* {\left. \bar{E}_{\ell} \right|}_x \right)} \\
& \qquad = \I {\left( {\left. \bar{E}_i \right|}_x, - A_{\rho} {\left( u_{\rho} \right)}_* B_{\rho} {\left( u_{\rho}^{-1} \right)}_* {\left. \bar{E}_{\ell} \right|}_x \right)} = {\left( \sinh(\rho) I - \cosh(\rho) B|_x \right)}^i_{\ell}.
\end{align*}
The former converges as \( \rho \to \infty \) by Lemma~\ref{L:Lie-bracket}, so the integral of the terms \( (i,j,k) = (1,2,3) \) also converges as \( \rho \to \infty \) by Lemma~\ref{L:limit}. We find that
\begin{align}
\int_{R_{\rho}} {\left\langle \Ad_{g^{-1}} \bar{s}^* \bar{\omega} \wedge g^* \mu \right\rangle} & = \boldsymbol{e}^{\rho} \int_R Q \, \mathrm{d} a + (\textnormal{convergent terms})
\end{align}
as \( \rho \to \infty \), where \( Q \colon R \to \mathbb{R} \) is the smooth function given by
\begin{align}
\begin{aligned}
Q & = \begin{aligned}[t]
& - \frac{(I-B)^2_1}{2} {\left\langle \Ad_{g^{-1}} L_1, g^* \mu {\left( \bar{E}_2 \right)} \right\rangle} + \frac{(I-B)^2_2}{2} {\left\langle \Ad_{g^{-1}} L_1, g^* \mu {\left( \bar{E}_1 \right)} \right\rangle} \\
& + \frac{(I-B)^1_1}{2} {\left\langle \Ad_{g^{-1}} L_2, g^* \mu {\left( \bar{E}_2 \right)} \right\rangle} - \frac{(I-B)^1_2}{2} {\left\langle \Ad_{g^{-1}} L_2, g^* \mu {\left( \bar{E}_1 \right)} \right\rangle}
\end{aligned} \\
& = \frac{1}{2} \det(I-B) {\left( \begin{aligned}
& {\left\langle \Ad_{g^{-1}} L_1, g^* \mu {\left( (I-B)^{-1} \bar{E}_1 \right)} \right\rangle} \\
& + {\left\langle \Ad_{g^{-1}} L_2, g^* \mu {\left( (I-B)^{-1} \bar{E}_2 \right)} \right\rangle}
\end{aligned} \right)}.
\end{aligned} \label{E:Q}
\end{align}

Now, let us consider the integral over \( T_{\rho} \). Let \( \bar{E} \) be the positively oriented unit vector field on the infinite cylinder \( T = \bigcup_{r \in (-\varepsilon, \infty)} \partial R_r \) that is tangent to each leaf curve \( \partial R_r \) (which is diffeomorphic to the circle), and let \( \bar{\varepsilon} \) be its dual. For each \( r \in (-\varepsilon, \infty) \), let \( a_r \colon \partial R \to (0, \infty) \) be the positive smooth function given by \( a_r {\left( u_r \right)}_* {\left. \bar{E} \right|}_{\partial R_r} = {\left. \bar{E} \right|}_{\partial R} \), i.e., \( a_r u_r^* {\left. \bar{\varepsilon} \right|}_{\partial R} = {\left. \bar{\varepsilon} \right|}_{\partial R_r} \). Then we have
\begin{align*}
& \int_{T_{\rho}} {\left\langle \Ad_{g^{-1}} \bar{s}^* \bar{\omega} \wedge g^* \mu \right\rangle} = \int_{T_{\rho}} {\left\langle \Ad_{g^{-1}} \bar{s}^* \bar{\omega} \wedge g^* \mu \right\rangle} {\left( \bar{E}, \frac{\partial}{\partial r} \right)} \, \bar{\varepsilon} \wedge \mathrm{d} r \\
& \qquad = \int_0^{\rho} {\left( \int_{\partial R_r} {\left\langle \Ad_{g^{-1}} \bar{s}^* \bar{\omega} \wedge g^* \mu \right\rangle} {\left( \bar{E}, \frac{\partial}{\partial r} \right)} \, {\left. \bar{\varepsilon} \right|}_{\partial R_r} \right)} \, \mathrm{d} r \\
& \qquad = \int_0^{\rho} {\left( \int_{\partial R} {\left\langle \Ad_{g^{-1}} \bar{s}^* \bar{\omega} \wedge g^* \mu \right\rangle} {\left( \bar{E}, \frac{\partial}{\partial r} \right)} \circ u_r^{-1} \, a_r \, {\left. \bar{\varepsilon} \right|}_{\partial R} \right)} \, \mathrm{d} r \\
& \qquad = \int_{\partial R} {\left( \int_0^{\rho} {\left\langle \Ad_{g^{-1}} \bar{s}^* \bar{\omega} \wedge g^* \mu \right\rangle} {\left( \bar{E}, \frac{\partial}{\partial r} \right)} \circ u_r^{-1} \, a_r \, \mathrm{d} r \right)} \, {\left. \bar{\varepsilon} \right|}_{\partial R}.
\end{align*}

Fix \( x \in \partial R \), and let \( x_r = u_r^{-1}(x) \in \partial R_r \) for each \( r \in [0, \infty) \). Note that \( g_* \frac{\partial}{\partial r} = 0 \) by Lemma~\ref{L:constant-frame-gauge-transformation}. For each \( r \in [0, \infty) \), it follows that
\begin{align*}
& {\left. {\left\langle \Ad_{g^{-1}} \bar{s}^* \bar{\omega} \wedge g^* \mu \right\rangle} {\left( \bar{E}, \frac{\partial}{\partial r} \right)} \right|}_{x_r} = - {\left\langle \Ad_{g {\left( x_r \right)}^{-1}} \bar{s}^* \bar{\omega} {\left( {\left. \frac{\partial}{\partial r} \right|}_{x_r} \right)}, g^* \mu {\left( {\left. \bar{E} \right|}_{x_r} \right)} \right\rangle} \\
& \qquad = - {\left\langle \Ad_{g(x)^{-1}} \bar{s}^* \bar{\omega} {\left( {\left. \frac{\partial}{\partial r} \right|}_{x_r} \right)}, u_r^* {\left. g \right|}_{\partial R}^* \mu {\left( {\left. \bar{E} \right|}_{x_r} \right)} \right\rangle} \\
& \qquad = \frac{1}{a_r(x)} \sum_{\mathrm{cyc}} {\left. {\left( \bar{s}^* \bar{\omega} \right)}^i_j {\left( \frac{\partial}{\partial r} \right)} \right|}_{x_r} {\left\langle \Ad_{g(x)^{-1}} L_k, {\left. g \right|}_{\partial R}^* \mu {\left( {\left. \bar{E} \right|}_x \right)} \right\rangle}.
\end{align*}
Here, the term \( (i,j) = (1,2) \) was already computed in Subsection~\ref{SS:I}, and the other two terms \( (i,j) \in \{ (2,3), (3,1) \} \) vanish by the observation~\eqref{E:connection-vanish} in Appendix~\ref{AA:CS-normal-frame}. Lemma~\ref{L:limit} implies that
\begin{align}
\int_{T_{\rho}} {\left\langle \Ad_{g^{-1}} \bar{s}^* \bar{\omega} \wedge g^* \mu \right\rangle} & = \rho \int_{\partial R} P {\left\langle \Ad_{g^{-1}} L_3, {\left. g \right|}_{\partial R}^* \mu \right\rangle} + (\textnormal{convergent terms})
\end{align}
as \( \rho \to \infty \), where \( P \colon R \to \mathbb{R} \) is the smooth function defined in~\eqref{E:P}. Consequently,
\begin{align}
\int_{[R, R_{\rho}]} (\textnormal{II}) & = \boldsymbol{e}^{\rho} \int_R Q \, \mathrm{d} a + \rho \int_{\partial R} P {\left\langle \Ad_{g^{-1}} L_3, {\left. g \right|}_{\partial R}^* \mu \right\rangle} + (\textnormal{convergent terms}) \label{E:II}
\end{align}
as \( \rho \to \infty \). We combine~\eqref{E:I} and~\eqref{E:II} to summarize the asymptotic behavior of~\eqref{E:terms} as follows. As \( \rho \to \infty \),
\begin{align}
\begin{aligned}
\int_{[R, R_{\rho}]} s^* \cs {\left( \bar{\omega} \right)} & = \begin{aligned}[t]
& \boldsymbol{e}^{\rho} \int_R Q \, \mathrm{d} a + \frac{\rho}{4 \boldsymbol{\pi}^2} \int_R P K \, \mathrm{d} a + \rho \int_{\partial R} P {\left\langle \Ad_{g^{-1}} L_3, {\left. g \right|}_{\partial R}^* \mu \right\rangle} \\
& + (\textnormal{convergent terms}),
\end{aligned}
\end{aligned} \label{E:PQ}
\end{align}
where \( P, Q \colon R \to \mathbb{R} \) are the smooth functions defined in~\eqref{E:P} and~\eqref{E:Q}.

\subsection{Description from infinity} \label{SS:infinity}

In this subsection, we simplify the smooth functions \( P, Q \colon R \to \mathbb{R} \) in~\eqref{E:PQ} by describing them ``from infinity'': We will show that \( P = 0 \) identically, and find another nice expression of \( Q \) that is independent of the choice of the smooth frame \( \bar{s} \). We begin with the notion of \emph{metric at infinity} introduced in Section~5 of~\cite{MR2386723}.

\begin{pdef} \label{D:metric-at-infinity}
Let \( \Sigma \) be an oriented convex smooth surface with metric \( \I_{\Sigma} \) embedded in \( M \). The \emph{metric at infinity} is the metric \( \I_{\Sigma}^{\infty} \) on \( \Sigma \) defined by
\begin{align}
\I_{\Sigma}^{\infty} {\left( X, Y \right)} & = \frac{1}{2} \I_{\Sigma} {\left( {\left( I - B_{\Sigma} \right)} X, {\left( I - B_{\Sigma} \right)} Y \right)} \qquad (X, Y \in \Gamma {\left( T \Sigma \right)}),
\end{align}
where \( I \) is the identity operator and \( B_{\Sigma} \) is the Weingarten map of \( \Sigma \). Since \( \Sigma \) is convex, \( \I_{\Sigma}^{\infty} \) is a well-defined metric on \( \Sigma \). The surface \( \Sigma \) with the metric \( \I_{\Sigma}^{\infty} \) is denoted by \( \Sigma^{\infty} \) to distinguish it from the original surface \( \Sigma \) with the metric \( \I_{\Sigma} \).
\end{pdef}

\begin{pwarn}
The sign of \( B_{\Sigma} \) in Definition~\ref{D:metric-at-infinity} is opposite to that in~\cite{MR2386723}. This is due to the different choice of orientation. Both definitions are exactly the same.
\end{pwarn}

For each \( r \in (-\varepsilon, \infty) \), let \( v_r \colon S_r^{\infty} \to S_r \) be the orientation-preserving diffeomorphism given by the set-identity map, and let \( u_r^{\infty} = v^{-1} \circ u_r \circ v_r \colon S_r^{\infty} \to S^{\infty} \) so that the following diagram commutes.
\begin{align}
\begin{tikzcd}[ampersand replacement=\&]
S_r \ar[d, "u_r"'] \& S_r^{\infty} \ar[l, "v_r"'] \ar[d, "u_r^{\infty}"] \\
S \& S^{\infty} \ar[l, "v"']
\end{tikzcd} \label{E:projection-infinity}
\end{align}
The surfaces \( S_r^{\infty} \) are conformally equivalent to each other in the following sense.

\begin{pprop} \label{P:conformal-equivalence}
For each \( r \in (-\varepsilon, \infty) \), we have
\begin{align}
\I_r^{\infty} & = \boldsymbol{e}^{2r} {\left( u_r^{\infty} \right)}^* \I^{\infty}.
\end{align}
Moreover, the pullback of each\/ \( \I_r^{\infty} \) via the projection \( \partial_{\infty} M = S_{\infty} \to S_r \) belongs to the conformal class of the conformal boundary \( \partial_{\infty} M \) (for each component).
\end{pprop}
\begin{proof}
For the former, it suffices to check that \( A_r {\left( u_r \right)}_* {\left( I - B_r \right)} = \boldsymbol{e}^r (I-B) {\left( u_r \right)}_* \), which easily follows from Lemma~\ref{L:foliation}. The latter is mentioned in Section~5 of~\cite{MR2386723}. See Appendix~\ref{AA:conformal-equivalence} of the present paper for its proof.
\end{proof}

Transferring data to the conformal boundary \( S_{\infty} = \partial_{\infty} M \) by identifying it with \( S^{\infty} \) as in~\cite{MR2386723} has a difficulty in our case due to a dimensional issue: It indeed requires us to find an intrinsic way to handle \( 3 \)-frames along the surface \( S_{\infty} \). An alternative way is to consider instead \emph{foliation at infinity}, which is a \( 3 \)-dimensional object able to handle \( 3 \)-frames. The construction is as follows. Replacing each leaf metric on the equidistant foliation \( {\left\{ S_r \right\}}_{r \in (-\varepsilon, \infty)} \) with the metric at infinity produces another foliation \( {\left\{ S_r^{\infty} \right\}}_{r \in (-\varepsilon, \infty)} \). It has the same underlying leaf surfaces but with different metrics. Let \( W = (S_{-\varepsilon}^{\infty}, S_{\infty}^{\infty}) \) be the noncompact smooth manifold formed by the equidistant foliation \( {\left\{ S_r^{\infty} \right\}}_{r \in (-\varepsilon, \infty)} \). By Proposition~\ref{P:conformal-equivalence}, the smooth manifold \( W = (S_{-\varepsilon}^{\infty}, S_{\infty}^{\infty}) \) with the metric \( \mathrm{d} r^2 \oplus \I_r^{\infty} \) is isometric to the warped product manifold \( (-\varepsilon, \infty) \tilde{\times} S^{\infty} \) with the metric \( \mathrm{d} r^2 \oplus \boldsymbol{e}^{2r} \I^{\infty} \), via the isometry \( (-\varepsilon, \infty) \tilde{\times} S^{\infty} \isoto W \) defined by \( (r,y) \mapsto {\left( u_r^{\infty} \right)}^{-1} (y) \).

Let \( FW \surjto W \) be the oriented orthonormal frame bundle. For convenience, let
\begin{align}
V_r & = \frac{1}{\sqrt{2}} {\left( I - B_r \right)}
\end{align}
for each \( r \in (-\varepsilon, \infty) \) so that \( \I_r^{\infty} = \I_r {\left( V_r {\left( v_r \right)}_* {}\cdot{}, V_r {\left( v_r \right)}_* {}\cdot{} \right)} \) by Definition~\ref{D:metric-at-infinity}. Each \( V_r \) is a linear automorphism on \( T_{x_r} S_r \) at each \( x_r \in S_r \).

\begin{plem}
For each \( r \in (-\varepsilon, \infty) \), the map \( V_r \circ {\left( v_r \right)}_* \) induces an isomorphism \( V_r {\left( v_r \right)}_* \colon {\left. FW \right|}_{S_r^{\infty}} \to {\left. FM \right|}_{S_r} \) between smooth principal\/ \( \mathrm{SO}(3) \)-bundles.
\end{plem}
\begin{proof}
One can imitate the proof of Lemma~\ref{L:Ar}.
\end{proof}

\noindent The following diagram commutes for each \( r \in (-\varepsilon, \infty) \) (cf.~\eqref{E:projection-infinity}).
\begin{align}
\begin{tikzcd}[ampersand replacement=\&]
{\left. FM \right|}_{S_r} \ar[d, "A_r {\left( u_r \right)}_*"'] \& \& {\left. FW \right|}_{S_r^{\infty}} \ar[ll, "V_r {\left( v_r \right)}_*"'] \ar[d, "\boldsymbol{e}^r {\left( u_r^{\infty} \right)}_*"] \\
{\left. FM \right|}_S \& \& {\left. FW \right|}_{S^{\infty}} \ar[ll, "V v_*"']
\end{tikzcd}
\end{align}
Moreover, we have an isomorphism in the left commutative diagram of~\eqref{E:frame-at-infinity} between smooth principal \( \mathrm{SO}(3) \)-bundles (which is not necessarily an isometry). The \emph{frame at infinity} denoted by \( s^{\infty} = {\left( E_1^{\infty}, E_2^{\infty}, E_3^{\infty} \right)} \colon W \injto FW \) is the pullback of the smooth frame \( s = {\left( E_1, E_2, E_3 \right)} \colon M \injto FM \) via this isomorphism, which is given by completing the right commutative diagram of~\eqref{E:frame-at-infinity}.
\begin{align}
\begin{tikzcd}[ampersand replacement=\&]
{\left. FM \right|}_{(S_{-\varepsilon}, S_{\infty})} \ar[d, two heads] \& \& FW \ar[ll, "V_{\bullet} {\left( v_{\bullet} \right)}_*"'] \ar[d, two heads] \\
(S_{-\varepsilon}, S_{\infty}) \& \& W \ar[ll, "v_{\bullet}"']
\end{tikzcd} \qquad \begin{tikzcd}[ampersand replacement=\&]
{\left. FM \right|}_{(S_{-\varepsilon}, S_{\infty})} \& \& FW \ar[ll, "V_{\bullet} {\left( v_{\bullet} \right)}_*"'] \\
(S_{-\varepsilon}, S_{\infty}) \ar[u, hook, "s"] \& \& W \ar[ll, "v_{\bullet}"'] \ar[u, hook, dashed, "s^{\infty}"']
\end{tikzcd} \label{E:frame-at-infinity}
\end{align}
Since the smooth frame \( s \colon M \injto FM \) is constant along the foliation \( {\left\{ S_r \right\}}_{r \in (-\varepsilon, \infty)} \), it is immediate to see that the smooth frame \( s^{\infty} \colon W \injto FW \) at infinity makes the left diagram of~\eqref{E:constant-frame-at-infinity} commute for each \( r \in (-\varepsilon, \infty) \).
\begin{align}
\begin{tikzcd}[ampersand replacement=\&]
S_r^{\infty} \ar[rr, hook, "{\left. s^{\infty} \right|}_{S_r^{\infty}}"] \ar[d, "u_r^{\infty}"'] \& \& {\left. FW \right|}_{S_r^{\infty}} \ar[d, "\boldsymbol{e}^r {\left( u_r^{\infty} \right)}_*"] \\
S^{\infty} \ar[rr, hook, "{\left. s^{\infty} \right|}_{S^{\infty}}"] \& \& {\left. FW \right|}_{S^{\infty}}
\end{tikzcd} \qquad \begin{tikzcd}[ampersand replacement=\&]
R_r^{\infty} \ar[rr, hook, "{\left. \bar{s}^{\infty} \right|}_{R_r^{\infty}}"] \ar[d, "u_r^{\infty}"'] \& \& {\left. FW \right|}_{R_r^{\infty}} \ar[d, "\boldsymbol{e}^r {\left( u_r^{\infty} \right)}_*"] \\
R^{\infty} \ar[rr, hook, "{\left. \bar{s}^{\infty} \right|}_{R^{\infty}}"] \& \& {\left. FW \right|}_{R^{\infty}}
\end{tikzcd} \label{E:constant-frame-at-infinity}
\end{align}
In a similar manner, considering the restriction of the same pullback, we also have the smooth frame \( \bar{s}^{\infty} = {\left( \bar{E}_1^{\infty}, \bar{E}_2^{\infty}, \bar{E}_3^{\infty} \right)} \colon (R_{-\varepsilon}^{\infty}, R_{\infty}^{\infty}) \injto {\left. FW \right|}_{(R_{-\varepsilon}^{\infty}, R_{\infty}^{\infty})} \) at infinity corresponding to the smooth frame \( \bar{s} = {\left( \bar{E}_1, \bar{E}_2, \bar{E}_3 \right)} \colon (R_{-\varepsilon}, R_{\infty}) \injto {\left. FM \right|}_{(R_{-\varepsilon}, R_{\infty})} \), which makes the right diagram of~\eqref{E:constant-frame-at-infinity} commute. Here, \( \bar{E}_3 = \frac{\partial}{\partial r} \) and \( \bar{E}_3^{\infty} = \frac{\partial}{\partial r} \).

Consider the smooth function \( P \colon R \to \mathbb{R} \) in~\eqref{E:PQ}. Let \( \bar{\omega}_W \in \Omega^1 {\left( FW, \mathfrak{so}(3) \right)} \) be the Levi-Civita connection on \( FW \surjto W \).

\begin{plem} \label{L:P}
For each \( i \in \{ 1, 2 \} \),\/ \( \bar{\del}^W_{\frac{\partial}{\partial r}} \bar{E}_i^{\infty} = 0 \) on \( R^{\infty} \).
\end{plem}
\begin{proof}
Consider the warped product model \( (-\varepsilon, \infty) \tilde{\times} S^{\infty} \) for \( W \). Fix \( i \in \{ 1, 2 \} \). Fix a point \( y \) in the interior of \( R^{\infty} \) and consider locally near \( y \). Let \( {\left( y^1, y^2 \right)} \colon U \to \mathbb{R}^2 \) be a normal local chart on \( R^{\infty} \) at \( y \). Then \( {\left( r, y^1, y^2 \right)} \colon (-\varepsilon, \infty) \times U \to \mathbb{R}^3 \) is a local chart on \( W \) at \( y \). Say
\begin{align*}
\bar{E}_i^{\infty} & = f^1 \frac{\partial}{\partial y^1} + f^2 \frac{\partial}{\partial y^2} \quad \text{on} \ (-\varepsilon, \infty) \times U.
\end{align*}
For each \( r \in (-\varepsilon, \infty) \), since \( {\left. \bar{E}_i^{\infty} \right|}_{\{ 0 \} \times U} = \boldsymbol{e}^r {\left( u_r^{\infty} \right)}_* {\left. \bar{E}_i^{\infty} \right|}_{\{ r \} \times U} \), each \( f^j \) satisfies \( f^j {\left( r, {}\cdot{} \right)} = \boldsymbol{e}^{-r} f^j {\left( 0, {}\cdot{} \right)} \). Also, Koszul's formula yields
\begin{align*}
{\left. \bar{\del}^W_{\frac{\partial}{\partial r}} \frac{\partial}{\partial y^j} \right|}_y & = {\left. \frac{\partial}{\partial y^j} \right|}_y
\end{align*}
for each \( j \in \{ 1, 2 \} \). These two imply that \( \bar{\del}^W_{\frac{\partial}{\partial r}} \bar{E}_i^{\infty} = 0 \) at \( y \) by the Leibniz rule.
\end{proof}

Fix \( x \in R \) and let \( y = v^{-1}(x) \in R^{\infty} \). On the one hand, Lemma~\ref{L:P} gives
\begin{align*}
{\left. {\left( (\bar{s}^{\infty})^* \bar{\omega}_W \right)}^1_2 {\left( \frac{\partial}{\partial r} \right)} \right|}_y & = \I_W {\left( {\left. \bar{E}_1^{\infty} \right|}_y, {\left. \bar{\del}^W_{\frac{\partial}{\partial r}} \bar{E}_2^{\infty} \right|}_y \right)} = \I_W {\left( {\left. \bar{E}_1^{\infty} \right|}_y, 0 \right)} = 0.
\end{align*}
On the other hand, Koszul's formula gives
\begin{align*}
& {\left. {\left( (\bar{s}^{\infty})^* \bar{\omega}_W \right)}^1_2 {\left( \frac{\partial}{\partial r} \right)} \right|}_y = \I_W {\left( {\left. \bar{E}_1^{\infty} \right|}_y, {\left. \bar{\del}^W_{\frac{\partial}{\partial r}} \bar{E}_2^{\infty} \right|}_y \right)} \\
& \qquad = \frac{1}{2} {\left( \I_W {\left( {\left. \bar{E}_2^{\infty} \right|}_y, {\left. {\left[ \bar{E}_1^{\infty}, \frac{\partial}{\partial r} \right]} \right|}_y \right)} - \I_W {\left( {\left. \bar{E}_1^{\infty} \right|}_y, {\left. {\left[ \bar{E}_2^{\infty}, \frac{\partial}{\partial r} \right]} \right|}_y \right)} \right)} \\
& \qquad = \frac{1}{2} {\left( \I^{\infty} {\left( {\left. \bar{E}_2^{\infty} \right|}_y, {\left. {\left[ \bar{E}_1^{\infty}, \frac{\partial}{\partial r} \right]} \right|}_y^{\top} \right)} - \I^{\infty} {\left( {\left. \bar{E}_1^{\infty} \right|}_y, {\left. {\left[ \bar{E}_2^{\infty}, \frac{\partial}{\partial r} \right]} \right|}_y^{\top} \right)} \right)},
\end{align*}
where the superscript \( \top \) denotes the projection of vector onto the leaf surface. Here, for each \( (i,j) \in \{ (1,2), (2,1) \} \), Lemma~\ref{L:Lie-bracket-infinity} gives
\begin{align*}
& \I^{\infty} {\left( {\left. \bar{E}_i^{\infty} \right|}_y, {\left. {\left[ \bar{E}_j^{\infty}, \frac{\partial}{\partial r} \right]} \right|}_y^{\top} \right)} = \I {\left( {\left. \bar{E}_i \right|}_x, V v_* {\left. {\left[ \bar{E}_j^{\infty}, \frac{\partial}{\partial r} \right]} \right|}_y^{\top} \right)} \\
& \qquad = \I {\left( {\left. \bar{E}_i \right|}_x, {\left( V|_x^{-1} \right)}^{\ell}_j \, V {\left. {\left[ \bar{E}_{\ell}, \frac{\partial}{\partial r} \right]} \right|}_x^{\top} \right)} + {\left( {\left. (I+B) \right|}_x \right)}^i_j \\
& \qquad = {\left( V|_x^{-1} \right)}^k_i {\left( V|_x^{-1} \right)}^{\ell}_j \, \I {\left( V {\left. \bar{E}_k \right|}_x, V {\left. {\left[ \bar{E}_{\ell}, \frac{\partial}{\partial r} \right]} \right|}_x^{\top} \right)} + {\left( {\left. (I+B) \right|}_x \right)}^i_j,
\end{align*}
where the Einstein summation convention is assumed. We obtain
\begin{align*}
& {\left. {\left( (\bar{s}^{\infty})^* \bar{\omega}_W \right)}^1_2 {\left( \frac{\partial}{\partial r} \right)} \right|}_y = \frac{1}{2} {\left( \begin{aligned}
& {\left( V|_x^{-1} \right)}^k_2 {\left( V|_x^{-1} \right)}^{\ell}_1 \\
& - {\left( V|_x^{-1} \right)}^k_1 {\left( V|_x^{-1} \right)}^{\ell}_2
\end{aligned} \right)} \I {\left( V {\left. \bar{E}_k \right|}_x, V {\left. {\left[ \bar{E}_{\ell}, \frac{\partial}{\partial r} \right]} \right|}_x^{\top} \right)} \\
& \qquad = \frac{1}{2 \det {\left. V \right|}_x} {\left( \I {\left( V {\left. \bar{E}_2 \right|}_x, V {\left. {\left[ \bar{E}_1, \frac{\partial}{\partial r} \right]} \right|}_x^{\top} \right)} - \I {\left( V {\left. \bar{E}_1 \right|}_x, V {\left. {\left[ \bar{E}_2, \frac{\partial}{\partial r} \right]} \right|}_x^{\top} \right)} \right)}.
\end{align*}
This is exactly \( {\left. P \right|}_x \) in~\eqref{E:P}. Therefore, we find that \( P = 0 \) in~\eqref{E:PQ}.

Consider the smooth function \( Q \colon R \to \mathbb{R} \) in~\eqref{E:PQ}. Let \( \mathrm{d} a^{\infty} \) be the area form on \( S^{\infty} \), and let \( g^{\infty} \colon (R_{-\varepsilon}^{\infty}, R_{\infty}^{\infty}) \to \mathrm{SO}(3) \) be the smooth map given by \( s^{\infty} = \bar{s}^{\infty} \cdot g^{\infty} \). This map coincides with the smooth map \( g \colon (R_{-\varepsilon}, R_{\infty}) \to \mathrm{SO}(3) \) via the map \( v \colon S^{\infty} \to S \) in the following sense.

\begin{plem} \label{L:Q}
\( g^{\infty} = g \circ v \) on \( R^{\infty} \).
\end{plem}
\begin{proof}
We imitate the proof of Lemma~\ref{L:constant-frame-gauge-transformation}. Fix \( y \in R^{\infty} \) and let \( x = v(y) \in R \). It suffices to show that \( {\left. g^{\infty} \right|}_y = {\left. g \right|}_x \). We have \( {\left. s^{\infty} \right|}_y = {\left. \bar{s}^{\infty} \right|}_y \cdot {\left. g^{\infty} \right|}_y \) and
\begin{align*}
{\left. s^{\infty} \right|}_y & = V v_* {\left( {\left. s \right|}_x \right)} = V v_* {\left( {\left. \bar{s} \right|}_x \cdot {\left. g \right|}_x \right)} = V v_* {\left( {\left. \bar{s} \right|}_x \right)} \cdot {\left. g \right|}_x = {\left. \bar{s}^{\infty} \right|}_y \cdot {\left. g \right|}_x.
\end{align*}
This completes the proof.
\end{proof}

Fix \( x \in R \) and let \( y = v^{-1}(x) \in R^{\infty} \). By~\eqref{E:Q} and Lemma~\ref{L:Q}, we have
\begin{align*}
{\left. v^* {\left( Q \, \mathrm{d} a \right)} \right|}_y & = \frac{\det {\left. V \right|}_x}{\sqrt{2}} \sum_{i = 1}^2 {\left\langle \Ad_{g(x)^{-1}} L_i, g^* \mu {\left( V^{-1} {\left. \bar{E}_i \right|}_x \right)} \right\rangle} {\left. {\left( v^* \mathrm{d} a \right)} \right|}_y \\
& = \frac{1}{\sqrt{2}} \sum_{i = 1}^2 {\left\langle \Ad_{g^{\infty}(y)^{-1}} L_i, {\left( g^{\infty} \right)}^* \mu {\left( {\left. \bar{E}_i^{\infty} \right|}_y \right)} \right\rangle} \, {\left. \mathrm{d} a^{\infty} \right|}_y.
\end{align*}
Here, by \( \Ad \)-invariance and~\eqref{E:Ad-invariant-form}, we have
\begin{align*}
& \frac{1}{\sqrt{2}} \sum_{i = 1}^2 {\left\langle \Ad_{{\left( g^{\infty} \right)}^{-1}} L_i, {\left( g^{\infty} \right)}^* \mu {\left( \bar{E}_i^{\infty} \right)} \right\rangle} = - \frac{1}{\sqrt{2}} \sum_{i = 1}^2 {\left\langle L_i, g^{\infty} \, \mathrm{d} {\left( g^{\infty} \right)}^{-1} {\left( \bar{E}_i^{\infty} \right)} \right\rangle} \\
& \qquad = \frac{1}{4 \sqrt{2} \boldsymbol{\pi}^2} \sum_{j = 1}^3 {\left( {\left( g^{\infty} \right)}^2_j \, \mathrm{d} {\left( g^{\infty} \right)}^3_j {\left( \bar{E}_1^{\infty} \right)} + {\left( g^{\infty} \right)}^3_j \, \mathrm{d} {\left( g^{\infty} \right)}^1_j {\left( \bar{E}_2^{\infty} \right)} \right)} \\
& \qquad = \frac{1}{4 \sqrt{2} \boldsymbol{\pi}^2} \sum_{j = 1}^3 {\left( \mathrm{d} {\left( g^{\infty} \right)}^1_j {\left( \bar{E}_2^{\infty} \right)} - \mathrm{d} {\left( g^{\infty} \right)}^2_j {\left( \bar{E}_1^{\infty} \right)} \right)} {\left( g^{\infty} \right)}^3_j \\
& \qquad = \frac{1}{4 \sqrt{2} \boldsymbol{\pi}^2} \sum_{j = 1}^3 {\left( \bar{E}_2^{\infty} {\left( g^{\infty} \right)}^1_j - \bar{E}_1^{\infty} {\left( g^{\infty} \right)}^2_j \right)} \, \I_W {\left( \frac{\partial}{\partial r}, E_j^{\infty} \right)}.
\end{align*}
This can be further simplified by the following formula.

\begin{plem}
We have
\begin{align*}
\sum_{j = 1}^3 {\left( \bar{E}_2^{\infty} {\left( g^{\infty} \right)}^1_j - \bar{E}_1^{\infty} {\left( g^{\infty} \right)}^2_j \right)} E_j^{\infty} & = {\left[ \bar{E}_2^{\infty}, \bar{E}_1^{\infty} \right]} + \sum_{\mathrm{cyc}} \I_W {\left( \frac{\partial}{\partial r}, E_i^{\infty} \right)} {\left[ E_j^{\infty}, E_k^{\infty} \right]}.
\end{align*}
\end{plem}
\begin{proof}
Only in this proof, let \( \bar{g} = {\left( g^{\infty} \right)}^{-1} \) for convenience. We have
\begin{align*}
{\left[ \bar{E}_2^{\infty}, \bar{E}_1^{\infty} \right]} & = \sum_{i, j = 1}^3 \bar{g}^i_2 \bar{g}^j_1 {\left[ E_i^{\infty}, E_j^{\infty} \right]} + \sum_{j = 1}^3 {\left( \bar{E}_2^{\infty} {\big( \bar{g}^j_1 \big)} - \bar{E}_1^{\infty} {\big( \bar{g}^j_2 \big)} \right)} E_j^{\infty}.
\end{align*}
Here, since \( \bar{g} \) is \( \mathrm{SO}(3) \)-valued,
\begin{align*}
\sum_{i, j = 1}^3 \bar{g}^i_2 \bar{g}^j_1 {\left[ E_i^{\infty}, E_j^{\infty} \right]} & = \sum_{\substack{i, j, k \\ \mathrm{cyc}}} {\left( \bar{g}^i_2 \bar{g}^j_1 - \bar{g}^i_1 \bar{g}^j_2 \right)} {\left[ E_i^{\infty}, E_j^{\infty} \right]} = \sum_{\mathrm{cyc}} - \bar{g}^k_3 {\left[ E_i^{\infty}, E_j^{\infty} \right]}.
\end{align*}
This completes the proof.
\end{proof}

\noindent Combining these observations yield
\begin{align}
\begin{aligned}
\int_R Q \, \mathrm{d} a & = \int_{R^{\infty}} v^* {\left( Q \, \mathrm{d} a \right)} \\
& = \frac{1}{4 \sqrt{2} \boldsymbol{\pi}^2} \int_{R^{\infty}} \sum_{\mathrm{cyc}} \I_W {\left( \frac{\partial}{\partial r}, E_i^{\infty} \right)} \, \I_W {\left( \frac{\partial}{\partial r}, {\left[ E_j^{\infty}, E_k^{\infty} \right]} \right)} \, \mathrm{d} a^{\infty}.
\end{aligned}
\end{align}
Note that the last expression does not involve any dependency on the choice of the smooth frame \( \bar{s} \). Also, it is fully determined by the data from infinity.

Consequently, we find that~\eqref{E:PQ} becomes
\begin{align*}
\int_{[R, R_{\rho}]} s^* \cs {\left( \bar{\omega} \right)} & = \begin{aligned}[t]
& \frac{\boldsymbol{e}^{\rho}}{4 \sqrt{2} \boldsymbol{\pi}^2} \int_{R^{\infty}} \sum_{\mathrm{cyc}} \I_W {\left( \frac{\partial}{\partial r}, E_i^{\infty} \right)} \, \I_W {\left( \frac{\partial}{\partial r}, {\left[ E_j^{\infty}, E_k^{\infty} \right]} \right)} \, \mathrm{d} a^{\infty} \\
& + (\textnormal{convergent terms})
\end{aligned}
\end{align*}
as \( \rho \to \infty \), and we take the summation over \( R \) to conclude that the asymptotics of~\eqref{E:asymptotics} is given as follows. As \( \rho \to \infty \),
\begin{align}
\begin{aligned}
\int_{[S, S_{\rho}]} s^* \cs {\left( \bar{\omega} \right)} & = \begin{aligned}[t]
& \frac{\boldsymbol{e}^{\rho}}{4 \sqrt{2} \boldsymbol{\pi}^2} \int_{S^{\infty}} \sum_{\mathrm{cyc}} \I_W {\left( \frac{\partial}{\partial r}, E_i^{\infty} \right)} \, \I_W {\left( \frac{\partial}{\partial r}, {\left[ E_j^{\infty}, E_k^{\infty} \right]} \right)} \, \mathrm{d} a^{\infty} \\
& + (\textnormal{convergent terms}).
\end{aligned}
\end{aligned} \label{E:asymptotics-result}
\end{align}
We reached the desired asymptotics, but still the geometric meaning of the divergent term is vague. It will be clarified in Section~\ref{S:connections}.

\section{Connections and hypersurfaces} \label{S:connections}

In order to provide the geometric meaning of the divergent term in~\eqref{E:asymptotics-result}, we need to study hypersurfaces in Weitzenb{\"o}ck geometry. For this purpose, Section~\ref{S:connections} consists of two parts: Weitzenb{\"o}ck geometry in Subsection~\ref{SS:Weitzenbock}, and its hypersurfaces in Subsection~\ref{SS:hypersurface}. Elucidating the geometric meaning of the exponentially divergent term in~\eqref{E:asymptotics-result}, we will complete the proof of Theorem~\ref{M:A} at the end of this section.

\subsection{Weitzenb{\"o}ck connection} \label{SS:Weitzenbock}

A \emph{connection of the parallelization}, also known as a \emph{Weitzenb{\"o}ck connection}, is a nonholonomic flat connection determined by a fixed global smooth frame. It has \emph{nonzero torsion} and \emph{zero curvature}, while the Levi-Civita connection has \emph{zero torsion} and \emph{nonzero curvature}. The following is a modern definition given in Section~5.7 of~\cite{MR0615912}.

\begin{pdef} \label{D:Weitzenbock-affine}
Let \( M \) be a parallelizable smooth \( n \)-manifold with or without boundary. The \emph{Weitzenb{\"o}ck connection} on \( M \) induced by a global smooth frame \( s = {\left( E_1, \dotsc, E_n \right)} \) for \( M \) is the affine connection \( \del^s \colon \Gamma(TM) \times \Gamma(TM) \to \Gamma(TM) \) on \( M \) defined by
\begin{align}
\del^s_X Y & = \sum_{i = 1}^n X {\left( Y^i \right)} E_i,
\end{align}
where \( X, Y \in \Gamma(TM) \) and \( Y = Y^1 E_1 + \dotsb + Y^n E_n \).
\end{pdef}

\noindent It is easy to see that a Weitzenb{\"o}ck connection on a parallelizable Riemannian manifold is metric-compatible if it is induced by a global orthonormal smooth frame. The Weitzenb{\"o}ck connection on \( \mathbb{R}^n \) induced by the standard global smooth frame \( {\big( \frac{\partial}{\partial x^1}, \dotsc, \frac{\partial}{\partial x^n} \big)} \) coincides with the Levi-Civita connection.

From the perspective of principal bundles, a Weitzenb{\"o}ck connection can be understood as the choice of a trivial connection in the following sense.

\begin{pdef}
Let \( G \) be a Lie group. The \emph{trivial principal \( G \)-bundle} over a smooth manifold \( M \) with or without boundary is defined by the projection \( M \times G \surjto M \) with the smooth right group action \( (x, g) \cdot h = (x, gh) \), where \( (x, g) \in M \times G \) and \( h \in G \). The \emph{trivial connection} on this is the pullback of the Maurer-Cartan \( 1 \)-form on \( G \) via the projection \( M \times G \surjto G \).
\end{pdef}

\begin{pdef} \label{D:Weitzenbock-principal}
Let \( G \) be a Lie group. Let \( \pi \colon P \surjto M \) be a smooth principal \( G \)-bundle. Suppose that it has a global smooth section \( \sigma \colon M \injto P \). Then it induces a global trivialization \( \varphi_{\sigma} \colon P \to M \times G \) defined by \( \varphi_{\sigma}^{-1}(x, g) = \sigma(x) \cdot g \).
\begin{align}
\begin{tikzcd}[ampersand replacement=\&]
P \ar[rr, "\varphi_{\sigma}"] \ar[rd, two heads, "\pi"'] \& \& M \times G \ar[ld, two heads] \\
\& M
\end{tikzcd}
\end{align}
The \emph{Weitzenb{\"o}ck connection} on \( \pi \colon P \surjto M \) induced by \( \sigma \) is the principal connection \( \omega^{\sigma} \in \Omega^1 {\left( P, \mathfrak{g} \right)} \) on \( \pi \colon P \surjto M \) defined by the pullback of the trivial connection on the trivial principal \( G \)-bundle \( M \times G \surjto M \) via \( \varphi_{\sigma} \colon P \to M \times G \).
\end{pdef}

\noindent Weitzenb{\"o}ck connections of Definition~\ref{D:Weitzenbock-principal} are flat (i.e., \( \Omega^{\sigma} = 0 \)) as well, by the Maurer-Cartan equation. In fact, it is easy to check that \( \sigma^* \omega^{\sigma} = 0 \).

Definitions~\ref{D:Weitzenbock-affine} and~\ref{D:Weitzenbock-principal} indeed coincide in the following sense.

\begin{pprop}
Let \( M \) be an oriented parallelizable Riemannian \( n \)-manifold with or without boundary. Let \( s \) be a global oriented orthonormal smooth frame for \( M \). Then the Weitzenb{\"o}ck affine connection\/ \( \del^s \) induced by \( s \) and the Weitzenb{\"o}ck principal connection \( \omega^s \) induced by \( s \) coincide.
\end{pprop}
\begin{proof}
It is straightforward to check that \( s^* \omega^s = 0 \). Indeed, the connection coefficients of \( \del^s \) with respect to \( s \) are all zero, since \( \del^s_X E_j = 0 \) for any \( X \in \Gamma(TM) \) and any \( j \in \{ 1, \dotsc, n \} \), where \( s = {\left( E_1, \dotsc, E_n \right)} \). This means that they coincide.
\end{proof}

There are two remarks regarding Weitzenb{\"o}ck connection.

\begin{prmrk} \label{R:flat-connection}
A connection is flat iff it is locally Weitzenb{\"o}ck. The difference between Weitzenb{\"o}ck connections and flat connections is \emph{holonomy}. Weitzenb{\"o}ck connections are globally nonholonomic, while flat connections can be holonomic.
\end{prmrk}

\begin{prmrk}
All Weitzenb{\"o}ck connections are \emph{gauge-equivalent} to each other. In the setup of Definition~\ref{D:Weitzenbock-principal}, suppose that \( \sigma, \sigma' \colon M \injto P \) are two global smooth sections. Then \( \omega^{\sigma'} = \varphi^* \omega^{\sigma} \) by the gauge transformation \( \varphi = \varphi_{\sigma}^{-1} \circ \varphi_{\sigma'} \colon P \isoto P \).
\end{prmrk}

\subsection{Hypersurfaces in Riemann-Cartan geometry} \label{SS:hypersurface}

A \emph{Riemann-Cartan manifold} proposed and studied by Cartan~\cites{MR1509253, MR1509255, MR1509263} is a Riemannian manifold endowed with a metric-compatible connection, which is not necessarily torsion-free. An important example is a parallelizable Riemannian manifold equipped with the Weitzenb{\"o}ck connection induced by a global orthonormal smooth frame. A Riemann-Cartan manifold \( {\left( M, \I_M, \del \right)} \) has a nontrivial \emph{torsion}, which is the smooth \( (1,2) \)-tensor field \( T^{\del} \) on \( M \) defined by
\begin{align}
T^{\del}(X, Y) & = \del_X Y - \del_Y X - [X, Y] \qquad (X, Y \in \Gamma(TM)),
\end{align}
unless \( \del \) is the Levi-Civita connection. Despite this nontrivial torsion, it is possible to naturally extend the framework of hypersurfaces in Riemannian geometry to Riemann-Cartan geometry. We give a brief account.

We start with the second fundamental form in Riemann-Cartan geometry.

\begin{pdef} \label{D:SFF}
Let \( S \) be an oriented smooth hypersurface embedded in an oriented Riemann-Cartan manifold \( {\left( M, \I_M, \del \right)} \) with or without boundary. Say \( S \) is oriented by the unit normal vector field \( N \) along \( S \). The \emph{(scalar) second fundamental form} of \( S \) with respect to \( \del \) is the smooth \( (0,2) \)-tensor field \( \II^{\del} \) on \( S \) defined by
\begin{align}
\II^{\del}(X, Y) & = \I_M {\left( N, \del_X Y \right)} \qquad (X, Y \in \Gamma(TS)),
\end{align}
where \( X \) and \( Y \) are extended arbitrarily smoothly to an open subset of \( M \).
\end{pdef}

\begin{plem}
In the setup of Definition~\ref{D:SFF}, the following two hold.
\begin{enumerate}
\item \( \II^{\del} \) is well defined, i.e., the value of\/ \( \II^{\del}(X, Y) \) at \( x \in S \) depends only on \( {\left. X \right|}_x \) and \( {\left. Y \right|}_x \), where \( X, Y \in \Gamma(TS) \).
\item \( \II^{\del}(X, Y) - \II^{\del}(Y, X) = \I_M {\left( N, T^{\del}(X, Y) \right)} \) for any \( X, Y \in \Gamma(TS) \).
\end{enumerate}
\end{plem}
\begin{proof}
The equality in~(b) holds for any smooth extensions of \( X, Y \in \Gamma(TS) \). Since the value of the torsion \( T^{\del}(X, Y) \) at \( x \in S \) depends only on \( {\left. X \right|}_x \) and \( {\left. Y \right|}_x \), this implies~(a).
\end{proof}

\noindent Note that the second fundamental form fails to be symmetric in general. Indeed, it is symmetric exactly when the following \( 2 \)-form vanishes.

\begin{pdef}[Torsion \( 2 \)-form] \label{D:torsion-form}
In the setup of Definition~\ref{D:SFF}, the \emph{torsion\/ \( 2 \)-form} of \( S \) with respect to \( \del \) is \( \tau^{\del} \in \Omega^2 {\left( S, \mathbb{R} \right)} \) defined by
\begin{align}
\tau^{\del}(X, Y) & = \I_M {\left( N, T^{\del}(X, Y) \right)} \qquad (X, Y \in \Gamma(TS)).
\end{align}
\end{pdef}

The Weingarten map in Riemann-Cartan geometry is given as follows.

\begin{pdef}
In the setup of Definition~\ref{D:SFF}, the \emph{Weingarten map} of \( S \) with respect to \( \del \) is the smooth \( (1,1) \)-tensor field \( B^{\del} \) on \( S \) determined by
\begin{align}
\II^{\del}(X, Y) & = \I_S {\left( B^{\del}(X), Y \right)} \quad \text{for all} \ X, Y \in \Gamma(TS).
\end{align}
The \emph{(extrinsic) Gaussian curvature} and the \emph{mean curvature} of \( S \) with respect to \( \del \) are smooth scalar fields on \( M \) defined by
\begin{align}
K^{\del} & = \det B^{\del} \quad \text{and} \quad H^{\del} = \tr B^{\del}
\end{align}
respectively.
\end{pdef}

\begin{plem} \label{L:Weingarten-equation}
In the setup of Definition~\ref{D:SFF}, we have
\begin{align}
B^{\del}(X) & = - \del_X N
\end{align}
for any \( X \in \Gamma(TS) \) and any smooth extension of \( N \) to an open subset of \( M \).
\end{plem}
\begin{proof}
Fix \( X, Y \in \Gamma(TS) \) and extend them arbitrarily smoothly to an open subset of \( M \). Since \( 2 \I_M {\left( \del_X N, N \right)} = X {\left( \I_M {\left( N, N \right)} \right)} = 0 \) along \( S \), we see that \( \del_X N \in \Gamma(TS) \). Therefore, along \( S \), we have
\begin{align*}
\II^{\del}(X, Y) & = \I_M {\left( N, \del_X Y \right)} = \I_M {\left( - \del_X N, Y \right)} = \I_S {\left( - \del_X N, Y \right)},
\end{align*}
as desired.
\end{proof}

Now, let us consider the Weitzenb{\"o}ck case. We denote quantities by the superscript \( s \) instead of the superscript \( \del^s \) for legibility, e.g., \( H^s \) instead of \( H^{\del^s} \). We present one lemma and two propositions.

\begin{plem} \label{L:Weitzenbock}
Let \( S \) be an oriented smooth hypersurface embedded in an oriented parallelizable Riemannian manifold \( M \) with or without boundary. Let \( s \) and \( s' \) be two global oriented orthonormal smooth frames for \( M \). Suppose \( s = s' \) on \( S \). Then
\begin{align}
\II^s & = \II^{s'}, \qquad \tau^s = \tau^{s'}, \qquad B^s = B^{s'}, \qquad K^s = K^{s'}, \qquad H^s = H^{s'}.
\end{align}
\end{plem}
\begin{proof}
Let \( s = {\left( E_1, \dotsc, E_n \right)} \) and \( s' = {\left( E_1', \dotsc, E_n' \right)} \). Say \( S \) is oriented by the unit normal vector field \( N \) along \( S \). Then \( N = N^1 E_1 + \dotsc + N^n E_n = N^1 E_1' + \dotsc + N^n E_n' \) along \( S \), since \( s = s' \) along \( S \). By Lemma~\ref{L:Weingarten-equation}, we find \( B^s = - \sum_{i = 1}^n \mathrm{d} N^i \otimes E_i = B^{s'} \). All the others follow from this.
\end{proof}

\begin{pprop} \label{P:mean-curvature-torsion}
Let \( S \) be an oriented smooth hypersurface embedded in an oriented parallelizable Riemannian \( n \)-manifold\/ \( {\left( M, \I_M \right)} \) with or without boundary. Let \( \iota \colon S \injto M \) be the inclusion. Let \( s = {\left( E_1, \dotsc, E_n \right)} \) be a global oriented orthonormal smooth frame for \( M \), and let\/ \( {\left( \varepsilon^1, \dotsc, \varepsilon^n \right)} \) be its dual. Say \( S \) is oriented by the unit normal vector field \( N = N^1 E_1 + \dotsb + N^n E_n \) along \( S \). Then
\begin{align}
H^s & = - \sum_{i = 1}^n E_i^{\top} {\left( N^i \right)} \quad \text{and} \quad \tau^s = \sum_{i = 1}^n N^i \iota^* \mathrm{d} \varepsilon^i.
\end{align}
\end{pprop}
\begin{proof}
If \( {\left( \bar{E}_1, \dotsc, \bar{E}_{n-1} \right)} \) is a local oriented orthonormal smooth frame for \( S \),
\begin{align*}
H^s & = \tr B^s = \sum_{j = 1}^{n-1} \I_M {\left( \bar{E}_j, - \sum_{i = 1}^n \bar{E}_j {\left( N^i \right)} E_i \right)} = - \sum_{i = 1}^n \sum_{j = 1}^{n-1} \I_M {\left( E_i, \bar{E}_j \right)} \bar{E}_j {\left( N^i \right)}.
\end{align*}
Here, \( E_i^{\top} = \sum_{j = 1}^{n-1} \I_M {\left( E_i, \bar{E}_j \right)} \bar{E}_j \). The former follows. For all \( i, j \in \{ 1, \dotsc, n \} \),
\begin{align*}
\sum_{k = 1}^n N^k \, \mathrm{d} \varepsilon^k {\left( E_i, E_j \right)} & = - \sum_{k = 1}^n N^k \varepsilon^k {\left( {\left[ E_i, E_j \right]} \right)} = - \I_M {\left( N, {\left[ E_i, E_j \right]} \right)}
\end{align*}
and \( T^s {\left( E_i, E_j \right)} = - {\left[ E_i, E_j \right]} \). The latter follows.
\end{proof}

\begin{pprop} \label{P:comparison}
In the setup of Proposition~\ref{P:mean-curvature-torsion}, let\/ \( \del \) be a metric-compatible connection (e.g., the Levi-Civita connection) on \( M \). Then
\begin{align}
B^s & = B^{\del} + \sum_{i = 1}^n N^i \del E_i \quad \text{and} \quad H^s = H^{\del} - \sum_{i = 1}^n \I_M {\left( N, \del_{E_i} E_i \right)}.
\end{align}
\end{pprop}
\begin{proof}
For the former, observe that
\begin{align*}
\del_X N & = \del_X {\left( \sum_{i = 1}^n N^i E_i \right)} = \sum_{i = 1}^n X {\left( N^i \right)} E_i + \sum_{i = 1}^n N^i \del_X E_i = \del^s_X N + \sum_{i = 1}^n N^i \del_X E_i,
\end{align*}
where \( X \in \Gamma(TS) \). For the latter, first note that
\begin{align*}
& \I_M {\left( N, \sum_{i = 1}^n N^i \del_N E_i \right)} = \sum_{i, j = 1}^n N^j N^i \, \I_M {\left( E_j, \del_N E_i \right)} \\
& \qquad = \sum_{i = 1}^n {\left( N^i \right)}{}^2 \, \I_M {\left( E_i, \del_N E_i \right)} + \sum_{i < j} N^j N^i {\left( \I_M {\left( E_j, \del_N E_i \right)} + \I_M {\left( \del_N E_j, E_i \right)} \right)} \\
& \qquad = \sum_{i = 1}^n {\left( N^i \right)}{}^2 \frac{1}{2} N {\left( \I_M {\left( E_i, E_i \right)} \right)} + \sum_{i < j} N^j N^i N {\left( \I_M {\left( E_j, E_i \right)} \right)} = 0.
\end{align*}
This implies that the trace of \( \sum_{i = 1}^n N^i \del E_i \) as a map \( TS \to TS \) is equal to the trace as a map \( TM \to TM \). Therefore,
\begin{align*}
\tr \! {\left( \sum_{i = 1}^n N^i \del E_i \right)} & = \sum_{j = 1}^n \I_M {\left( E_j, \sum_{i = 1}^n N^i \del_{E_j} E_i \right)} = - \sum_{i, j = 1}^n \I_M {\left( \del_{E_j} E_j, N^i E_i \right)}.
\end{align*}
This completes the proof of the latter.
\end{proof}

Finally, the following provides a satisfactory answer to the geometric meaning of the leading coefficient in~\eqref{E:asymptotics-result}.

\begin{pprop}
Let \( S \) be an oriented smooth surface embedded in an oriented parallelizable Riemannian\/ \( 3 \)-manifold\/ \( {\left( M, \I_M \right)} \) with or without boundary, and let \( s = {\left( E_1, E_2, E_3 \right)} \) be a global oriented orthonormal smooth frame for \( M \). Say \( S \) is oriented by the unit normal vector field \( N = N^1 E_1 + N^2 E_2 + N^3 E_3 \) along \( S \). Then
\begin{align}
\tau^s & = - \sum_{\mathrm{cyc}} N^i \, \I_M {\left( N, {\left[ E_j, E_k \right]} \right)} \, \mathrm{d} a.
\end{align}
\end{pprop}
\begin{proof}
Let \( {\left( \bar{E}_1, \bar{E}_2 \right)} \) be a local oriented orthonormal smooth frame for \( S \). Let \( \bar{g} \) be the \( \mathrm{SO}(3) \)-valued smooth map given by \( {\left( \bar{E}_1, \bar{E}_2, N \right)} = {\left( E_1, E_2, E_3 \right)} \cdot \bar{g} \). Then
\begin{align*}
\tau^s {\left( \bar{E}_1, \bar{E}_2 \right)} & = \sum_{i, j = 1}^3 \bar{g}^i_1 \bar{g}^j_2 \, \I_M {\left( N, T^s {\left( E_i, E_j \right)} \right)} = \sum_{\substack{i, j, k \\ \mathrm{cyc}}} {\big( \bar{g}^i_1 \bar{g}^j_2 - \bar{g}^i_2 \bar{g}^j_1 \big)} \, \I_M {\left( N, T^s {\left( E_i, E_j \right)} \right)}.
\end{align*}
Here, \( \bar{g}^i_1 \bar{g}^j_2 - \bar{g}^j_1 \bar{g}^i_2 = \bar{g}^k_3 = N^k \) and \( T^s {\left( E_i, E_j \right)} = - {\left[ E_i, E_j \right]} \).
\end{proof}

A careful reader might notice that we have reached the convergence of the limit~\eqref{E:A} in Theorem~\ref{M:A} with the assumption that \( s \) is constant along the foliation \( {\left\{ S_r \right\}}_{r \in (-\varepsilon, \infty)} \), not the foliation \( {\left\{ S_r \right\}}_{r \in [0, \infty)} \). Therefore, we need some technical arguments to complete the proof of Theorem~\ref{M:A} as follows. Suppose the setup of Theorem~\ref{M:A}\@. Let \( s' \colon (S_{-\varepsilon}, S_{\infty}) \injto {\left. FM \right|}_{(S_{-\varepsilon}, S_{\infty})} \) be the smooth frame such that \( s = s' \) on \( [S, S_{\infty}) \) and it is constant along the foliation \( {\left\{ S_r \right\}}_{r \in (-\varepsilon, \infty)} \), which can be constructed by defining \( {\left. s' \right|}_{S_r} = {\left( A_r {\left( u_r \right)}_* \right)}^{-1} {\left( {\left. s \right|}_S \right)} \) for all \( r \in (-\varepsilon, \infty) \). Then
\begin{align*}
\int_{C_r} s^* \cs {\left( \bar{\omega} \right)} & = \int_C s^* \cs {\left( \bar{\omega} \right)} + \int_{[S, S_r]} s^* \cs {\left( \bar{\omega} \right)} = \int_C s^* \cs {\left( \bar{\omega} \right)} + \int_{[S, S_r]} {\left( s' \right)}^* \cs {\left( \bar{\omega} \right)} \\
& = \int_C s^* \cs {\left( \bar{\omega} \right)} - \frac{\boldsymbol{e}^r}{4 \sqrt{2} \boldsymbol{\pi}^2} \int_{S^{\infty}} \tau {\left( s'{}^{\infty} \right)} + (\textnormal{convergent terms})
\end{align*}
as \( r \to \infty \). Here, \( \tau {\left( s'{}^\infty \right)} = \tau {\left( s^{\infty} \right)} \) by Lemma~\ref{L:Weitzenbock}, since \( s'{}^{\infty} = s^{\infty} \) on \( S^{\infty} \). This completes the proof of Theorem~\ref{M:A}\@.

\section{Volume and Chern-Simons invariant} \label{S:monodromy}

The goal of this section is to find the asymptotics of the \( \mathrm{PSL}_2 {\left( \mathbb{C} \right)} \)-Chern-Simons invariant. We first give a brief introduction to the subject regarding the \( \mathrm{PSL}_2 {\left( \mathbb{C} \right)} \)-Chern-Simons invariant in Subsection~\ref{SS:monodromy}, and then we engage in the proof of Theorem~\ref{M:B} and Corollary~\ref{M:C} in Subsection~\ref{SS:asymptotics}.

\subsection{Monodromy and Chern-Simons invariant} \label{SS:monodromy}

It is a well-known fact that every connected Riemannian manifold possesses a natural flat principal bundle. For the hyperbolic case, this produces the \( \mathrm{PSL}_2 {\left( \mathbb{C} \right)} \)-Chern-Simons invariant. We give a brief account of this subject in this subsection, before going into Theorem~\ref{M:B} in the next subsection. Details and proofs omitted in this subsection can be found in~\cite{BaseilhacCS}. See also~\cite{MR807069} and~\cite{MR1218931} for more details.

If \( G \) is a Lie group and \( M \) is a connected smooth manifold with base point \( x \) and without boundary, we have the Riemann-Hilbert one-to-one correspondence between flat principal bundles and holonomy representations as follows.
\begin{align}
& \frac{\{ \textnormal{flat principal \( G \)-bundles over \( M \)} \}}{\textnormal{isomorphism}} \longbijto \frac{\Hom {\left( \pi_1(M, x), G \right)}}{\textnormal{conjugation by \( G \)}}.
\end{align}
More precisely, the inverse is given as follows. Suppose that a holonomy representation \( \varrho \colon \pi_1(M, x) \to G \) is given. We take the following model for the universal covering \( \mathrm{cov} \colon \tilde{M} \surjto M \) (e.g., Theorem~8.4 in Chapter~III of~\cite{MR1224675}).
\begin{align}
\tilde{M} & = {\left\{ {\left[ \beta \right]}_{\mathrm{rel} \, \partial} \, \middle| \, \textnormal{\( \beta \) is a path in \( M \) with \( \beta(0) = x \)} \right\}} \quad \text{and} \quad \mathrm{cov} {\left( {\left[ \beta \right]}_{\mathrm{rel} \, \partial} \right)} = \beta(1). \label{E:universal-covering}
\end{align}
Here, \( {\left[ \beta \right]}_{\mathrm{rel} \, \partial} \) denotes the homotopy class of a path \( \beta \) relative to its endpoints. We define a smooth manifold \( \tilde{M} \times_{\varrho} G \) as the quotient of the product \( \tilde{M} \times G \) by the equivalence relation given by
\begin{align}
& {\left( {\left[ \beta \right]}_{\mathrm{rel} \, \partial}, g \right)} \sim {\left( {\left[ \gamma \right]} \cdot {\left[ \beta \right]}_{\mathrm{rel} \, \partial}, \varrho {\left( {\left[ \gamma \right]} \right)} g \right)} \quad \text{for all} \ {\left[ \gamma \right]} \in \pi_1(M, x),
\end{align}
and define the projection \( \pi_{\varrho} \colon \tilde{M} \times_{\varrho} G \surjto M \) by \( \pi_{\varrho} {\left( {\left[ {\left[ \beta \right]}_{\mathrm{rel} \, \partial}, g \right]} \right)} = \beta(1) \). Then this is a smooth principal \( G \)-bundle over \( M \), and the trivial connection on the trivial principal \( G \)-bundle \( \tilde{M} \times G \surjto \tilde{M} \) descends to a flat connection \( \omega_{\varrho} \) on it.

In particular, if \( M \) is a connected (oriented) Riemannian manifold without boundary, there is a canonical choice of holonomy representation as follows. Let \( G = \Isom^{(+)} {\big( \tilde{M} \big)} \) be the (orientation-preserving) isometry group of the universal covering manifold \( \tilde{M} \) of \( M \), which is a Lie group by the Myers-Steenrod theorem. Then there is a unique (up to isomorphism) flat principal \( G \)-bundle over \( M \) corresponding to the \emph{monodromy representation} \( \hat{\varrho} \colon \pi_1(M) \injto G \). It does not depend on the choice of a base point of \( M \), since a different choice only yields a conjugate.

We can consider the Chern-Simons \( 3 \)-form of this natural flat principal \( G \)-bundle over \( M \), but we need a global smooth section to discuss its Chern-Simons invariant. In particular, if \( M \) is a connected oriented complete Riemannian \( 3 \)-manifold without boundary having constant sectional curvature (i.e., an oriented \( 3 \)-space form), there is a natural bundle map between the oriented orthonormal frame bundle and the flat principal \( G \)-bundle corresponding to the monodromy representation, which is constructed as follows. Fix a base point \( x \in M \), and consider the universal covering \( \mathrm{cov} \colon \tilde{M} \surjto M \) of~\eqref{E:universal-covering}. Note that \( \tilde{M} \) is isometric to either the sphere \( \mathbb{S}^3_r \), the Euclidean space \( \mathbb{R}^3 \), or the hyperbolic space \( \mathbb{H}^3_r \) for some radius \( r > 0 \) by the Killing-Hopf theorem. Let \( \pi \colon FM \surjto M \) and \( \tilde{\pi} \colon F \tilde{M} \surjto \tilde{M} \) be the oriented orthonormal frame bundles. Let \( \hat{\varrho} \colon \pi_1(M, x) \injto G = \Isom^+ {\big( \tilde{M} \big)} \) be the monodromy representation. Fixing a base frame \( {\big( \tilde{x}, \tilde{f} \big)} \in F \tilde{M} \) at the point \( \tilde{x} = {\left[ x \right]}_{\mathrm{rel} \, \partial} \in \tilde{M} \), we can identify \( F \tilde{M} \) with \( G \) by
\begin{align}
& \Phi \colon G \isoto F \tilde{M} \;\!;\, \phi \mapsto {\big( \phi {\left( \tilde{x} \right)}, \phi_* \tilde{f} \big)}. \label{E:identification}
\end{align}
Then we obtain the following diagram.
\begin{align}
\begin{tikzcd}[ampersand replacement=\&]
F \tilde{M} \ar[rr, "\tilde{q} = \tilde{\pi} \times \Phi^{-1}"] \ar[d, two heads] \& \& \tilde{M} \times G \ar[d, two heads] \\
FM \ar[rr, dashed, "q"] \ar[rd, two heads, "\pi"'] \& \& \tilde{M} \times_{\hat{\varrho}} G \ar[ld, two heads, "\pi_{\hat{\varrho}}"] \\
\& M
\end{tikzcd}
\end{align}
The map \( \tilde{q} \) descends to a smooth bundle map \( q \). Therefore, if further \( M \) is parallelizable, a global smooth frame \( s \colon M \injto FM \) induces a global smooth section
\begin{align}
& \sigma = q \circ s \colon M \injto \tilde{M} \times_{\hat{\varrho}} G. \label{E:sigma}
\end{align}

Now, let us focus on the hyperbolic case. Recall that \( \Isom^+ {\left( \mathbb{H}^3 \right)} \iso \mathrm{PSL}_2 {\left( \mathbb{C} \right)} \). If \( M \) is a connected oriented complete hyperbolic \( 3 \)-manifold without boundary, it possesses a natural flat principal \( \mathrm{PSL}_2 {\left( \mathbb{C} \right)} \)-bundle. The \emph{\( \mathrm{PSL}_2 {\left( \mathbb{C} \right)} \)-Chern-Simons\/ \( 3 \)-form} means the corresponding Chern-Simons \( 3 \)-form. If further \( M \) is compact, we can define the \emph{\( \mathrm{PSL}_2 {\left( \mathbb{C} \right)} \)-Chern-Simons invariant} \( \CS_{\mathrm{PSL}_2 {\left( \mathbb{C} \right)}} {\left( M, \sigma \right)} \) as the integral of the pullback of the \( \mathrm{PSL}_2 {\left( \mathbb{C} \right)} \)-Chern-Simons \( 3 \)-form over \( M \), which depends on the choice of a global smooth section \( \sigma \colon M \injto \mathbb{H}^3 \times_{\hat{\varrho}} \mathrm{PSL}_2 {\left( \mathbb{C} \right)} \).

Consider the following upper half space model for the hyperbolic space \( \mathbb{H}^3 \).
\begin{align}
\mathbb{H}^3 & = {\left\{ {\left( x^1, x^2, t \right)} \in \mathbb{R}^3 \, \middle| \, t > 0 \right\}} \quad \text{with} \quad \I_{\mathbb{H}^3} = \frac{(\mathrm{d} x^1)^2 + (\mathrm{d} x^2)^2 + \mathrm{d} t^2}{t^2}. \label{E:hyperbolic-space}
\end{align}
Let \( \tilde{\pi} \colon F \mathbb{H}^3 \surjto \mathbb{H}^3 \) be the oriented orthonormal frame bundle. Let \( \theta \in \Omega^1 {\left( F \mathbb{H}^3, \mathbb{R}^3 \right)} \) be its solder form and \( \bar{\omega} \in \Omega^1 {\left( F \mathbb{H}^3, \mathfrak{so}(3) \right)} \) be its Levi-Civita connection. We have an identification \( \Phi \colon \mathrm{PSL}_2 {\left( \mathbb{C} \right)} \isoto F \mathbb{H}^3 \) of~\eqref{E:identification} with \( \tilde{f} = {\left. {\left( \frac{\partial}{\partial x^1}, \frac{\partial}{\partial x^2}, \frac{\partial}{\partial t} \right)} \right|}_{\hat{\jmath}} \) at \( \hat{\jmath} = (0,0,1) \). Let \( {\left\{ h, e, f \right\}} \) be the basis for \( \mathfrak{sl}_2 {\left( \mathbb{C} \right)} \) over \( \mathbb{C} \) defined by
\begin{align}
h & = \begin{pmatrix*}[r]
1 & 0 \\
0 & -1
\end{pmatrix*}, \qquad e = \begin{pmatrix*}[r]
0 & 1 \\
0 & 0
\end{pmatrix*}, \qquad f = \begin{pmatrix*}[r]
0 & 0 \\
1 & 0
\end{pmatrix*}.
\end{align}
Its dual \( {\left\{ h^*, e^*, f^* \right\}} \) consists of complex-valued left-invariant \( 1 \)-forms on \( \mathrm{PSL}_2 {\left( \mathbb{C} \right)} \). An important equality on \( \mathrm{PSL}_2 {\left( \mathbb{C} \right)} \) is that
\begin{align}
\boldsymbol{i} h^* \wedge e^* \wedge f^* & = \Phi^* {\left( \theta^1 \wedge \theta^2 \wedge \theta^3 - \frac{1}{4} \, \mathrm{d} {\left( \sum_{\mathrm{cyc}} \theta^i \wedge \bar{\omega}^j_k \right)} + \boldsymbol{i} \boldsymbol{\pi}^2 \cs {\left( \bar{\omega} \right)} \right)}. \label{E:hef}
\end{align}
Here, \( \cs {\left( \bar{\omega} \right)} \) is the \( \mathrm{SO}(3) \)-Chern-Simons \( 3 \)-form of \( \mathbb{H}^3 \). This equality is presented in Lemma~3.1 of~\cite{MR807069} and Proposition~3.13 of~\cite{BaseilhacCS}. This implies the following well-known equality between two natural Chern-Simons \( 3 \)-forms.

\begin{pprop} \label{P:CS-form}
Let \( M \) be a connected oriented parallelizable complete hyperbolic\/ \( 3 \)-manifold without boundary. Let \( \pi \colon FM \surjto M \) be the oriented orthonormal frame bundle, and let \( \bar{\omega} \in \Omega^1 {\left( FM, \mathfrak{so}(3) \right)} \) be its Levi-Civita connection. Let \( s \colon M \injto FM \) be a global smooth frame, and let\/ \( {\left( \varepsilon^1, \varepsilon^2, \varepsilon^3 \right)} \) be its dual. Let \( \sigma \) be the global smooth section in~\eqref{E:sigma} induced by \( s \). Then
\begin{align}
4 \boldsymbol{i} \boldsymbol{\pi}^2 \sigma^* \cs {\left( \omega_{\hat{\varrho}} \right)} & = \mathrm{d} \vol - \frac{1}{4} \, \mathrm{d} {\left( \sum_{\mathrm{cyc}} \varepsilon^i \wedge s^* \bar{\omega}^j_k \right)} + \boldsymbol{i} \boldsymbol{\pi}^2 s^* \cs {\left( \bar{\omega} \right)}. \label{E:CS-form}
\end{align}
Here,\/ \( \cs {\left( \omega_{\hat{\varrho}} \right)} \) is the\/ \( \mathrm{PSL}_2 {\left( \mathbb{C} \right)} \)-Chern-Simons\/ \( 3 \)-form of \( M \).
\end{pprop}

\begin{prmrk}
In Proposition~\ref{P:CS-form}, the monodromy representation \( \hat{\varrho} \colon \pi_1(M) \injto G \) actually depends on the choice of a base point \( x \in M \), and the global smooth section \( \sigma \colon M \injto \mathbb{H}^3 \times_{\hat{\varrho}} G \) depends on the choice of a base frame, where \( G = \mathrm{PSL}_2 {\left( \mathbb{C} \right)} \). However, the pullback \( \sigma^* \cs {\left( \omega_{\hat{\varrho}} \right)} \) does not depend on those choices.
\end{prmrk}

\begin{proof}[Proof of Proposition~\ref{P:CS-form}]
As a local argument, it suffices to show the equality at the level of the universal covering manifold \( \mathbb{H}^3 \). Let \( \pi_{\mathsf{g}} \colon \mathbb{H}^3 \times \mathrm{PSL}_2 {\left( \mathbb{C} \right)} \surjto \mathrm{PSL}_2 {\left( \mathbb{C} \right)} \) be the projection, and let \( \mu \) be the Maurer-Cartan \( 1 \)-form on \( \mathrm{PSL}_2 {\left( \mathbb{C} \right)} \). Then
\begin{align}
4 \boldsymbol{i} \boldsymbol{\pi}^2 \sigma^* \cs {\left( \pi_{\mathsf{g}}^* \mu \right)} & = - \frac{4 \boldsymbol{i} \boldsymbol{\pi}^2}{6} \sigma^* \pi_{\mathsf{g}}^* {\left\langle \mu \wedge {\left[ \mu \wedge \mu \right]} \right\rangle} = \boldsymbol{i} \sigma^* \pi_{\mathsf{g}}^* {\left( h^* \wedge e^* \wedge f^* \right)}.
\end{align}
Since \( \sigma = q \circ s \) and \( \Phi \circ \pi_{\mathsf{g}} \circ q = \id_{F \mathbb{H}^3} \), the desired equality follows from~\eqref{E:hef}.
\end{proof}

\begin{pcor} \label{C:CS}
Let \( M \) be a connected closed oriented hyperbolic\/ \( 3 \)-manifold. Let \( s \) be a global oriented orthonormal smooth frame for \( M \). Let \( \sigma \) be the global smooth section in~\eqref{E:sigma} induced by \( s \). Then
\begin{align}
4 \boldsymbol{i} \boldsymbol{\pi}^2 \CS_{\mathrm{PSL}_2 {\left( \mathbb{C} \right)}} {\left( M, \sigma \right)} & = \Vol(M) + \boldsymbol{i} \boldsymbol{\pi}^2 \CS_{\mathrm{SO}(3)} {\left( M, s \right)}.
\end{align}
\end{pcor}

\noindent By Proposition~\ref{P:metric-CS}, we see that \( 4 \boldsymbol{i} \boldsymbol{\pi}^2 \CS_{\mathrm{PSL}_2 {\left( \mathbb{C} \right)}} {\left( M \right)} \in \mathbb{C} / 2 \boldsymbol{i} \boldsymbol{\pi}^2 \mathbb{Z} \) does not require the choice of a global smooth section and is a well-defined invariant for a connected closed oriented hyperbolic \( 3 \)-manifold \( M \).

\subsection{Asymptotics of Chern-Simons invariant} \label{SS:asymptotics}

In this subsection, we prove Theorem~\ref{M:B} and Corollary~\ref{M:C}\@. For the asymptotics of the \( \mathrm{PSL}_2 {\left( \mathbb{C} \right)} \)-Chern-Simons invariant, we cannot ignore the exact form in~\eqref{E:CS-form}, since each \( C_r \) has boundary. An important observation is that this mysterious exact form equals the difference between usual mean curvature and Weitzenb{\"o}ck mean curvature as follows.

\begin{pprop} \label{P:exact-form}
Let \( S \) be an oriented smooth surface embedded in an oriented parallelizable Riemannian\/ \( 3 \)-manifold \( M \) with or without boundary. Let \( \iota \colon S \injto M \) be the inclusion. Let \( \pi \colon FM \surjto M \) be the oriented orthonormal frame bundle, and let \( \bar{\omega} \in \Omega^1 {\left( FM, \mathfrak{so}(3) \right)} \) be its Levi-Civita connection. Let \( s \colon M \injto FM \) be a global smooth frame, and let\/ \( {\left( \varepsilon^1, \varepsilon^2, \varepsilon^3 \right)} \) be its dual. Then
\begin{align}
- \iota^* \sum_{\mathrm{cyc}} \varepsilon^i \wedge s^* \bar{\omega}^j_k & = {\left( \bar{H} - H^s \right)} \, \mathrm{d} a.
\end{align}
Here, \( \bar{H} \) is the usual (i.e., Levi-Civita) mean curvature of \( S \).
\end{pprop}
\begin{proof}
Say \( S \) is oriented by the unit normal vector field \( N \) along \( S \). Let \( {\left( \bar{E}_1, \bar{E}_2 \right)} \) be a local oriented orthonormal smooth frame for \( S \). Let \( \bar{g} \) be the \( \mathrm{SO}(3) \)-valued smooth map given by \( {\left( \bar{E}_1, \bar{E}_2, N \right)} = {\left( E_1, E_2, E_3 \right)} \cdot \bar{g} \). Then
\begin{align*}
& - \sum_{\mathrm{cyc}} \varepsilon^i \wedge s^* \bar{\omega}^j_k {\left( \bar{E}_1, \bar{E}_2 \right)} = - \sum_{\mathrm{cyc}} \sum_{\ell = 1}^3 {\left( \bar{g}^i_1 \bar{g}^{\ell}_2 - \bar{g}^i_2 \bar{g}^{\ell}_1 \right)} s^* \bar{\omega}^j_k {\left( E_{\ell} \right)} \\
& \qquad = - \sum_{\substack{i, j, k \\ \mathrm{cyc}}} {\left( \bar{g}^k_3 s^* \bar{\omega}^j_k {\left( E_j \right)} - \bar{g}^j_3 s^* \bar{\omega}^j_k {\left( E_k \right)} \right)} = - \sum_{j, k = 1}^3 \bar{g}^k_3 s^* \bar{\omega}^j_k {\left( E_j \right)} \\
& \qquad = - \sum_{j, k = 1}^3 N^k \, \I_M {\left( E_j, \bar{\del}^M_{E_j} E_k \right)} = \sum_{j = 1}^3 \I_M {\left( \bar{\del}^M_{E_j} E_j, N \right)} = \bar{H} - H^s
\end{align*}
by Proposition~\ref{P:comparison}.
\end{proof}

Suppose the setup of Theorem~\ref{M:B}\@. By Propositions~\ref{P:CS-form} and~\ref{P:exact-form}, we have
\begin{align}
4 \boldsymbol{i} \boldsymbol{\pi}^2 \int_{C_{\rho}} \sigma^* \cs {\left( \omega_{\hat{\varrho}} \right)} & = \Vol {\left( C_{\rho} \right)} + \frac{1}{4} \int_{\partial C_{\rho}} {\left( \bar{H} - H^s \right)} \, \mathrm{d} a_{\rho} + \boldsymbol{i} \boldsymbol{\pi}^2 \int_{C_{\rho}} s^* \cs {\left( \bar{\omega} \right)} \label{E:integral}
\end{align}
for all \( \rho \ge 0 \). Here, the volume plus the quarter of the integrated mean curvature is exactly the \emph{\( \mathrm{W} \)-volume} introduced in Definition~3.1 of~\cite{MR2386723}. (The sign difference is due to the different choice of orientation.) The asymptotics of \( \mathrm{W} \)-volume is computed in Lemma~4.2 of~\cite{MR2386723} as follows. For all \( \rho \ge 0 \),
\begin{align}
\Vol {\left( C_{\rho} \right)} + \frac{1}{4} \int_{\partial C_{\rho}} \bar{H} \, \mathrm{d} a_{\rho} & = \mathrm{W} {\left( C_{\rho} \right)} = - \boldsymbol{\pi} \rho \chi {\left( \partial C \right)} + \mathrm{W} {\left( C \right)}.
\end{align}
Since we know the last integral of~\eqref{E:integral} by Theorem~\ref{M:A}, it remains to find the asymptotics of the integrated Weitzenb{\"o}ck mean curvature. Say \( s = {\left( E_1, E_2, E_3 \right)} \). For convenience, let \( N^i = \I_M {\left( E_i, \frac{\partial}{\partial r} \right)} \) for each \( i \in \{ 1, 2, 3 \} \). For each \( i \in \{ 1, 2, 3 \} \), \( r \mapsto {\left. N^i \right|}_{\partial C_r} \circ u_r^{-1} \) is constant on \( [0, \infty) \). For each \( x \in \partial C \), Proposition~\ref{P:mean-curvature-torsion} gives
\begin{align*}
{\big. H^s \big|}_{u_{\rho}^{-1}(x)} & = - \sum_{i = 1}^3 {\left. E_i^{\top} \right|}_{u_{\rho}^{-1}(x)} {\left( N^i \right)} = - \sum_{i = 1}^3 {\left( u_{\rho}^{-1} \right)}_* A_{\rho}^{-1} {\left. E_i^{\top} \right|}_x {\left( N^i \right)} \\
& = - \sum_{i = 1}^3 A_{\rho}^{-1} {\left. E_i^{\top} \right|}_x {\left( N^i \right)}, \\
{\big. H^{s^{\infty}} \big|}_{v^{-1}(x)} & = - \sum_{i = 1}^3 {\left. {\left( E_i^{\infty} \right)} {}^{\top} \right|}_{v^{-1}(x)} {\left( N^i \circ v \right)} = - \sum_{i = 1}^3 {\left( v^{-1} \right)}_* V^{-1} {\left. E_i^{\top} \right|}_x {\left( N^i \circ v \right)} \\
& = - \sum_{i = 1}^3 V^{-1} {\left. E_i^{\top} \right|}_x {\left( N^i \right)}.
\end{align*}
Therefore, for all \( \rho \ge 0 \),
\begin{align*}
& - \frac{1}{4} \int_{\partial C_{\rho}} H^s \, \mathrm{d} a_{\rho} = \frac{1}{4} \int_{\partial C} \sum_{i = 1}^3 A_{\rho}^{-1} E_i^{\top} {\left( N^i \right)} \det \! {\left( A_{\rho} \right)} \, \mathrm{d} a \\
& \qquad = \frac{1}{4} \int_{\partial C} \sum_{i = 1}^3 \adj \! {\left( A_{\rho} \right)} \, E_i^{\top} {\left( N^i \right)} \, \mathrm{d} a = \frac{\boldsymbol{e}^{\rho}}{4 \sqrt{2}} \int_{\partial C} \sum_{i = 1}^3 \adj \! {\left( V \right)} \, E_i^{\top} {\left( N^i \right)} \, \mathrm{d} a + \boldsymbol{e}^{-\rho} \mathcal{J} \\
& \qquad = \frac{\boldsymbol{e}^{\rho}}{4 \sqrt{2}} \int_{\partial C} \sum_{i = 1}^3 V^{-1} E_i^{\top} {\left( N^i \right)} \det \! {\left( V \right)} \, \mathrm{d} a + \boldsymbol{e}^{-\rho} \mathcal{J} \\
& \qquad = - \frac{\boldsymbol{e}^{\rho}}{4 \sqrt{2}} \int_{\partial C^{\infty}} H^{s^{\infty}} \, \mathrm{d} a^{\infty} + \boldsymbol{e}^{-\rho} \mathcal{J},
\end{align*}
where \( \mathcal{J} \) is an integral independent of \( \rho \). Consequently, we find that~\eqref{E:integral} becomes
\begin{align}
\begin{aligned}
4 \boldsymbol{i} \boldsymbol{\pi}^2 \int_{C_{\rho}} \sigma^* \cs {\left( \omega_{\hat{\varrho}} \right)} & = \begin{aligned}[t]
& - \frac{\boldsymbol{e}^{\rho}}{4 \sqrt{2}} \int_{\partial C^{\infty}} {\left( H {\left( s^{\infty} \right)} \, \mathrm{d} a^{\infty} + \boldsymbol{i} \tau {\left( s^{\infty} \right)} \right)} - \boldsymbol{\pi} \rho \chi {\left( \partial C \right)} \\
& + (\textnormal{convergent terms})
\end{aligned}
\end{aligned}
\end{align}
as \( \rho \to \infty \). This completes the proof of Theorem~\ref{M:B}\@. By Lemma~4.5 of~\cite{MR2386723}, the renormalized volume and the renormalized \( \mathrm{W} \)-volume satisfies
\begin{align}
\Vol^{\textnormal{R}} {\left( M, C \right)} & = \mathrm{W}^{\textnormal{R}} {\left( M, C \right)} + \frac{\boldsymbol{\pi}}{2} \chi {\left( \partial_{\infty} M \right)}.
\end{align}
We obtain Corollary~\ref{M:C} as well.

\begin{prmrk} \label{R:variation}
Given a convex-cocompact hyperbolic \( 3 \)-manifold, the renormalized volume \( \Vol^{\textnormal{R}} \) defines a smooth function on the Teichm{\"u}ller space of the conformal boundary, which locally (i.e., up to quotient) serves as the deformation space of convex-cocompact hyperbolic structures. The variational formula of renormalized volume in Lemma~8.5 of~\cite{MR2386723} (originally appearing in~\cite{MR0889594} and~\cite{MR1997440}) states that the differential on the Teichm{\"u}ller space satisfies \( \mathrm{d} \Vol^{\textnormal{R}} = \re \theta \) for a certain holomorphic \( 1 \)-form \( \theta \) on the Teichm{\"u}ller space given by the difference between two natural complex projective structures on the conformal boundary. Motivated by this, one might ask whether the equality
\begin{align}
\mathrm{d} {\big( \Vol^{\textnormal{R}} + \boldsymbol{i} \boldsymbol{\pi}^2 \CS_{\mathrm{SO}(3)}^{\textnormal{R}} \big)} & = \theta \label{E:attempt}
\end{align}
defines the renormalized \( \mathrm{SO}(3) \)-Chern-Simons invariant, without going through the renormalization procedure. If this is the case, then
\begin{align}
\begin{gathered}
\boldsymbol{\pi}^2 \mathrm{d} \CS_{\mathrm{SO}(3)}^{\textnormal{R}} = \im \theta, \\
\boldsymbol{\pi}^2 \partial \CS_{\mathrm{SO}(3)}^{\textnormal{R}} = - \frac{\boldsymbol{i}}{2} \theta = - \boldsymbol{i} \, \partial \Vol^{\textnormal{R}}, \qquad \boldsymbol{\pi}^2 \bar{\partial} \CS_{\mathrm{SO}(3)}^{\textnormal{R}} = \frac{\boldsymbol{i}}{2} \bar{\theta} = \boldsymbol{i} \, \bar{\partial} \Vol^{\textnormal{R}},
\end{gathered}
\end{align}
which imply that
\begin{align}
\bar{\partial} {\big( \Vol^{\textnormal{R}} + \boldsymbol{i} \boldsymbol{\pi}^2 \CS_{\mathrm{SO}(3)}^{\textnormal{R}} \big)} & = 0,
\end{align}
i.e., they form a holomorphic function. This is not possible, as the renormalized volume \( \Vol^{\textnormal{R}} \) is known to become a K{\"a}hler potential for the Weil-Petersson metric: \( \partial \bar{\partial} \Vol^{\textnormal{R}} = 2 \boldsymbol{i} \omega_{\textnormal{WP}} \). Therefore, the attempt~\eqref{E:attempt} fails.
\end{prmrk}

\begin{prmrk}
It is possible to generalize Theorem~\ref{M:A}, Theorem~\ref{M:B}, and Corollary~\ref{M:C} to arbitrary flat connections, not only global smooth frames. First of all, the Chern-Simons invariant can be generalized by replacing \( \omega \) with \( \omega - \omega' \). Indeed, if \( \omega \) and \( \omega' \) are two connections on a smooth principal \( G \)-bundle \( \pi \colon P \surjto M \) with curvatures \( \Omega \) and \( \Omega' \), then the pullback of
\begin{align}
& {\left\langle {\left( \omega - \omega' \right)} \wedge {\left\langle \Omega - \Omega' \right\rangle} \right\rangle} - \frac{1}{6} {\left\langle {\left( \omega - \omega' \right)} \wedge {\left[ {\left( \omega - \omega' \right)} \wedge {\left( \omega - \omega' \right)} \right]} \right\rangle} \in \Omega^3 {\left( P, \mathbb{C} \right)}
\end{align}
via a local smooth section does not depend on the choice of the local smooth section, so it produces a globally well-defined complex-valued \( 3 \)-form on \( M \). In particular, if the connection \( \omega' \) is chosen to be the Weitzenb{\"o}ck connection \( \omega^{\sigma} \) induced by a global smooth section \( \sigma \colon M \injto P \), then it coincides with the familiar pullback \( \sigma^* \cs {\left( \omega \right)} \), since \( \sigma^* \omega^{\sigma} = 0 \).

Then the notion of constant global smooth frame can be generalized to the notion of \emph{constant flat connection} by considering horizontal sections (cf.~Remark~\ref{R:flat-connection}). Also, we can define the \emph{flat connection at infinity} via the isomorphism in the left commutative diagram of~\eqref{E:frame-at-infinity}. Now, we can generalize Theorem~\ref{M:A}, Theorem~\ref{M:B}, and Corollary~\ref{M:C} to constant flat connections so that the original ones are the Weitzenb{\"o}ck case of the generalized version. The idea is to consider a smooth triangulation \( S = \bigcup_{i = 1}^n \Delta^i \) whose triangles are sufficiently small so that each triangle \( \Delta^i \) has an open neighborhood \( U^i \) in \( S \) and a horizontal section \( s^i \colon [U^i_0, U^i_{\infty}) \injto {\left. FM \right|}_{[U^i_0, U^i_{\infty})} \) that is constant along the foliation \( {\left\{ U_r \right\}}_{r \in [0, \infty)} \). However, it is not clear whether there are cases in which such a generalization is truly meaningful, since we know that there already exist a lot of constant global smooth frames (Proposition~\ref{P:constant-frame}).
\end{prmrk}

\appendix

\section{Lemmas and proofs}

We state several omitted lemmas, computations, and proofs.

\subsection{Lemmas in geometry and analysis}

We enumerate three lemmas here.

\begin{plem}[Smooth approximation theorem] \label{L:SAT}
Let \( M \) and \( N \) be smooth manifolds without boundary. Suppose that \( N \) is compact. Let \( f \colon M \to N \) be a continuous map that is smooth on a closed subset \( A \) of \( M \). Then there exists a smooth map \( g \colon M \to N \) such that \( f = g \) on \( A \) and \( f \simeq g \ \mathrm{rel} \, A \).
\end{plem}
\begin{proof}
See, e.g., Theorem~11.8 in Chapter~II of~\cite{MR1224675}.
\end{proof}

\begin{plem} \label{L:curvature}
Let\/ \( (M, g) \) be a Riemannian \( n \)-manifold with or without boundary having constant sectional curvature \( \kappa \). Let\/ \( {\left( E_1, \dotsc, E_n \right)} \) be a local smooth frame for \( M \), and let\/ \( {\left( \varepsilon^1, \dotsc, \varepsilon^n \right)} \) be its dual. Then the Levi-Civita curvature\/ \( 2 \)-forms satisfy
\begin{align}
\bar{\Omega}^i_j & = \sum_{k = 1}^n \kappa g_{jk} \varepsilon^i \wedge \varepsilon^k.
\end{align}
\end{plem}
\begin{proof}
The curvature of \( M \) satisfies \( \bar{R} (X, Y) Z = \kappa {\left( g(Y, Z) X - g(X, Z) Y \right)} \) for all \( X, Y, Z \in \Gamma(TM) \) (e.g., Proposition~8.36 of~\cite{MR3887684}). Therefore, we have
\begin{align*}
\bar{\Omega}^i_j(X, Y) E_i & = \bar{R} (X, Y) E_j = \kappa g_{jk} {\left( Y^k X - X^k Y \right)} = \kappa g_{jk} \varepsilon^i \wedge \varepsilon^k {\left( X, Y \right)} E_i
\end{align*}
for all \( X, Y \in \Gamma(TM) \), where the Einstein summation convention is assumed.
\end{proof}

\begin{plem} \label{L:limit}
Let \( R \) be a compact oriented Riemannian manifold with or without boundary, and let\/ \( \mathrm{d} v \) be its volume form. Let \( a, b, c, \kappa_1, \kappa_2 \colon R \to \mathbb{R} \) be smooth functions. Suppose that \( \kappa_1 \) and \( \kappa_2 \) are nonpositive. For each \( r \in [0, \infty) \), let
\begin{align*}
f_r & = \frac{a \cosh^2(r) + b \cosh(r) \sinh(r) + c \sinh^2(r)}{{\left( \cosh(r) - \kappa_1 \sinh(r) \right)} {\left( \cosh(r) - \kappa_2 \sinh(r) \right)}} \quad \text{and} \quad \ell = \frac{a+b+c}{{\left( 1 - \kappa_1 \right)} {\left( 1 - \kappa_2 \right)}}.
\end{align*}
Then the following limits both converge. In fact, the former converges to zero.
\begin{align}
& \lim_{\rho \to \infty} \int_R {\left( f_{\rho} - \ell \right)} \, \mathrm{d} v, \qquad \lim_{\rho \to \infty} \int_R \int_0^{\rho} {\left( f_r - \ell \right)} \, \mathrm{d} r \, \mathrm{d} v.
\end{align}
\end{plem}
\begin{proof}
For each \( r \in [0, \infty) \), a direct computation yields
\begin{align*}
f_r - \ell & = \frac{\boldsymbol{e}^{-r} {\left( u \cosh(r) + v \sinh(r) \right)}}{{\left( 1 - \kappa_1 \right)} {\left( 1 - \kappa_2 \right)} {\left( \cosh(r) - \kappa_1 \sinh(r) \right)} {\left( \cosh(r) - \kappa_2 \sinh(r) \right)}}
\end{align*}
for some smooth functions \( u, v \colon R \to \mathbb{R} \). Therefore, for each \( r \in [0, \infty) \),
\begin{align*}
{\left| f_r - \ell \right|} & \le \frac{\boldsymbol{e}^{-r} {\left( |u| \cosh(r) + |v| \sinh(r) \right)}}{{\left( 1 - \bar{\kappa}_1 \right)} {\left( 1 - \bar{\kappa}_2 \right)} {\left( \cosh(r) - \bar{\kappa}_1 \sinh(r) \right)} {\left( \cosh(r) - \bar{\kappa}_2 \sinh(r) \right)}},
\end{align*}
where \( \bar{\kappa}_i = \max_R \kappa_i \in (-\infty, 0] \) for each \( i \in \{ 1, 2 \} \). Then, for each \( r \in [0, \infty) \),
\begin{align*}
\int_R {\left| f_r - \ell \right|} \, \mathrm{d} v & \le \frac{\boldsymbol{e}^{-r} {\left( {\left( \int_R |u| \, \mathrm{d} v \right)} \cosh(r) + {\left( \int_R |v| \, \mathrm{d} v \right)} \sinh(r) \right)}}{{\left( 1 - \bar{\kappa}_1 \right)} {\left( 1 - \bar{\kappa}_2 \right)} {\left( \cosh(r) - \bar{\kappa}_1 \sinh(r) \right)} {\left( \cosh(r) - \bar{\kappa}_2 \sinh(r) \right)}} \\
& = \boldsymbol{e}^{-r} g(r),
\end{align*}
where \( g(r) \to 0 \) as \( r \to \infty \). This proves the former. For large \( \rho_1 \) and \( \rho_2 \),
\begin{align*}
{\left| \int_R \int_{\rho_1}^{\rho_2} {\left( f_r - \ell \right)} \, \mathrm{d} r \, \mathrm{d} v \right|} & \le \int_{\rho_1}^{\rho_2} \boldsymbol{e}^{-r} g(r) \, \mathrm{d} r \le \int_{\rho_1}^{\rho_2} \boldsymbol{e}^{-r} \, \mathrm{d} r \longto 0
\end{align*}
as \( \rho_1, \rho_2 \to \infty \). This proves the latter by Cauchy's criterion.
\end{proof}

\subsection{Lemmas in Section~\ref{S:asymptotics}}

We assume the setup of Section~\ref{S:asymptotics} in this subsection.

\begin{plem} \label{L:self-adjoint}
For each \( r \in (-\varepsilon, \infty) \) and each \( n \in \mathbb{Z}_{\ge 0} \), \( B^n A_r^{-1} \) is self-adjoint.
\end{plem}
\begin{proof}
The case \( r = 0 \) is obvious, so consider the case \( r \ne 0 \). We use induction on \( n \in \mathbb{Z}_{\ge 0} \). Since \( A_r \) is self-adjoint, the base case \( n = 0 \) holds. We have
\begin{align*}
I & = \cosh(r) A_r^{-1} - \sinh(r) B A_r^{-1}, \\
B^{n+1} A_r^{-1} & = B^n {\left( \coth(r) A_r^{-1} - \csch(r) I \right)} = \coth(r) B^n A_r^{-1} - \csch(r) B^n.
\end{align*}
This completes the induction.
\end{proof}

\begin{plem} \label{L:Lie-bracket}
Fix \( x \in R \) and let \( x_r = u_r^{-1}(x) \in R_r \). For each \( j \in \{ 1, 2 \} \), we have
\begin{align*}
A_r {\left( u_r \right)}_* {\left. {\left[ \bar{E}_1, \bar{E}_2 \right]} \right|}_{x_r} & = \frac{A_r {\left. {\left[ \bar{E}_1, \bar{E}_2 \right]} \right|}_x + \sinh(r) {\left. {\left( \bar{E}_1 (B) \bar{E}_2 - \bar{E}_2 (B) \bar{E}_1 \right)} \right|}_x}{\det {\left. A_r \right|}_x}, \\
A_r {\left( u_r \right)}_* {\left. {\left[ \bar{E}_j, \frac{\partial}{\partial r} \right]} \right|}_{x_r}^{\top} & = \sum_{\ell = 1}^2 {\left( A_r|_x^{-1} \right)}^{\ell}_j \, A_r {\left. {\left[ \bar{E}_{\ell}, \frac{\partial}{\partial r} \right]} \right|}_x^{\top} + \sinh(r) {\left( I - B^2 \right)} A_r^{-1} {\left. \bar{E}_j \right|}_x.
\end{align*}
\end{plem}
\begin{proof}
We have
\begin{align*}
A_r {\left( u_r \right)}_* {\left. {\left[ \bar{E}_1, \bar{E}_2 \right]} \right|}_{x_r} & = A_r {\left( u_r \right)}_* {\left. {\left[ {\left. \bar{E}_1 \right|}_{R_r}, {\left. \bar{E}_2 \right|}_{R_r} \right]} \right|}_{x_r} = A_r {\left. {\left[ A_r^{-1} {\left. \bar{E}_1 \right|}_R, A_r^{-1} {\left. \bar{E}_2 \right|}_R \right]} \right|}_x.
\end{align*}
At \( x \), the vector \( A_r {\left[ A_r^{-1} \bar{E}_1, A_r^{-1} \bar{E}_2 \right]} \) equals
\begin{align*}
& A_r {\left( {\left( A_r^{-1} \right)}^k_1 {\left( A_r^{-1} \right)}^{\ell}_2 {\left[ \bar{E}_k, \bar{E}_{\ell} \right]} + {\left( A_r^{-1} \bar{E}_1 \right)} {\left( A_r^{-1} \right)} \bar{E}_2 - {\left( A_r^{-1} \bar{E}_2 \right)} {\left( A_r^{-1} \right)} \bar{E}_1 \right)} \\
& \qquad = \begin{aligned}[t]
& A_r {\left( {\left( A_r^{-1} \right)}^1_1 {\left( A_r^{-1} \right)}^2_2 - {\left( A_r^{-1} \right)}^2_1 {\left( A_r^{-1} \right)}^1_2 \right)} {\left[ \bar{E}_1, \bar{E}_2 \right]} \\
& - {\left( A_r^{-1} \bar{E}_1 \right)} {\left( A_r \right)} A_r^{-1} \bar{E}_2 + {\left( A_r^{-1} \bar{E}_2 \right)} {\left( A_r \right)} A_r^{-1} \bar{E}_1
\end{aligned} \\
& \qquad = A_r \det \! {\left( A_r^{-1} \right)} {\left[ \bar{E}_1, \bar{E}_2 \right]} - {\left( {\left( A_r^{-1} \right)}^k_1 {\left( A_r^{-1} \right)}^{\ell}_2 - {\left( A_r^{-1} \right)}^k_2 {\left( A_r^{-1} \right)}^{\ell}_1 \right)} \bar{E}_k {\left( A_r \right)} \bar{E}_{\ell} \\
& \qquad = \det \! {\left( A_r^{-1} \right)} A_r {\left[ \bar{E}_1, \bar{E}_2 \right]} - \det \! {\left( A_r^{-1} \right)} {\left( \bar{E}_1 {\left( A_r \right)} \bar{E}_2 - \bar{E}_2 {\left( A_r \right)} \bar{E}_1 \right)},
\end{align*}
where the Einstein summation convention is assumed. This proves the former.

For the latter, let \( \tilde{u}_r \colon (S_{r - \varepsilon}, S_{r + \varepsilon}) \to (S_{-\varepsilon}, S_{+\varepsilon}) \) be the ``\( r \)-translation'' map, which is an orientation-preserving diffeomorphism. Note that \( {\left. \tilde{u}_r \right|}_{S_r} = u_r \colon S_r \to S \). Let \( \tilde{A}_r \) be the extension of \( A_r \) to \( (S_{-\varepsilon}, S_{+\varepsilon}) \) by replacing \( B \) with the family of Weingarten maps of the leaf surfaces in \( (S_{-\varepsilon}, S_{+\varepsilon}) \). We have
\begin{align*}
A_r {\left( u_r \right)}_* {\left. {\left[ \bar{E}_j, \frac{\partial}{\partial r} \right]} \right|}_{x_r}^{\top} & = A_r {\left. {\left[ {\left( \tilde{u}_r \right)}_* \bar{E}_j, {\left( \tilde{u}_r \right)}_* \frac{\partial}{\partial r} \right]} \right|}_x^{\top} = A_r {\left. {\left[ \tilde{A}_r^{-1} \bar{E}_j, \frac{\partial}{\partial r} \right]} \right|}_x^{\top}.
\end{align*}
At \( x \), we have
\begin{align*}
A_r {\left[ \tilde{A}_r^{-1} \bar{E}_j, \frac{\partial}{\partial r} \right]}^{\top} & = \sum_{\ell = 1}^2 A_r {\left( A_r^{-1} \right)}^{\ell}_j {\left[ \bar{E}_{\ell}, \frac{\partial}{\partial r} \right]}^{\top} - A_r \frac{\partial \tilde{A}_r|_{S_t}^{-1}}{\partial t} \bar{E}_j.
\end{align*}
Here, the second term equals
\begin{align*}
\frac{\partial \tilde{A}_r|_{S_t}}{\partial t} A_r^{-1} \bar{E}_j & = - \sinh(r) \frac{\partial B_t}{\partial t} A_r^{-1} \bar{E}_j = - \sinh(r) {\left( B^2 - I \right)} A_r^{-1} \bar{E}_j.
\end{align*}
This proves the latter.
\end{proof}

\begin{plem} \label{L:Lie-bracket-infinity}
Fix \( x \in R \) and let \( y = v^{-1}(x) \in R^{\infty} \). For each \( j \in \{ 1, 2 \} \), we have
\begin{align*}
V v_* {\left. {\left[ \bar{E}_1^{\infty}, \bar{E}_2^{\infty} \right]} \right|}_y & = \frac{V {\left. {\left[ \bar{E}_1, \bar{E}_2 \right]} \right|}_x + \frac{1}{\sqrt{2}} {\left. {\left( \bar{E}_1 (B) \bar{E}_2 - \bar{E}_2 (B) \bar{E}_1 \right)} \right|}_x}{\det {\left. V \right|}_x}, \\
V v_* {\left. {\left[ \bar{E}_j^{\infty}, \frac{\partial}{\partial r} \right]} \right|}_y^{\top} & = \sum_{\ell = 1}^2 {\left( V|_x^{-1} \right)}^{\ell}_j \, V {\left. {\left[ \bar{E}_{\ell}, \frac{\partial}{\partial r} \right]} \right|}_x^{\top} + (I+B) {\left. \bar{E}_j \right|}_x.
\end{align*}
\end{plem}
\begin{proof}
We have
\begin{align*}
V v_* {\left. {\left[ \bar{E}_1^{\infty}, \bar{E}_2^{\infty} \right]} \right|}_y & = V v_* {\left. {\left[ {\left. \bar{E}_1^{\infty} \right|}_{R^{\infty}}, {\left. \bar{E}_2^{\infty} \right|}_{R^{\infty}} \right]} \right|}_y = V {\left. {\left[ V^{-1} {\left. \bar{E}_1 \right|}_R, V^{-1} {\left. \bar{E}_2 \right|}_R \right]} \right|}_x.
\end{align*}
At \( x \), the vector \( V {\left[ V^{-1} \bar{E}_1, V^{-1} \bar{E}_2 \right]} \) equals
\begin{align*}
& V {\left( {\left( V^{-1} \right)}^k_1 {\left( V^{-1} \right)}^{\ell}_2 {\left[ \bar{E}_k, \bar{E}_{\ell} \right]} + {\left( V^{-1} \bar{E}_1 \right)} {\left( V^{-1} \right)} \bar{E}_2 - {\left( V^{-1} \bar{E}_2 \right)} {\left( V^{-1} \right)} \bar{E}_1 \right)} \\
& \qquad = \begin{aligned}[t]
& V {\left( {\left( V^{-1} \right)}^1_1 {\left( V^{-1} \right)}^2_2 - {\left( V^{-1} \right)}^2_1 {\left( V^{-1} \right)}^1_2 \right)} {\left[ \bar{E}_1, \bar{E}_2 \right]} \\
& - {\left( V^{-1} \bar{E}_1 \right)} {\left( V \right)} V^{-1} \bar{E}_2 + {\left( V^{-1} \bar{E}_2 \right)} {\left( V \right)} V^{-1} \bar{E}_1
\end{aligned} \\
& \qquad = V \det \! {\left( V^{-1} \right)} {\left[ \bar{E}_1, \bar{E}_2 \right]} - {\left( {\left( V^{-1} \right)}^k_1 {\left( V^{-1} \right)}^{\ell}_2 - {\left( V^{-1} \right)}^k_2 {\left( V^{-1} \right)}^{\ell}_1 \right)} \bar{E}_k {\left( V \right)} \bar{E}_{\ell} \\
& \qquad = \det \! {\left( V^{-1} \right)} V {\left[ \bar{E}_1, \bar{E}_2 \right]} - \det \! {\left( V^{-1} \right)} {\left( \bar{E}_1 {\left( V \right)} \bar{E}_2 - \bar{E}_2 {\left( V \right)} \bar{E}_1 \right)},
\end{align*}
where the Einstein summation convention is assumed. This proves the former.

For the latter, let \( \tilde{v} \colon (S_{-\varepsilon}^{\infty}, S_{+\varepsilon}^{\infty}) \to (S_{-\varepsilon}, S_{+\varepsilon}) \) be the set-identity map, which is an orientation-preserving diffeomorphism. Let \( \tilde{V} \) be the extension of \( V \) to \( (S_{-\varepsilon}, S_{+\varepsilon}) \) by replacing \( B \) with the family of Weingarten maps of the leaf surfaces in \( (S_{-\varepsilon}, S_{+\varepsilon}) \). We have
\begin{align*}
V v_* {\left. {\left[ \bar{E}_j^{\infty}, \frac{\partial}{\partial r} \right]} \right|}_y^{\top} & = V {\left. {\left[ \tilde{v}_* \bar{E}_j^{\infty}, \tilde{v}_* \frac{\partial}{\partial r} \right]} \right|}_x^{\top} = V {\left. {\left[ \tilde{V}^{-1} \bar{E}_j, \frac{\partial}{\partial r} \right]} \right|}_x^{\top}.
\end{align*}
At \( x \), we have
\begin{align*}
V {\left[ \tilde{V}^{-1} \bar{E}_j, \frac{\partial}{\partial r} \right]}^{\top} & = \sum_{\ell = 1}^2 V {\left( V^{-1} \right)}^{\ell}_j {\left[ \bar{E}_{\ell}, \frac{\partial}{\partial r} \right]}^{\top} - V \frac{\partial \tilde{V}|_{S_r}^{-1}}{\partial r} \bar{E}_j.
\end{align*}
Here, the second term equals
\begin{align*}
\frac{\partial \tilde{V}|_{S_r}}{\partial r} V^{-1} \bar{E}_j & = - \frac{1}{\sqrt{2}} \frac{\partial B_r}{\partial r} V^{-1} \bar{E}_j = - \frac{1}{\sqrt{2}} {\left( B^2 - I \right)} V^{-1} \bar{E}_j = (I+B) \bar{E}_j.
\end{align*}
This proves the latter.
\end{proof}

\subsection{A proof of~(\ref{E:CS-normal-frame})} \label{AA:CS-normal-frame}

First of all, note that \( - \frac{1}{2} \tr = 4 \boldsymbol{\pi}^2 {\left\langle {}\cdot{}, {}\cdot{} \right\rangle} \) is an inner product on \( \mathfrak{so}(3) \). Let \( {\left\{ L_1, L_2, L_3 \right\}} \) be the orthonormal basis for \( \mathfrak{so}(3) \) defined by
\begin{align}
L_1 & = \begin{pmatrix*}[r]
0 & 0 & 0 \\
0 & 0 & -1 \\
0 & 1 & 0
\end{pmatrix*}, \qquad L_2 = \begin{pmatrix*}[r]
0 & 0 & 1 \\
0 & 0 & 0 \\
-1 & 0 & 0
\end{pmatrix*}, \qquad L_3 = \begin{pmatrix*}[r]
0 & -1 & 0 \\
1 & 0 & 0 \\
0 & 0 & 0
\end{pmatrix*}.
\end{align}
Consider locally near a point in \( (R_{-\varepsilon}, R_{\infty}) \). Say the point belongs to \( R_r \). Only in this proof, let \( \bar{\upsilon} = \bar{s}^* \bar{\omega} \) for convenience. By Lemma~\ref{L:curvature}, we have
\begin{align}
\bar{\upsilon} & = - \bar{\upsilon}^2_3 \otimes L_1 - \bar{\upsilon}^3_1 \otimes L_2 - \bar{\upsilon}^1_2 \otimes L_3, \\
\bar{s}^* \bar{\Omega} & = \bar{\varepsilon}^2 \wedge \bar{\varepsilon}^3 \otimes L_1 + \bar{\varepsilon}^3 \wedge \bar{\varepsilon}^1 \otimes L_2 + \bar{\varepsilon}^1 \wedge \bar{\varepsilon}^2 \otimes L_3.
\end{align}
Therefore, we have
\begin{align}
\begin{aligned}
\bar{s}^* \cs {\left( \bar{\omega} \right)} & = {\left\langle \bar{\upsilon} \wedge \bar{s}^* \bar{\Omega} \right\rangle} - \frac{1}{6} {\left\langle \bar{\upsilon} \wedge {\left[ \bar{\upsilon} \wedge \bar{\upsilon} \right]} \right\rangle} \\
& = \frac{1}{4 \boldsymbol{\pi}^2} {\left( \bar{\upsilon}^1_2 \wedge \bar{\upsilon}^2_3 \wedge \bar{\upsilon}^3_1 - \bar{\upsilon}^1_2 \wedge \bar{\varepsilon}^1 \wedge \bar{\varepsilon}^2 - \bar{\upsilon}^2_3 \wedge \bar{\varepsilon}^2 \wedge \bar{\varepsilon}^3 - \bar{\upsilon}^3_1 \wedge \bar{\varepsilon}^3 \wedge \bar{\varepsilon}^1 \right)}.
\end{aligned}
\end{align}

Meanwhile, since \( \bar{\varepsilon}^3 = \mathrm{d} r \), we have \( 0 = \mathrm{d} \bar{\varepsilon}^3 = - \bar{\upsilon}^3_1 \wedge \bar{\varepsilon}^1 - \bar{\upsilon}^3_2 \wedge \bar{\varepsilon}^2 \) by Cartan's first equation. This implies that
\begin{align}
\bar{\upsilon}^3_1 {\left( \bar{E}_2 \right)} & = \bar{\upsilon}^3_2 {\left( \bar{E}_1 \right)} \quad \text{and} \quad \bar{\upsilon}^3_1 {\left( \bar{E}_3 \right)} = \bar{\upsilon}^3_2 {\left( \bar{E}_3 \right)} = 0. \label{E:connection-vanish}
\end{align}
Also, for any \( i, j \in \{ 1, 2 \} \),
\begin{align}
\bar{\upsilon}^3_i {\left( \bar{E}_j \right)} & = - \bar{\upsilon}^i_3 {\left( \bar{E}_j \right)} = \I_M {\left( \bar{E}_i, - \bar{\del}^M_{\bar{E}_j} \bar{E}_3 \right)} = \I_M {\left( \bar{E}_i, B_r \bar{E}_j \right)} = {\left( B_r \right)}^i_j.
\end{align}
Therefore, \( \bar{\upsilon}^3_i = {\left( B_r \right)}^i_1 \, \bar{\varepsilon}^1 + {\left( B_r \right)}^i_2 \, \bar{\varepsilon}^2 \) for each \( i \in \{ 1, 2 \} \). It follows that
\begin{gather}
\begin{aligned}
\bar{\upsilon}^1_2 \wedge \bar{\upsilon}^2_3 \wedge \bar{\upsilon}^3_1 & = \bar{\upsilon}^1_2 \wedge {\left( - {\left( B_r \right)}^2_1 \, \bar{\varepsilon}^1 - {\left( B_r \right)}^2_2 \, \bar{\varepsilon}^2 \right)} \wedge {\left( {\left( B_r \right)}^1_1 \, \bar{\varepsilon}^1 + {\left( B_r \right)}^1_2 \, \bar{\varepsilon}^2 \right)} \\
& = \det \! {\left( B_r \right)} \, \bar{\upsilon}^1_2 \wedge \bar{\varepsilon}^1 \wedge \bar{\varepsilon}^2,
\end{aligned} \\
\bar{\upsilon}^2_3 \wedge \bar{\varepsilon}^2 \wedge \bar{\varepsilon}^3 + \bar{\upsilon}^3_1 \wedge \bar{\varepsilon}^3 \wedge \bar{\varepsilon}^1 = {\left( \bar{\upsilon}^2_3 {\left( \bar{E}_1 \right)} + \bar{\upsilon}^3_1 {\left( \bar{E}_2 \right)} \right)} \, \bar{\varepsilon}^1 \wedge \bar{\varepsilon}^2 \wedge \bar{\varepsilon}^3 = 0.
\end{gather}
Consequently, these two yield
\begin{align}
\bar{s}^* \cs {\left( \bar{\omega} \right)} & = \frac{1}{4 \boldsymbol{\pi}^2} {\left( \det B_r - 1 \right)} \, \bar{\upsilon}^1_2 \wedge \bar{\varepsilon}^1 \wedge \bar{\varepsilon}^2.
\end{align}
Here, \( \det B_r - 1 = K_r \) by the Gauss equation.

\subsection{A proof of the latter of Proposition~\ref{P:conformal-equivalence}} \label{AA:conformal-equivalence}

For each \( r \in (-\varepsilon, \infty) \), let \( w_r = u_r^{-1} \circ u_{\infty} \colon S_{\infty} \to S_r \) be the projection. By the former of Proposition~\ref{P:conformal-equivalence},
\begin{align}
w_r^* \I_r^{\infty} & = u_{\infty}^* {\left( u_r^{-1} \right)}^* \boldsymbol{e}^{2r} {\left( u_r^{\infty} \right)}^* \I^{\infty} = \boldsymbol{e}^{2r} u_{\infty}^* \I^{\infty}
\end{align}
for all \( r \in (-\varepsilon, \infty) \). That is, all the metrics \( w_r^* \I_r^{\infty} \) on \( S_{\infty} \) are conformally equivalent to each other. Therefore, it suffices to show the statement for a single \( r \in (-\varepsilon, \infty) \).

By the Killing-Hopf theorem, we may regard \( M \) as the quotient \( \mathbb{H}^3 / \Gamma \) of the hyperbolic space \( \mathbb{H}^3 \) of~\eqref{E:hyperbolic-space} by some Kleinian group \( \Gamma \subseteq \mathrm{PSL}_2(\mathbb{C}) \). Fix \( x \in S \) and consider locally near \( x \). Fix a representative of \( x \) in \( \mathbb{H}^3 \). What we have to show is summarized as follows.

\underline{Claim.}\, \textit{Fix \( x \in \mathbb{H}^3 \). Let \( S \) be any oriented smooth surface embedded in\/ \( \mathbb{H}^3 \) near \( x \) such that the normal exponential map \( w^{-1} \colon S \to S_{\infty} \subseteq \partial_{\infty} \mathbb{H}^3 \) is a diffeomorphism. Then \( {\left. w^* \I^{\infty} \right|}_y \) belongs to the conformal class of \( \partial_{\infty} \mathbb{H}^3 \) at \( y = w^{-1}(x) \).}

Via an isometry in \( \mathrm{PSL}_2(\mathbb{C}) \) sending \( y \) to \( \infty \), the vector \( {\left. \frac{\partial}{\partial r} \right|}_x \) is moved to a vector pointing exactly upward. Via translation, dilation, and the inversion, it is moved to the vector \( - {\left. \frac{\partial}{\partial t} \right|}_{\hat{\jmath}} \) at \( \hat{\jmath} = (0,0,1) \). Therefore, without loss of generality we may assume in advance that \( {\left. \frac{\partial}{\partial r} \right|}_x = - {\left. \frac{\partial}{\partial t} \right|}_{\hat{\jmath}} \). Then \( S \) is described by a graph \( t = f {\left( x^1, x^2 \right)} \) near \( x = \hat{\jmath} \), and
\begin{align}
\frac{\partial}{\partial r} & = \lambda {\left( f_1 \frac{\partial}{\partial x^1} + f_2 \frac{\partial}{\partial x^2} - \frac{\partial}{\partial t} \right)} \quad \text{near} \ x = \hat{\jmath}
\end{align}
with some positive \( \lambda \) and the partial derivatives satisfying \( f_1 {\left( \hat{\jmath} \right)} = f_2 {\left( \hat{\jmath} \right)} = 0 \).

The normal exponential map \( w^{-1} \colon S \to S_{\infty} \) is given by the point that the geodesic ray starting at \( \frac{\partial}{\partial r} \) heads toward. One can use elementary geometric techniques to readily find that it is given in the local coordinates \( x^1 \) and \( x^2 \) by
\begin{align}
w^{-1} {\left( x^1, x^2 \right)} & = {\left( x^1, x^2 \right)} + \frac{f}{1 + \sqrt{f_1^2 + f_2^2 + 1}} {\left( f_1, f_2 \right)}.
\end{align}
Its differential at \( \hat{\jmath} \) is
\begin{align}
{\left. \mathrm{d} {\left( w^{-1} \right)} \right|}_{\hat{\jmath}} & = {\left. \begin{pmatrix*}[c]
1 + \frac{1}{2} f_{11} & \frac{1}{2} f_{12} \\
\frac{1}{2} f_{21} & 1 + \frac{1}{2} f_{22}
\end{pmatrix*} \right|}_{\hat{\jmath}} = I + \frac{1}{2} H_f,
\end{align}
where \( H_f \) is the Hessian matrix of \( f \) at \( \hat{\jmath} \).

Meanwhile, since \( {\left. \I \right|}_{\hat{\jmath}} = (\mathrm{d} x^1)^2 + (\mathrm{d} x^2)^2 \), for each \( i, j \in \{ 1, 2 \} \) we have
\begin{align*}
{\left. B^i_j \right|}_{\hat{\jmath}} & = \I_{\mathbb{H}^3} {\left( {\left. {\left( \frac{\partial}{\partial x^i} + f_i \frac{\partial}{\partial t} \right)} \right|}_{\hat{\jmath}}, - \bar{\del}^{\mathbb{H}^3}_{\! {\left. {\left( \frac{\partial}{\partial x^j} + f_j \frac{\partial}{\partial t} \right)} \right|}_{\hat{\jmath}}} \frac{\partial}{\partial r} \right)} = - \I_{\mathbb{H}^3} {\left( {\left. \frac{\partial}{\partial x^i} \right|}_{\hat{\jmath}}, \bar{\del}^{\mathbb{H}^3}_{\! {\left. \frac{\partial}{\partial x^j} \right|}_{\hat{\jmath}}} \frac{\partial}{\partial r} \right)} \\
& = - {\left. \frac{\partial}{\partial x^j} \right|}_{\hat{\jmath}} \I_{\mathbb{H}^3} {\left( \frac{\partial}{\partial x^i}, \frac{\partial}{\partial r} \right)} + \I_{\mathbb{H}^3} {\left( \bar{\del}^{\mathbb{H}^3}_{\! {\left. \frac{\partial}{\partial x^j} \right|}_{\hat{\jmath}}} \frac{\partial}{\partial x^i}, {\left. \frac{\partial}{\partial r} \right|}_{\hat{\jmath}} \right)} \\
& = - {\left. \frac{\partial}{\partial x^j} \right|}_{\hat{\jmath}} \frac{\lambda f_i}{f^2} - \I_{\mathbb{H}^3} {\left( \bar{\del}^{\mathbb{H}^3}_{\! {\left. \frac{\partial}{\partial x^j} \right|}_{\hat{\jmath}}} \frac{\partial}{\partial x^i}, {\left. \frac{\partial}{\partial t} \right|}_{\hat{\jmath}} \right)} = - f_{ij} {\left( \hat{\jmath} \right)} - \delta_{ij}.
\end{align*}
That is, \( {\left. B \right|}_{\hat{\jmath}} = - I - H_f \). Consequently, \( {\left. (I-B) \right|}_{\hat{\jmath}} \, {\left. \mathrm{d} w \right|}_y = 2I \), so that
\begin{align}
{\left. w^* \I^{\infty} \right|}_y & = \frac{1}{2} {\left. \I {\left( (I-B) w_* {}\cdot{}, (I-B) w_* {}\cdot{} \right)} \right|}_y = 2 {\left. {\left( {\left( \mathrm{d} x^1 \right)}^2 + {\left( \mathrm{d} x^2 \right)}^2 \right)} \right|}_y,
\end{align}
which belongs to the conformal class of \( \partial_{\infty} \mathbb{H}^3 \) at \( y \).

\subsubsection*{Acknowledgements}

The author expresses gratitude to the advisor Jinsung Park for sharing profound insights and for invaluable discussions and comments on various aspects of this research. The author is grateful to the coadvisor Hyungryul Baik for helpful feedback and discussions. The author also thanks Inhyeok Choi for fruitful discussions on convex-cocompact hyperbolic geometry.


\begin{bibdiv}
\begin{biblist}

\bib{BaseilhacCS}{article}{
    author={Baseilhac, St{\'e}phane},
    title={Chern Simons theory in dimension three},
    status={unpublished lecture note},
    date={2009},
    eprint={https://imag.umontpellier.fr/~baseilhac/CS.pdf},
}

\bib{MR3934591}{article}{
    author={Bridgeman, Martin},
    author={Brock, Jeffrey},
    author={Bromberg, Kenneth},
    title={Schwarzian derivatives, projective structures, and the Weil-Petersson gradient flow for renormalized volume},
    journal={Duke Math. J.},
    volume={168},
    date={2019},
    number={5},
    pages={867--896},
    issn={0012-7094},
    review={\MR{3934591}},
    doi={10.1215/00127094-2018-0061},
}

\bib{MR4668096}{article}{
    author={Bridgeman, Martin},
    author={Brock, Jeffrey},
    author={Bromberg, Kenneth},
    title={The Weil-Petersson gradient flow of renormalized volume and 3-dimensional convex cores},
    journal={Geom. Topol.},
    volume={27},
    date={2023},
    number={8},
    pages={3183--3228},
    issn={1465-3060},
    review={\MR{4668096}},
    doi={10.2140/gt.2023.27.3183},
}

\bib{MR4571574}{article}{
    author={Bridgeman, Martin},
    author={Bromberg, Kenneth},
    author={Vargas Pallete, Franco},
    title={Convergence of the gradient flow of renormalized volume to convex cores with totally geodesic boundary},
    journal={Compos. Math.},
    volume={159},
    date={2023},
    number={4},
    pages={830--859},
    issn={0010-437X},
    review={\MR{4571574}},
    doi={10.1112/s0010437x2300708x},
}

\bib{MR3732685}{article}{
    author={Bridgeman, Martin},
    author={Canary, Richard D.},
    title={Renormalized volume and the volume of the convex core},
    language={English, with English and French summaries},
    journal={Ann. Inst. Fourier (Grenoble)},
    volume={67},
    date={2017},
    number={5},
    pages={2083--2098},
    issn={0373-0956},
    review={\MR{3732685}},
    doi={10.5802/aif.3130},
}

\bib{MR0615912}{book}{
    author={Bishop, Richard L.},
    author={Goldberg, Samuel I.},
    title={Tensor analysis on manifolds},
    note={Corrected reprint of the 1968 original},
    publisher={Dover Publications, Inc., New York},
    date={1980},
    pages={viii+280 pp. (loose errata)},
    isbn={0-486-64039-6},
    review={\MR{0615912}},
}

\bib{MR1224675}{book}{
    author={Bredon, Glen E.},
    title={Topology and geometry},
    series={Graduate Texts in Mathematics},
    volume={139},
    publisher={Springer-Verlag, New York},
    date={1993},
    pages={xiv+557},
    isbn={0-387-97926-3},
    review={\MR{1224675}},
    doi={10.1007/978-1-4757-6848-0},
}

\bib{MR1509253}{article}{
    author={Cartan, {\'E}lie},
    title={Sur les vari{\'e}t{\'e}s {\`a} connexion affine et la th{\'e}orie de la relativit{\'e} g{\'e}n{\'e}ralis{\'e}e (premi{\`e}re partie)},
    language={French},
    journal={Ann. Sci. \'Ecole Norm. Sup. (3)},
    volume={40},
    date={1923},
    pages={325--412},
    issn={0012-9593},
    review={\MR{1509253}},
}

\bib{MR1509255}{article}{
    author={Cartan, {\'E}lie},
    title={Sur les vari{\'e}t{\'e}s {\`a} connexion affine, et la th{\'e}orie de la relativit{\'e} g{\'e}n{\'e}ralis{\'e}e (premi{\`e}re partie) (Suite)},
    language={French},
    journal={Ann. Sci. \'Ecole Norm. Sup. (3)},
    volume={41},
    date={1924},
    pages={1--25},
    issn={0012-9593},
    review={\MR{1509255}},
}

\bib{MR1509263}{article}{
    author={Cartan, {\'E}lie},
    title={Sur les vari{\'e}t{\'e}s {\`a} connexion affine, et la th{\'e}orie de la relativit{\'e} g{\'e}n{\'e}ralis{\'e}e (deuxi{\`e}me partie)},
    language={French},
    journal={Ann. Sci. \'Ecole Norm. Sup. (3)},
    volume={42},
    date={1925},
    pages={17--88},
    issn={0012-9593},
    review={\MR{1509263}},
}

\bib{MR0353327}{article}{
    author={Chern, Shiing-Shen},
    author={Simons, James},
    title={Characteristic forms and geometric invariants},
    journal={Ann. of Math. (2)},
    volume={99},
    date={1974},
    pages={48--69},
    issn={0003-486X},
    review={\MR{0353327}},
    doi={10.2307/1971013},
}

\bib{EinsteinTP}{article}{
    author={Einstein, Albert},
    title={Riemann-Geometrie mit Aufrechterhaltung des Begriffes des Fernparallelismus},
    journal={Preussische Akademie der Wissenschaften, Phys.-math. Klasse, Sitzungsberichte},
    date={1928},
    pages={217-221},
}

\bib{EpsteinWS}{article}{
    author={Epstein, Charles L.},
    title={Envelopes of Horospheres and Weingarten Surfaces in Hyperbolic 3-Space},
    status={preprint},
    date={1984},
    eprint={https://www2.math.upenn.edu/~cle/papers/WeingartenSurfaces.pdf},
}

\bib{MR1337109}{article}{
    author={Freed, Daniel S.},
    title={Classical Chern-Simons Theory. I},
    journal={Adv. Math.},
    volume={113},
    date={1995},
    number={2},
    pages={237--303},
    issn={0001-8708},
    review={\MR{1337109}},
    doi={10.1006/aima.1995.1039},
}

\bib{MR3159164}{article}{
    author={Guillarmou, Colin},
    author={Moroianu, Sergiu},
    title={Chern-Simons line bundle on Teichm\"{u}ller space},
    journal={Geom. Topol.},
    volume={18},
    date={2014},
    number={1},
    pages={327--377},
    issn={1465-3060},
    review={\MR{3159164}},
    doi={10.2140/gt.2014.18.327},
}

\bib{MR1758076}{article}{
    author={Graham, C. Robin},
    title={Volume and area renormalizations for conformally compact Einstein metrics},
    booktitle={The Proceedings of the 19th Winter School ``Geometry and Physics'' (Srn{\'\i}, 1999)},
    journal={Rend. Circ. Mat. Palermo (2) Suppl.},
    date={2000},
    number={63},
    pages={31--42},
    issn={1592-9531},
    review={\MR{1758076}},
}

\bib{MR1644988}{article}{
    author={Henningson, M.},
    author={Skenderis, K.},
    title={The holographic Weyl anomaly},
    journal={J. High Energy Phys.},
    date={1998},
    number={7},
    pages={Paper 23, 12},
    issn={1126-6708},
    review={\MR{1644988}},
    doi={10.1088/1126-6708/1998/07/023},
}

\bib{MR1218931}{article}{
    author={Kirk, Paul},
    author={Klassen, Eric},
    title={Chern-Simons Invariants of 3-Manifolds Decomposed along Tori and the Circle Bundle Over the Representation Space of {$T^2$}},
    journal={Comm. Math. Phys.},
    volume={153},
    date={1993},
    number={3},
    pages={521--557},
    issn={0010-3616},
    review={\MR{1218931}},
}

\bib{MR3784525}{article}{
    author={Kojima, Sadayoshi},
    author={McShane, Greg},
    title={Normalized entropy versus volume for pseudo-Anosovs},
    journal={Geom. Topol.},
    volume={22},
    date={2018},
    number={4},
    pages={2403--2426},
    issn={1465-3060},
    review={\MR{3784525}},
    doi={10.2140/gt.2018.22.2403},
}

\bib{MR2328927}{article}{
    author={Krasnov, Kirill},
    author={Schlenker, Jean-Marc},
    title={Minimal surfaces and particles in 3-manifolds},
    journal={Geom. Dedicata},
    volume={126},
    date={2007},
    pages={187--254},
    issn={0046-5755},
    review={\MR{2328927}},
    doi={10.1007/s10711-007-9132-1},
}

\bib{MR2386723}{article}{
    author={Krasnov, Kirill},
    author={Schlenker, Jean-Marc},
    title={On the Renormalized Volume of Hyperbolic 3-Manifolds},
    journal={Comm. Math. Phys.},
    volume={279},
    date={2008},
    number={3},
    pages={637--668},
    issn={0010-3616},
    review={\MR{2386723}},
    doi={10.1007/s00220-008-0423-7},
}

\bib{MR3887684}{book}{
    author={Lee, John M.},
    title={Introduction to Riemannian manifolds},
    series={Graduate Texts in Mathematics},
    volume={176},
    edition={2},
    publisher={Springer, Cham},
    date={2018},
    pages={xiii+437},
    isbn={978-3-319-91754-2},
    isbn={978-3-319-91755-9},
    review={\MR{3887684}},
}

\bib{LeeCM}{article}{
    author={Lee, Dongha},
    title={Complex-valued extension of mean curvature for surfaces in Riemann-Cartan geometry},
    status={arXiv preprint},
    date={2025},
    eprint={https://arxiv.org/abs/2502.12647},
}

\bib{MR3586015}{book}{
    author={Marden, Albert},
    title={Hyperbolic manifolds},
    note={An introduction in 2 and 3 dimensions},
    publisher={Cambridge University Press, Cambridge},
    date={2016},
    pages={xviii+515},
    isbn={978-1-107-11674-0},
    review={\MR{3586015}},
    doi={10.1017/CBO9781316337776},
}

\bib{MR3228423}{article}{
    author={McIntyre, Andrew},
    author={Park, Jinsung},
    title={Tau function and Chern-Simons invariant},
    journal={Adv. Math.},
    volume={262},
    date={2014},
    pages={1--58},
    issn={0001-8708},
    review={\MR{3228423}},
    doi={10.1016/j.aim.2014.05.005},
}

\bib{MR0815482}{article}{
    author={Neumann, Walter D.},
    author={Zagier, Don},
    title={Volumes of hyperbolic three-manifolds},
    journal={Topology},
    volume={24},
    date={1985},
    number={3},
    pages={307--332},
    issn={0040-9383},
    review={\MR{0815482}},
    doi={10.1016/0040-9383(85)90004-7},
}

\bib{MR1813434}{article}{
    author={Patterson, S. J.},
    author={Perry, Peter A.},
    title={The divisor of Selberg's zeta function for Kleinian groups},
    note={Appendix~A by Charles Epstein},
    journal={Duke Math. J.},
    volume={106},
    date={2001},
    number={2},
    pages={321--390},
    issn={0012-7094},
    review={\MR{1813434}},
    doi={10.1215/S0012-7094-01-10624-8},
}

\bib{MR3188032}{article}{
    author={Schlenker, Jean-Marc},
    title={The renormalized volume and the volume of the convex core of quasifuchsian manifolds},
    journal={Math. Res. Lett.},
    volume={20},
    date={2013},
    number={4},
    pages={773--786},
    issn={1073-2780},
    review={\MR{3188032}},
    doi={10.4310/MRL.2013.v20.n4.a12},
}

\bib{MR0648524}{article}{
    author={Thurston, William P.},
    title={Three-dimensional manifolds, Kleinian groups and hyperbolic geometry},
    journal={Bull. Amer. Math. Soc. (N.S.)},
    volume={6},
    date={1982},
    number={3},
    pages={357--381},
    issn={0273-0979},
    review={\MR{0648524}},
    doi={10.1090/S0273-0979-1982-15003-0},
}

\bib{MR1997440}{article}{
    author={Takhtajan, Leon A.},
    author={Teo, Lee-Peng},
    title={Liouville Action and Weil-Petersson Metric on Deformation Spaces, Global Kleinian Reciprocity and Holography},
    journal={Comm. Math. Phys.},
    volume={239},
    date={2003},
    number={1-2},
    pages={183--240},
    issn={0010-3616},
    review={\MR{1997440}},
    doi={10.1007/s00220-003-0878-5},
}

\bib{MR4028105}{article}{
    author={Vargas Pallete, Franco},
    title={Local convexity of renormalized volume for rank-1 cusped manifolds},
    journal={Math. Res. Lett.},
    volume={26},
    date={2019},
    number={3},
    pages={903--919},
    issn={1073-2780},
    review={\MR{4028105}},
    doi={10.4310/MRL.2019.v26.n3.a10},
}

\bib{MR4887772}{article}{
    author={Vargas Pallete, Franco},
    title={Upper bounds on renormalized volume for Schottky groups},
    journal={Math. Ann.},
    volume={392},
    date={2025},
    number={1},
    pages={733--750},
    issn={0025-5831},
    review={\MR{4887772}},
    doi={10.1007/s00208-024-03086-2},
}

\bib{MR807069}{article}{
    author={Yoshida, Tomoyoshi},
    title={The {$\eta$}-invariant of hyperbolic 3-manifolds},
    journal={Invent. Math.},
    volume={81},
    date={1985},
    number={3},
    pages={473--514},
    issn={0020-9910},
    review={\MR{807069}},
    doi={10.1007/BF01388583},
}

\bib{MR0889594}{article}{
    author={Zograf, P. G.},
    author={Takhtadzhyan, L. A.},
    title={On the uniformization of Riemann surfaces and on the Weil-Petersson metric on the Teichm{\"u}ller and Schottky spaces},
    language={Russian},
    journal={Mat. Sb. (N.S.)},
    volume={132(174)},
    date={1987},
    number={3},
    pages={304--321, 444},
    issn={0368-8666},
    translation={
        journal={Math. USSR-Sb.},
        volume={60},
        date={1988},
        number={2},
        pages={297--313},
        issn={0025-5734},
    },
    review={\MR{0889594}},
    doi={10.1070/SM1988v060n02ABEH003170},
}

\end{biblist}
\end{bibdiv}
\end{document}